\documentclass{amsart}

\usepackage[latin2]{inputenc}
\usepackage[ngerman,english]{babel}

\usepackage{amsthm}
\usepackage{amssymb}
\usepackage{amsmath}

\usepackage{mathrsfs}

\usepackage{tikz}
\usetikzlibrary{decorations.pathreplacing,decorations.pathmorphing,calc}

\usepackage{graphicx}

\usepackage[final]{pdfpages}

\usepackage{enumerate}

\usepackage{mathbbol}
\usepackage[nointegrals]{wasysym}

\usepackage{dlfltxbcodetips}

\usepackage[linktocpage]{hyperref}

\newtheorem{thm}{Theorem}[section]
\newtheorem*{thm*}{Theorem}
\newtheorem{lem}[thm]{Lemma}
\newtheorem{kor}[thm]{Corollary}
\theoremstyle{definition}
\newtheorem{dfn}[thm]{Definition}
\newtheorem{rem}[thm]{Remark}

\renewcommand{\varangle}{\sphericalangle}

\renewcommand{\theequation}{\arabic{section}.\arabic{equation}}

\newcommand{\claim}[1]{\noindent Claim #1: }
\newcommand{\proofofclaim}[1]{\noindent \textit{Proof of Claim #1.} }
\newcommand{\proofofclaimend}[1]{\noindent \textit{End of Proof of Claim #1.}\\ }

\newcommand{\diam}{\operatorname{diam}}
\newcommand{\supp}{\operatorname{supp}}
\newcommand{\sspan}{\operatorname{span}}
\newcommand{\R}{\mathbb{R}}
\newcommand{\cH}{\mathcal{H}}
\newcommand{\K}{\mathcal{K}}
\newcommand{\M}{\mathcal{M}}
\newcommand{\dd}{\mathrm{d}}
\newcommand{\p}{p}
\newcommand{\N}{n}
\newcommand{\NN}{m}
\newcommand{\n}{{N}}

\newcommand{\height}[1]{\mathfrak{h}_{#1}}
\newcommand{\face}[1]{\mathfrak{fc}_{#1} }
\newcommand{\volume}{\mathfrak{v}} 
\newcommand{\vol}{\omega_{\N}}
\newcommand{\vo}[1]{\omega_{#1}}
\newcommand{\aff}{\textnormal{aff}}
\newcommand{\Lip}{\operatorname{Lip}}
\newcommand{\Avg}[1]{
\ifthenelse{\equal{#1}{}}
{
\operatorname{Avg}
}
{
\underset{#1}{\operatorname{Avg}}
}
}

\usepackage[left=3.3cm,right=3.3cm,top=3cm,bottom=2.5cm]{geometry}
\begin{document}
\title[Integral Menger Curvature and Rectifiability]{Integral Menger Curvature and Rectifiability of $n$-dimensional Borel sets in Euclidean $N$-space}
\author{Martin Meurer}
\date{\today}
\address{RWTH Aachen University, Institut f\"ur Mathematik, Templergraben 55, D-52062 Aachen, Germany}
\email{meurer@instmath.rwth-aachen.de}
\thanks{The author would like to thank Heiko von der Mosel for all his helpful advice. Furthermore he thanks 
Armin Schikorra and Sebastian Scholtes for many fruitful discussions.}

\begin{abstract}
In this work we show that an $n$-dimensional Borel set in Euclidean $N$-space with finite integral Menger curvature
is $n$-rectifiable, meaning that it can be covered by countably many images of Lipschitz
continuous functions up to a null set in the sense of Hausdorff measure.
This generalises L\'{e}ger's \cite{Leger} rectifiability result for one-dimensional sets to arbitrary dimension
and co-dimension.
In addition, we characterise possible integrands and discuss examples known from the literature.

Intermediate results of independent interest include upper bounds of different versions of P. Jones's $\beta$-numbers
in terms of integral Menger curvature without assuming lower Ahlfors regularity, in contrast to the 
results of Lerman and Whitehouse \cite{MR2558685}.
\end{abstract}

\maketitle

\section{Introduction}
For three points $x,y,z \in \R^{\n}$, we denote by $c(x,y,z)$ the inverse of 
the radius of the circumcircle determined by these three points.
This expression is called \textit{Menger curvature} of $x,y,z$. 
For a Borel set $E \subset \R^{\n}$, we define by
\[\M_{2}(E):=\int_E \int_E \int_E c^2(x,y,z) \ \dd \cH^{1}(x) \dd \cH^{1}(y) \dd \cH^{1} (z)\]
the \textit{total Menger curvature} of $E$, where $\cH^{1}$ denotes the one-dimensional Hausdorff measure.
In 1999, J.C. L\'{e}ger proved the following theorem. 

\begin{thm*}[\cite{Leger}]
	If $E\subset \R^{\n}$ is some Borel set with $0< \cH^{1}(E)< \infty$ and $\M_{2}(E)< \infty$, then
	$E$ is $1$-rectifiable, i.e., there exists a countable 
	family of Lipschitz functions $f_i:\R \to \R^{\n}$ such that $\cH^{1}(E \setminus \bigcup_i f_{i}(\R)) =0$.
\end{thm*}

This result is an important step in the proof of Vitushkin's conjecture (for more details see 
\cite{MR3154530,MR2676222}), which states that a compact 
set with finite one-dimensional Hausdorff measure is removable for bounded analytic functions 
if and only if it is purely $1$-unrectifiable, which means that every $1$-rectifiable subset of this set has
Hausdorff measure zero. A higher dimensional analogue of Vitushkin's conjecture is proven in \cite{MR3264510} but
without using a higher dimensional version of L\'{e}ger's theorem since in the higher dimensional setting there
seems to be no connection between the $\N$-dimensional Riesz transform and curvature (cf.~introduction 
of \cite{MR3264510}).

There exist several generalisations of L{\'e}ger's result. Hahlomaa proved in \cite{MR2456269, MR2297880, MR2163108} 
that if $X$ is a metric space and $\M_{2}(X)< \infty$, then $X$ is 1-rectifiable.
Another version of this theorem dealing with sets of fractional Hausdorff dimension equal or less than $\frac{1}{2}$
is given by Lin and Mattila in \cite{MR1814107}. 

In the present work, we generalise the result of L{\'e}ger to arbitrary dimension and co-dimension, i.e.,
for $\N$-dimensional subsets of $\R^{\n}$ 
where $\N \in \mathbb{N}$ satisfies $\N < \n$. In the case $\N=\n$ every $E\subset \R^{\n}$ is $\N$-rectifiable. 
On the one hand, it is quite clear which conclusion we want to obtain, namely that the set $E$ is $\N$-rectifiable,
which means that there exists a countable family of Lipschitz functions $f_i:\R^{\N} \to \R^{\n}$ such that 
$\cH^{\N}(E \setminus \bigcup_i f_{i}(\R^{\N})) =0$. 
On the other hand, it is by no means clear how to define integral Menger curvature for $\N$-dimensional sets.
L\'{e}ger himself suggested an expression  
which turns out to be improper for our proof\footnote{Hence, we agree with a remark made by 
Lerman and Whitehouse at the end of the introduction of \cite{MR2558685}.} (cf.~section \ref{22.10.2014.2}).
We characterise possible integrands for our result in Definition \ref{4.10.12.1},
but for now let us start with an explicit example:
\[\K(x_{0},\dots,x_{\N+1})= 
	\frac{\cH^{\N+1}(\Delta(x_{0},\dots,x_{\N+1}))}{\Pi_{0 \le i < j \le \N+1}d(x_{i},x_{j})},\]
where the numerator denotes the $(\N+1)$-dimensional volume of the simplex ($\Delta(x_{0},\dots,x_{\N+1})$)
 spanned by the vertices
$x_{0}, \dots,x_{\N+1}$, and
$d(x_{i},x_{j})$ is the distance between $x_i$ and $x_j$.
Using the law of sines, we obtain for $\N=1$
\[\K(x_{0},x_{1},x_{2})= \frac{\cH^{2}(\Delta(x_{0},x_{1},x_{2}))}{d(x_{0},x_{1})d(x_{0},x_{2})d(x_{1},x_{2})}
	=\frac{1}{4}c(x_{0},x_{1},x_{2}).\]
Hence, $\K$ can be regarded as a generalisation of the original Menger curvature for higher dimensions.
We set 
\begin{align}\label{27.08.2015.1}
	\M_{\K^{2}}(E)&:=\int_E \dots \int_E \K^{2}(x_{0},\dots,x_{\N+1}) \ \dd \cH^{\N}(x_{0}) \dots \dd \cH^{\N} (x_{\N+1}).
\end{align}
Now we can state our main theorem for this specific integrand (see Theorem \ref{maintheorem} for the general version).
\begin{thm}\label{13.11.2014.1}
	If $E\subset \R^{\n}$ is some Borel set with  $\M_{\K^{2}}(E)< \infty$, then
	$E$ is $\N$-rectifiable.
\end{thm}

Let us briefly overview a couple of results for 
the higher dimensional case. There exist well-known equivalent characterisations of $\N$-rectifiability,
for example, in terms of approximating tangent planes \cite[Thm. 15.19]{MR1333890}, 
orthogonal projections \cite[Thm. 18.1, Besicovitch-Federer projection theorem]{MR1333890},
and in terms of densities \cite[Thm. 17.6 and Thm. 17.8 (Preiss's theorem)]{MR1333890}.
Recently Tolsa and Azzam proved in \cite{2015arXiv150101569T} and \cite{TolsaAzzam} a characterisation of $\N$-rectifiability using the so called
$\beta$-numbers\footnote{Introduced by P. W. Jones in \cite{MR1069238} and \cite{MR1103619}.} 
defined for $k > 1$, 
$x \in \mathbb{R}^{\n}$, $t>0$, $\p\ge 1$ by
\[ \beta_{\p;k;\mu} (x,t) := \inf_{P \in \mathcal{P}(\n,\N)} \left( \frac{1}{t^{\N}} \int_{B(x,kt)} 
	\left( \frac{d(y,P)}{t} \right)^{\p} \dd  \mu (y) \right)^{\frac{1}{\p}}, \]
where $ \mathcal{P}(\n,\N) $ denotes the set of all $\N$-dimensional planes in $\mathbb{R}^{\n}$,
$d(y,P)$ is the distance of $y$ to the $\N$-dimensional plane $P$
and $\mu$ is a Borel measure on $\R^{\n}$.
They showed in particular that an $\cH^{\N}$-measurable set $E \subset \R^{\n}$ with $\cH^{\N}(E)<\infty$ is
$\N$-rectifiable \textit{if and only if}
\begin{align}\label{28.08.2015.1}
	\int_{0}^{1} \beta_{2;1;\cH^{\N}|_{E}}(x,r)^{2}\frac{\dd r}{r} &< \infty \quad \quad \quad \text{ for } \cH^{\N}-a.e. x \in E.
\end{align}
This result is remarkable in relation to our result since the $\beta$-numbers and even 
an expression similar to \eqref{28.08.2015.1}
play an important role in our proof. Nevertheless at the moment, we do not see how Tolsa's result could be used to
shorten our proof of Theorem \ref{13.11.2014.1}.
There are further characterisations of rectifiability by Tolsa and Toro in \cite{1402.2799} and \cite{1408.6979}.

Now we present some of our own intermediate results 
that finally lead to the proof of Theorem \ref{13.11.2014.1}, but that might also be of independent interest itself.
There is a connection between those $\beta$-numbers and integral Menger curvature \eqref{27.08.2015.1}. In section \ref{1.11.2014.1},
we prove the following theorem (see Theorem \ref{lem2.5} for a more general version):

\begin{thm}\label{13.11.2014.2}
	Let $\mu$ be some arbitrary Borel measure on $\R^{\n}$ with compact support such that there is a
	constant $C \ge 1$ with $\mu(B) \le C (\diam B)^{\N}$ for all balls $B \subset \R^{\n}$, 
	where $\diam B$ denotes the diameter of the ball $B$.
	Let $B(x,t)$ be a fixed ball with $\mu(B(x,t)) \ge \lambda t^{\N}$ for some $\lambda > 0$ and let $k >2$.
	Then there exist some constants $k_{1}> 1$ and $C \ge 1$ such that 
	\[ \beta_{2;k}(x,t)^{2} \le \frac{C}{t^{\N}}
		\int_{B(x,k_{1}t)} \dots \int_{B(x,k_{1}t)} \chi_{D}(x_{0},\dots,x_{\N})
		\K^{2}(x_{0},\dots,x_{\N+1}) \ \dd \mu(x_{0}) \dots \dd \mu (x_{\N+1}),\]
	where 
	$D=\{(x_{0},\dots,x_{\N+1})\in B(x,k_{1}t)^{\N+2} | d(x_{i},x_{j})\ge \frac{t}{k_{1}}, i\neq j\}$.
\end{thm}

A measure $\mu$ is said to be $\N$-dimensional Ahlfors regular if and only if there exists
some constant $C\ge 1$ so that $\frac{1}{C} (\diam B)^{\N} \le \mu(B) \le C (\diam B)^{\N}$
for all balls $B$ with centre on the support of $\mu$.
We mention that we do not have to assume for this theorem that the measure $\mu$ is $\N$-dimensional Ahlfors regular.
We only need the upper bound on $\mu(B)$ for each ball $B$ and the condition $\mu(B(x,t)) \ge \lambda t^{\N}$ 
for \textit{one} specific ball $B(x,t)$.

Lerman and Whitehouse obtain a comparable result in \cite[Thm.~1.1]{MR2558685}. The main differences are that,
on the one hand, they have to use an $\N$-dimensional Ahlfors regular measure, but, on the other hand,
they work in a real separable Hilbert space of possibly infinite dimension instead of $\R^{\n}$.
The higher dimensional Menger curvatures they used (see \cite[introduction and section 6]{MR2558685}) 
are examples of integrands that also 
fit in our more general setting\footnote{A characterisation of all possible integrands for our result 
can be found at the beginning 
of section \ref{1.11.2014.2}. In section \ref{22.10.2014.2}, we discuss one of the integrands of Lerman and Whitehouse.}.
This means that all of our results are valid if one uses their integrands instead of the initial $\K$
presented as an example above.

In addition to rectifiability, there is the notion of \textit{uniform rectifiability}, which implies rectifiability. 
A set is uniformly rectifiable if it is
Ahlfors regular\footnote{A set $E$ is $\N$-dimensional Ahlfors regular if and only if the restricted Hausdorff 
measure $\cH^{\N} \textsf{L} E$ is $\N$-dimensional Ahlfors regular.} and if it fulfils a second condition 
in terms of $\beta$-numbers (cf. \cite[Thm.~1.57, (1.59)]{MR1251061}). 
In \cite{MR2558685} and \cite{MR2848529}, Lerman and Whitehouse give an alternative characterisation of 
uniform rectifiability 
by proving that for an Ahlfors regular set this $\beta$-number term is comparable to a term expressed 
with integral Menger curvature. One of the two inequalities needed is given in
in \cite[Thm. 1.3]{MR2558685}, and is similar to our following theorem, which is
a consequence of Theorem \ref{13.11.2014.2} in connection with Fubini's theorem 
(see Theorem \ref{13.11.2014.4} for a more general version).
We emphasise again that in our 
case the measure $\mu$ does not have to be Ahlfors regular.

\begin{thm}\label{13.11.2014.3}
	Let $\mu,\lambda$ and $k$ be as in the previous theorem. There exists a constant $C \ge 1$ such that
	\[\int  \int_{0}^{\infty} \beta_{2;k}(x,t)^{2}\Eins_{\left\{ \mu(B(x,t)) \ge \lambda t^{\N} \right\}} \frac{\dd t}{t} 
		 \dd \mu(x) \le C\mathcal{M}_{\K^{2}}(\mu).\]
\end{thm}

In the last years, there occurred several papers working with integral Menger curvatures. 
Some deal with (one-dimensional) space curves and get higher regularity ($C^{1,\alpha}$) of the arc length 
parametrisation if the integral Menger curvature is finite, e.g \cite{MR2489022,MR2668877}.
Others handle higher dimensional objects in \cite{1205.4112,MR3061777,SvdMsurface} occasionally 
using versions of integral Menger curvatures 
similar to ours\footnote{Our main result does not work with their integrands, but most of the partial results 
are valid, cf.~section \ref{22.10.2014.2}.}. 
Remarkable are the results of Blatt and Kolasinski \cite{MR2921162,MR3021472}. They proved among other things that
for $p > \N(\N+1)$ and some compact $\N$-dimensional $C^{1}$ manifold $\Sigma$ 
\[\int_{\Sigma} \dots \int_{\Sigma}
	\left(\frac{\cH^{\N+1}(\Delta(x_{0},\dots,x_{\N+1}))}{\diam(\Delta(x_{0},\dots,x_{\N+1}))^{\N+2}}\right)^{p} 
	\dd \cH^{\N}(x_{0}) \dots, \dd \cH^{\N}(x_{\N+1}) < \infty\]
is \textit{equivalent} to having a local representation of $\sigma$ as the graph of a function belonging to the Sobolev Slobodeckij space
$W^{2-\frac{\N(\N+1)}{\p},p}$.
Finally, we mention that in \cite{MR3091327,MR3105400} Menger curvature energies are recently used as knot energies 
in geometric knot theory to avoid some of the drawbacks of self-repulsive potentials like the M\"obius energy
\cite{MR1098918,MR1259363}.

\textbf{Organisation of this work.}
In section 3, we give the precise formulation of our main result and
discuss some examples of integrands known from several papers working with integral Menger 
curvatures.
In section \ref{beta}, we present some results for a Borel measure including the general versions of
Theorems \ref{13.11.2014.2} and \ref{13.11.2014.3}, namely Theorem \ref{lem2.5} and \ref{13.11.2014.4}.
The following sections \ref{27.10.2014.2} to \ref{notanullset} give the proof of our main result.
We remark that all statements in section \ref{construction}, \ref{gamma} and \ref{notanullset},
except section \ref{27.10.2014.3}, depend on the construction given in chapter \ref{construction}.

\section{Preliminaries}\label{27.10.2014.1}
\subsection{Basic notation and linear algebra facts}
Let $\N,\NN,\n \in \mathbb{N}$ with $1\le \N < \n$ and $1 \le \NN < \n$.
If $E \subset \R^{\n}$ is some subset of $\R^{\n}$, we write $\overline{E}$ for its closure and $\mathring{E}$
for its interior.\index{$\overline{E}$}\index{$\mathring{E}$}
We set $d(x,y):=|x-y|$ \index{$d(x,y)$} where $x,y \in \R^{\n}$ and $|\cdot|$ is the usual Euclidean norm.
Furthermore, for $x \in \R^{\n}$ and $E_{1},E_{2} \subset \R^{\n}$, we set $d(x,E_{2})=\inf_{y \in E_{2}}d(x,y)$,
$d(E_{1},E_{2})=\inf_{z \in E_{1}}d(z,E_{2})$
and $\# E$ means the number of elements of $E$.\index{$\#E$}
By $B(x,r)$ we denote the closed ball in $\R^{\n}$ with centre $x$ and radius $r$,
and we define by $\vol$ \index{$\vol$} the $\N$-dimensional volume of the $\N$-dimensional unit ball.
Let $G(\n,\NN)$ be the Grassmannian, the space of all $\NN $-dimensional linear subspaces of $\R^\n$ and
$\mathcal{P}(\n,\NN )$ the set of all $\NN $-dimensional affine subspaces of $\R^\n$.
For $P \in \mathcal{P}(\n,\NN )$, we define $\pi_P$ as the orthogonal projection on $P$.
If $P \in \mathcal{P}(\n,\NN )$, we have that $P - \pi_{P}(0) \in G(\n,\NN )$, hence $P - \pi_{P}(0)$
is the linear subspace parallel to $P$.
Furthermore, we set $\pi_P^{\perp}:=\pi_{P-\pi_{P}(0)}^{\perp} := \pi_{(P-\pi_{P}(0))^{\perp}}$
where $\pi_{(P-\pi_{P}(0))^{\perp}}$ is the orthogonal projection on the orthogonal complement of $P-\pi_{P}(0)$.
This implies that $\pi_P^{\perp}=\pi_{\tilde P}^{\perp}$ and $\pi_P \neq \pi_{\tilde P}$
whenever $P$ is parallel but not equal to $\tilde P$.

Furthermore, for $A \subset \R^{\n}$ and $x \in \R^{\n}$, we set $A+x:=\{y \in \R^{\N}|y-x \in A\}$. \index{$A+x$}
By $\sspan(A)$, we denote the linear subspace of $\R^{\n}$ spanned by the elements of $A$. If 
$A=\{o_{1},\dots,o_{\NN}\}$ or $A=A_{1} \cup A_{2}$, we may 
write $\sspan(o_{1},\dots,o_{\NN})$ resp. $\sspan(A_{1},A_{2})$ instead of $\sspan(A)$.

\begin{rem}\label{24.04.2013.1}
	Let $P \in \mathcal{P}(\n,\NN)$ and $a,x \in \R^{\n}$.
	We have $ \pi_{P}(a) = \pi_{P-x}(a-x)+x$.
\end{rem}

\begin{rem}\label{2.11.11.1}\label{12.4.11}
	Let $b,a,a_i \in \R^\n$, $\alpha_i \in \R$ for $i=1,..l$, $l \in \mathbb{N}$ with 
	$b = a + \sum_{i =1}^l \alpha_i(a_i-a)$
	and $P \in \mathcal{P}(\n,m)$.
	Then we have
	$\pi_{P}(b)= \pi_{P}\left(a\right) +\sum_{i =1}^l \alpha_i \big[\pi_{P} (a_i) - \pi_{P}(a)\big]$ and
	$ d(b,P) \le d(a,P) + \sum_{i =1}^l |\alpha_i|\left(d(a_i,P) + d(a,P) \right)$.
\end{rem}

\begin{figure}[h]
	\begin{center}
	\begin{tikzpicture}[scale=0.49]
		\draw[-] (21,4) -- (2,4) node [below] {$P_{2}$} ;
		\draw[-] (21 ,4)  -- (27,8);
		\draw[-] (8,8)  -- (17.2,8);
		\draw[dotted] (17.2,8)  -- (20.5,8);
		\draw[-] (20.5,8)  -- (27,8);
		\draw[-] (2,4)  -- (8,8);

		\draw[-] (11.2,4)  -- (17.2,8);

		\draw[-] (19,7) -- (6,2) node [below] {$P_{1}$};
		\draw[-] (19,7)  -- (25,11);
		\draw[dotted] (12,6)  -- (17.2,8);
		\draw[-] (17.2,8)  -- (25,11);
		\draw[-] (6,2)  -- (9,4);
		\draw[dotted] (9,4)  -- (12,6);
		
		\draw[dashed] (12.7,5)  -- (17.5,5) node [right] {$\pi_{P_{2}}(a_{1})$};
		\draw (17.5,5) circle (0.03);
		\draw[dashed] (17.5,6.85) -- (12.7,5);
		\draw (12.7,5) circle (0.04) ;
		\draw[-] (12.6,5.04)  -- (10.6,5.5) node [left] {$\pi_{P_1\cap P_2}(a_{1})$};
		\draw[dashed] (17.5,5)  -- (17.5,6.85);
		\draw (17.5,6.85) circle (0.03);
		\draw[-] (17.6,6.8) -- (18.5,6.4) node [right] {$a_{1}$};

		\draw[dashed] (15.7,7)  -- (22,7) node [right] {$\pi_{P_{2}}(a_{2})$};
		\draw (22,7) circle (0.03);
		\draw[dashed]  (22,9.43) -- (15.7,7);
		\draw (15.7,7) circle (0.04) ;
		\draw[-] (15.6,7.04)  -- (13.6,7.5) node [left] {$\pi_{P_1\cap P_2}(a_{1})$};
		\draw[dashed] (22,9.43)  -- (22,7);
		\draw (22,9.43) circle (0.03) ;
		\draw[-] (21.88,9.5) -- (21,10) node [left] {$a_{2}$};
		
		\draw[-] (11.8,4.4)  -- (13,3) node [right] {$P_{1}\cap P_{2}$};
	\end{tikzpicture}
	\end{center}
	\vspace{-6mm}
	\caption[...tba]{Illustration of Lemma \ref{20.2.2012.4}: $\frac{|a_1- \pi_{P_2}(a_1)|}{|a_1- \pi_{P_1 \cap P_2}(a_1)|}= \frac{|a_2- \pi_{P_2}(a_2)|}{|a_2- \pi_{P_1 \cap P_2}(a_2)|}$}
\end{figure}
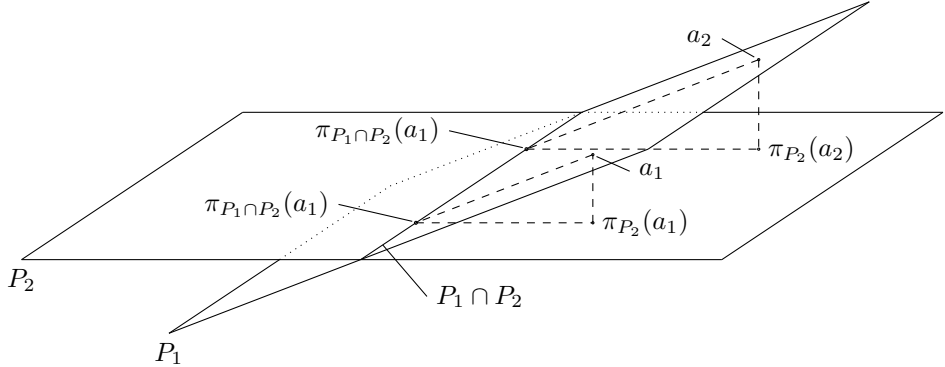

\begin{lem} \label{20.2.2012.4}
	Let $P_1, P_2 \in \mathcal{P}(\n,m)$ with $ \dim P_1 = \dim P_2 = \NN < \n$ and $\dim (P_1 \cap P_2)= \NN -1$.
	For $a_1,a_2 \in P_1 \setminus P_2$, we have
	$\frac{|a_1- \pi_{P_2}(a_1)|}{|a_1- \pi_{P_1 \cap P_2}(a_1)|}
		= \frac{|a_2- \pi_{P_2}(a_2)|}{|a_2- \pi_{P_1 \cap P_2}(a_2)|}$.
\end{lem}
\begin{proof}
	Translate the whole setting so that $P_{1},P_{2}$ are linear subspaces. Then express $a_{1}$ by an orthonormal
	base of $P_{1}$ and compute that $\frac{|a_1- \pi_{P_2}(a_1)|}{|a_1- \pi_{P_1 \cap P_2}(a_1)|}$
	is independent of $a_{1}$.
\end{proof}

\begin{rem} \label{26.08.12.1}
	Let $A,B$ be affine subspaces of $\R^{\n}$ with $A \subset B$ and let $a \in \R^{\n}$. We have
	$\pi_{A}(\pi_{B}(a))=\pi_{A}(a)=\pi_{B}(\pi_{A}(a))$.
\end{rem}

\subsection{Simplices}
\begin{dfn}\label{24.09.13.1}
	Let $x_i \in \R^{\n}$ for $i=0,1,\dots,\NN$. We define 
	$\Delta(x_0,\dots,x_\NN)=\Delta(\{x_0,\dots,x_\NN\})$ \index{$\Delta(x_0,\dots,x_\NN)$} as the convex hull of the 
	set $\{x_0,\dots,x_\NN \}$ and call it \textit{simplex} or $\NN$\textit{-simplex}
	if $\NN$ 
	is the Hausdorff dimension of $\Delta(x_0,\dots,x_\NN)$.
	If the vertices of $T=\Delta(x_0, \dots, x_{\NN})$ are in some set $G \subset \R^{\n}$, i.e., 
	$x_0, \dots, x_{\NN} \in G$,
	we write $T=\Delta(x_0, \dots, x_{\NN}) \in G$.
\end{dfn}

With $\aff(E)$ we denote the smallest affine subspace of $\R^{\n}$ that contains the set $E \subset \R^{\n}$.
\index{\aff}
If $E=\{x_{0}\}$, we set $\aff(E)=\{x_{0}\}$.
\begin{dfn}
	Let $T=\Delta(x_0,\dots,x_\NN) \in \R^{\n}$. For $i,j \in \{0,1, \dots ,\NN \}$ we set
	\begin{align*}
		\face{i}T = \face{x_i}T& = \Delta(\{x_0,\dots ,x_\NN\} \setminus \{x_i\}), \index{$\face{i}T$} \\
		\face{i,j}T = \face{x_i,x_j}T& = \Delta(\{x_0,\dots ,x_\NN\} \setminus \{x_i,x_j\}), \\
		\height{i}T = \height{x_i}T& = d\big(x_i,\aff(\{x_0,\dots,x_\NN\}\setminus \{x_i\})\big).\index{$\height{i}T$}
	\end{align*}
\end{dfn}

\begin{dfn}
	Let $T=\Delta(x_0,\dots,x_\NN)$ be an $\NN$-simplex in $\R^{\n}$. 
	If $\height{i}T \ge \sigma$ for all $i=0,1,\dots,\NN$, we call
	$T$ an $(\NN,\sigma)$\textit{-simplex}.
\end{dfn}

\begin{rem} \label{25.09.2013.10}
	Let $T=\Delta(x_0,\dots,x_\NN)$ an $(\NN,\sigma)$-simplex. For all $i \in \{0,\dots, \NN \}$, we have 
	$ d(x_i,\aff(A_i)) \ge \height{i}T \ge \sigma$
	for every $\emptyset \neq A_i \subset \{x_0,\dots,x_{\NN} \} \setminus \{x_i\}.$
\end{rem}

\begin{dfn} \label{30.09.2013.2}
	Let $T=\Delta(x_0,\dots,x_\NN)$ be an $\NN$-simplex in $\R^{\n}$. 
	By $\cH^{m}(T)$ we denote the volume of $T$ and 
	we define the \textit{normalized volume}
	$\volume(T):= \NN!\ \mathcal{H}^{\NN}(T)$
	which is the volume of the parallelotope spanned by the simplex $T$ (cf. \cite{Stein}).
	We also have a characterisation of $\volume(T)$ by the Gram determinant
	$\volume(T)=\sqrt{\text{Gram}(x_{1}-x_{0},\dots,x_{\NN}-x_{0})}$,
	where the Gram determinant of vectors $v_1, \dots, v_{\NN} \in \R^{\n}$ is defined by
	$\text{Gram}(v_1,\dots,v_{\NN}):=\det\left( (v_{1},\dots,v_{\NN})^{T}(v_1,\dots,v_{\NN}) \right)$.
\end{dfn}

\begin{rem} \label{23.08.12.2}
	Let $T=\Delta(x_0,\dots,x_\NN)$ be an $\NN$-simplex.
	The volume of the parallelotope, spanned by $T$, fulfils
	$\volume(T)= \height{i}T \ \volume(\face{i}T)$
	which implies
	$\mathcal{H}^{\NN}(T)= \frac{1}{\NN} \height{i}T \ \mathcal{H}^{\NN-1}(\face{i}T)$
	for the volume of a simplex.
\end{rem}

\begin{lem} \label{25.01.2012.1}
	Let $T=\Delta(x_0,\dots,x_\NN)$ be an $\NN$-simplex. 
	We have
	$\frac{\height{i}T}{\height{i}\face{j}T} = \frac{\height{j}T}{\height{j}\face{i}T}$.
\end{lem}
\begin{proof}
	We have
	$\frac{\height{i}(T)}{\height{i}(\face{j}T)} 
		 = \frac{\volume(T)}{\height{i}(\face{j}T) \ \volume(\face{i}T)}
		 = \frac{\height{j}(T) \ \volume(\face{j}T)}{\height{i}(\face{j}T) \ \height{j}(\face{i}T) \ \volume(\face{i,j}T)}
		 = \frac{\height{j}(T) \ \volume(\face{j}T)}{\height{j}(\face{i}T) \ \volume(\face{j}T)}
		 = \frac{\height{j}(T)}{\height{j}(\face{i}T)}.$
\end{proof}

\begin{lem} \label{17.11.11.2}
	Let $0 < h < H$, $1 \le \NN \le \n+1$ and $y_0, x_i \in \R^\n$, $i=0,1,\dots, \NN$. 
	If $T_x=\Delta(x_0,\dots,x_\NN)$ is an $(\NN,H)$-simplex and $d(y_0,x_0) \le h$, then $T_y=\Delta(y_0,x_1,\dots,x_\NN)$ 
	is an $(\NN,H-h)$-simplex.
\end{lem}
\begin{proof}
	We have
	$\height{0}T_y \ge \height{0}T_x - d(x_0,y_0) \ge H-h$.
	Now, we show that $\height{1}T_y \ge H-h$. 
	If $m=1$, we have $\height{1}T_{y} = d(y_0,x_1)= \height{0}T_{y}$. So we can assume that $m\ge 2$ for the rest of this 
	proof.
	We set $ z_0 := \pi_{\aff(\face{1}T_y)}(x_0),$ 
	$T_z:=\Delta(z_0,x_1,\dots,x_\NN)$
	and start with some intermediate results:\\
	I.\, 	Due to $\height{0}T_y \ge H-h>0$, $T_y$ is an $\NN$-simplex.\\
	II.\,	We have $d(x_0,z_0) = d(x_0,\aff(\face{1}T_y))\le d(x_0,y_0) \le h$.\\
	III.\, 	We have $z_0=x_2 +r_0(y_0 - x_2) +  \sum_{j=3}^{\NN} r_j(x_j-x_2)$ for some
		$r_i \in \R$, $i=0,3,\dots,\NN$ because $z_0 \in \aff(\face{1}T_y)$.\\
	IV.\,	With III., Remark \ref{2.11.11.1} and because of $\pi_{\aff(\face{0}T_x)}(x_i)=x_i$ for $i=2,\dots \NN$
		we get
		\begin{align*}
			\height{0}T_z &= |z_0 - \pi_{\aff(\face{0}T_x)}(z_0) |
			= | r_0 y_0 - r_0 \pi_{\aff(\face{0}T_x)}(y_0)|
			= r_0 \height{0}(T_y)
		\end{align*}
		and analogously $\height{0}(\face{1}T_z) = r_0 \height{0}(\face{1}T_y)$.\\
	V.\,	With Remark \ref{26.08.12.1}, we get $\pi_{\aff(\face{0,1}T_x)}(z_0)=\pi_{\aff(\face{0,1}T_x)}(x_0)$ and 
		hence we obtain
		\begin{align*}
			\height{0}(\face{1}T_z)
			&= d(\pi_{\aff(\face{1}T_y)}(x_0),\pi_{\aff(\face{0,1}T_x)}(z_0))
			= d(\pi_{\aff(\face{1}T_y)}(x_0),\pi_{\aff(\face{1}T_y)}(\pi_{\aff(\face{0,1}T_x)}(z_0)))\\
			& \le d(x_0,\pi_{\aff(\face{0,1}T_x)}(z_0))
			 = \height{0}(\face{1}T_x).
		\end{align*}
	
	Now, with Lemma \ref{25.01.2012.1} ($i=1$, $j=0$, $T=T_{y}$), IV and V we deduce
	\begin{align*}
		\height{1}T_y 
		& \ge \height{0}T_z \frac{\height{1}(\face{0}T_x)}{\height{0}(\face{1}T_x)}
		 \ge \left( \height{0}T_x - d(x_0,z_0) \right) \frac{\height{1}(\face{0}T_x)}{\height{0}(\face{1}T_x)}.
	\end{align*}
	If $\frac{\height{1}(\face{0}T_x)}{\height{0}(\face{1}T_x)} \ge 1$ this gives us directly $\height{1}T_y \ge H - h$.
	In the other case, use Lemma \ref{25.01.2012.1} and II to obtain
	$\height{1}T_y > \height{1}T_x - d(x_0,z_0) \ge H - h$.
	Since, for $i=2,\dots, \NN$, the points $x_i$ fulfil the same requirements as $x_1$, we are able to prove
	$\height{i}T_y\ge H-h$ for all $i=1,\dots,\NN$ in the same way.
	So, $T_y$ is an $(\NN,H-h)$-simplex.
\end{proof}

\begin{lem} \label{29.02.2012.1}
	Let $C>0$, $1 \le \NN \le \n$ and let $G \subset \R^{\n}$
	be a finite set so that for all $(\NN+1)$-simplices $S=\Delta(x_0, \dots , x_{\NN+1}) \in G$, 
	there exists some $i \in \{0, \dots, \NN +1 \}$ so that $\face{i}(S)$
	is no $(\NN,C)$-simplex.

	Then there exists some $\NN$-simplex $T_z=\Delta(z_0, \dots, z_{\NN}) \in G$ so that for all $a \in G$,
	there exists some $i\in \{0,\dots, \NN \}$ with $d(a, \aff(\face{i}(T_z)) < 2 C$.
\end{lem}
\begin{proof}
	Since $G$ is finite, we are able to choose $T_z=\Delta(z_0, \dots, z_{\NN}) \in G$ so that
	\begin{align} \label{20.2.2012.2}
		\volume(T_z) &= \max_{w_0,\dots, w_{\NN}\in G}\volume(\Delta(w_0, \dots, w_{\NN})).
	\end{align}
	We can assume that $T_z$ is an $(\NN,2C)$-simplex, otherwise there would exist some $i\in \{0,\dots,\NN\}$ with
	$\height{i}(T_z) < 2C$ and so for all $a \in G$ with \eqref{20.2.2012.2} we would obtain
	$d(a,\aff(\face{i}(T_z))) <2C.$

	Now, choose an arbitrary $y_0\in G$. Set $S:=\Delta(y_0,z_0, \dots , z_{\NN})$. The properties of $G$ imply
	that one face of $S$ is no $(\NN,C)$-simplex. Without loss of generality we assume that 
	$T_y:=\face{z_{0}}(S)$ is not an $(\NN,C)$-simplex (but an $\NN$-simplex). So there exists 
	some $i \in \{0,\dots, \NN \}$ with $\height{i}(T_{y})<C$.
	If $i=0$, we are done. So let $i \neq 0$.
	We set $h:= \pi_{\aff(\face{i}T_y)}(z_i)$ and using Remark \ref{26.08.12.1}, we get
	$\pi_{\aff(\face{0,i}T_y)}(h)=\pi_{\aff(\face{i}T_y)}[\pi_{\aff(\face{0,i}T_y)}(z_i)]$. This implies
	\begin{align}
		d(h,\aff(\face{0,i}T_y)) 
		& = d(\pi_{\aff(\face{i}T_y)}(z_i),\pi_{\aff(\face{i}T_y)}[\pi_{\aff(\face{0,i}T_y)}(z_i)])
		 \le \height{i}(\face{0}T_{y}). \label{20.2.2012.3}
	\end{align}
	Now, we use Lemma \ref{20.2.2012.4}, with
	$a_1=y_0$, $a_2=h \in P_1:= \aff(\face{i}(T_y))$, $P_2:=\aff(\face{i}(T_z))$, $P_1 \cap P_2= \aff(\face{0,i}(T_y))$ 
	and \eqref{20.2.2012.3} to obtain 
	\begin{align*}
		\height{0}(\face{i}T_y) 
		& \le \height{i}(\face{0}T_y) \frac{  d(z_i,\aff(\face{i}(T_z)))}{d(z_i,\aff(\face{i}(T_z))) - d( z_i,h)}.
	\end{align*}
	Now use \eqref{20.2.2012.2} to get $d(y_0,\aff(\face{i}(T_z))) \le d(z_i,\aff(\face{i}(T_z)))$ and deduce
	with $d(z_i,\aff(\face{i}(T_z)))=\height{i}T_z \ge 2C$ and $d( z_i,h) =\height{i}(T_y) < C$ that
	$\height{0}(\face{i}T_y) < 2\height{i}(\face{0}T_y)$.
	Finally, with Lemma \ref{25.01.2012.1}, we have
	$d(y_0,\aff(\face{0}(T_z)))= \height{0}(T_y)
		= \height{i}(T_y) \frac{\height{0}(\face{i}T_y)}{\height{i}(\face{0}T_y)} 
		< 2C$.
\end{proof}

\begin{lem}\label{18.11.11.1}
	Let $H>0$, $1 \le \NN \le \n$ and $D\subset \R^{\n}$ be a bounded set. 
	Assume that every simplex $S=\Delta(y_0,\dots,y_{\NN}) \in D$ is not an $(\NN,H)$-simplex.
	Then there exists some $l \in \mathbb{N} \cup \{0\}$, $l \le \NN -1$ and $x_0,\dots,x_l \in \overline{D}$ so that
	$\overline{D} \subset U_{H}(\aff(x_0,\dots,x_l))=\{x \in \R^{\n} | d(x,\aff(x_0,\dots,x_l)\le H\}$.
	\index{$U_{H}$}
\end{lem}
\begin{proof}
	We assume $\# D \ge 2$, otherwise the statement is trivial.
	Let $l \in \{0,\dots,\NN-1\}$ be the largest value such that there exists an $(l, H)$-simplex in $D$.
	If $l=0$, we have $\overline{D} \subset U_H(\aff(x_0))=B(x_0,H)$ for an arbitrary $x_0 \in D$.

	Now suppose $ l \ge 1$. Since $D$ is bounded, there exists $x_{0},\dots,x_{l} \in \bar D$ such that
	the volume $K:=\volume(\triangle(x_0,\dots,x_l))$ is maximal. For some arbitrary $x_{l+1} \in \bar D$
	the definition of $l$ and Lemma \ref{17.11.11.2} imply that $\triangle(x_0,\dots,x_l)$ is not an
	$l+1,H$-simplex.
	Hence there exists some $\tilde l \in \{0,\dots,l+1 \}$ so that $\height{\tilde l}(T)<H$.
	Furthermore we have $\volume(\face{\tilde l}(T)) \le K$ and $\volume(\face{l+1}(T))=K$.
	With Remark \ref{23.08.12.2} we obtain $\height{l+1}(T)=\le H \frac{K}{K}$.
	It follows that $\overline{D} \subset U_H(\aff(x_0,\dots,x_l))$ because $x_{l+1} \in \overline D$  
	was arbitrarily chosen.
\end{proof}

\begin{lem}\label{16.04.2013.1}
	Let $1\le \NN \le \n-1$, $B$ be a closed ball in $\R^{\n}$ and $F \subset B$ be 
	a $\cH^{\NN}$-measurable set with $\cH^{\NN}(F)=\infty$.
	There exists a small constant $0<\sigma=\sigma(F,B)\le \frac{\diam B}{2}$ and 
	some $(\NN+1,(\NN+3)\sigma)$-simplex $T=\Delta(x_0, \dots, x_{\NN+1}) \in B$ with
	$\cH^{\NN}(B(x_0,\sigma) \cap F)= \infty$ and $\cH^{\NN}(B(x_i,\sigma) \cap F) > 0$ for all 
	$i \in \{1, \dots, \NN+1 \}$.
\end{lem}
\begin{proof}
	We set $\mu := \cH^{\NN} \textsf{ L } F$. Since $\mu(B)=\infty$
	there exists some $x_0 \in B$ with $\mu(B(x_0,h)) = \infty$ for all $h>0$. 

	There exists some $c_1 > 0$ with $\mu(B\setminus \mathring{B}(x_0,c_1))>0$. With Lemma \ref{30.08.11.1}, there
	exists some $x_1 \in B \setminus \mathring{B}(x_0,c_1) $ with $\mu(B(x_1,h))>0$ for all $h>0$ and
	the simplex $T_1$ fulfils $\height{1}(T_1) = d(x_0,x_1) \ge c_1$.

	Now we assume that we already have $c_{l} >0$ and a simplex 
	$T_{l}=\Delta(x_0, \dots, x_l) \in \R^{\n}$ with $\height{l}(T_l) \ge c_l$ and $\mu(B(x_i,h))>0$ 
	for all $i \in \{0,\dots, l \}$ and $h>0$ where $ l \le \NN$.
	So there exists some $0<c_{l+1}<\frac{c_l}{2}$ with
	$\mu\left(\left(F \cap B\left(x_0,\frac{c_{l}}{2}\right)\right)\setminus \mathring{U}_{c_{l+1}}(\aff( x_0,\dots,x_{l}))\right)>0$
	and, with Lemma \ref{30.08.11.1}, there exists some $x_{l+1} \in F \subset B$ so that
	$T_{l+1}:=\Delta(x_0, \dots, x_{l+1})$ fulfils $\height{l+1}(T_{l+1}) \ge c_{l+1}$
	and $\mu(B(x_{l+1},h))>0$ for all $h>0$.

	Since $\height{i}(T_{i}) \ge C_{i}> 0$ for all $i \in \{1,\dots,m+1\}$ we obtain $\volume(T)>0$ and hence
	there exists some constant $c>0$ so that $T:=T_{\NN+1}$ is an $(\NN+1,c)$-simplex. 

	To conclude the proof set $\sigma := \frac{c}{\NN+3}$.
\end{proof}

\subsection{Angles between affine subspaces}\label{23.04.2013.1}
\begin{dfn}
	For $G_1,G_2 \in G(\n,\NN )$, we define $\varangle(G_1,G_2):=\|\pi_{G_1} - \pi_{G_2} \|$,
	where the right hand side is the usual norm of the linear map $\pi_{G_1} - \pi_{G_2}$.
	For $P_1,P_2 \in \mathcal{P}(\n,\NN )$, we define
	$\varangle(P_1,P_2):=\varangle(P_1-\pi_{P_1}(0),P_2-\pi_{P_2}(0))$.
\end{dfn}

\begin{rem}\label{11.09.12.2}\label{27.04.12.1}
	For $P_{1},P_{2}, P_{3}  \in \mathcal{P}(\n,\NN )$  and $w \in \R^\n$, we have
	$\varangle(P_1,P_2)= \varangle(P_1,P_2+w)$ and $\varangle(P_1,P_3) \le \varangle(P_1,P_2) + \varangle(P_2,P_3)$.
	The angle $\varangle$ is a metric on the Grassmannian $G(\n,\NN )$ but not on $\mathcal{P}(\n,\NN )$
	because for $P \in \mathcal{P}(\n,\NN )$, there exists some $w \in \R^{\n}$ so that 
	$\varangle(P,P-w)=0$, but $P \neq P-w$.
\end{rem}

\begin{lem} \label{11.09.12.1}
	Let $U \in G(\n,\NN )$ and $v \in \R^{\n}$ with $|v|=|\pi_U(v)|$. Then we have $v = \pi_U(v)$.
\end{lem}
\begin{proof}
	We have $|\pi_U(v)|^2 = |v|^2 = |\pi_U(v) + \pi_U^{\perp}(v)|^2 = |\pi_U(v)|^2 + |\pi_U^{\perp}(v)|^2$
	and so $\pi_U^{\perp}(v)=0$ which implies $v=\pi_U(v)+\pi_U^{\perp}(v)=\pi_U(v)$.
\end{proof}

\begin{lem} \label{12.7.11.1}\label{6.9.11.1}
	Let $P_1, P_2 \in \mathcal{P}(\n,\NN )$ with $\varangle(P_1,P_2) <1$ and $x,y \in P_1$. We have
	\[ d(x,y) \le {\textstyle \frac{d(\pi_{P_{2}}(x),\pi_{P_{2}}(y))}{1-\varangle(P_{1},P_{2})}} \ \
		\text{ and } \ \ d(\pi_{P_{2}}^\perp(x),\pi_{P_{2}}^\perp(y)) 
		\le {\textstyle \frac{\varangle(P_{1},P_{2})}{1-\varangle(P_{1},P_{2})}} d(\pi_{P_{2}}(x),\pi_{P_{2}}(y)). \]
\end{lem}
\begin{proof}
	First assume that $P_{1},P_{2} \in G(\n,\NN )$.
	With $z:= \frac{x-y}{|x-y|} \in P_{1}$ and $\pi_{P_{2}}^{\perp}(z) + \pi_{P_{2}}(z)=z=\pi_{P_{1}}(z)$
	we get
	$| \pi_{P_{2}}^{\perp}(x)-\pi_{P_{2}}^{\perp}(y)| 
		 = |x-y | |\pi_{P_{2}}^{\perp}(z) + \pi_{P_{2}}(z) - \pi_{P_{2}}(z)|
		 \le |x-y | \varangle(P_{1},P_{2})$,
	This implies $d(x,y) \le d(\pi_{P_{2}}(x),\pi_{P_{2}}(y)) + d(x,y) \varangle(P_{1},P_{2})$.
	These two estimates give the assertion in the case $P_{1},P_{2} \in G(\n,\NN )$.
	Now choose $t_1 \in P_1$, $t_2 \in P_2$ such that $P_1-t_1, P_2-t_2 \in G(\n,\NN )$ and
	use Lemma \ref{12.7.11.1}, Remark \ref{24.04.2013.1} and Remark \ref{11.09.12.2}
	to get the whole result.
\end{proof}

\begin{kor} \label{24.04.2012.1}
	Let $P \in \mathcal{P}(\n,\NN)$, $G \in G(\n,\NN)$ and $\varangle(P,G)<1$. There exists some affine map 
	$a :G \to G^{\perp}$ with $G(a)=P$, where $G(a)$ is the graph of the map $a$,
	and $a$ is Lipschitz continuous with Lipschitz constant $\frac{\varangle(P,G)}{1-\varangle(P,G)}$.
\end{kor}
\begin{proof}
	Set $a(y)= \pi_{P_2}^{\perp}(\pi_{P_2}^{-1}{}\big|_{P_1}(y))$ and use Lemma \ref{6.9.11.1}.
\end{proof}

\begin{kor}\label{24.10.11.1}
	Let $G_1,G_2 \in G(\n,\NN)$ and
	$o_1,\dots, o_{\NN}$ be an orthonormal basis of $G_1$. If
	$d(o_i,G_2) \le \tilde \sigma \le \tilde \sigma_1:=10^{-1}(10^{\NN}+1)^{-1}$, then 
	$\varangle(G_1,G_2)\le 4\NN (10^{\NN} + 1) \tilde \sigma$.
\end{kor}
\begin{proof}
	For $i=1,\dots, \NN$, set $h_i:=\pi_{P_2}(o_i)$ and use Lemma 2.3 from \cite{MR3078345}.
\end{proof}

For $x,y \in \R^{\n}$, we set $\langle x,y \rangle$ to be the usual scalar product in $\R^{\n}$. 
\index{$\langle x,y \rangle$}
\begin{lem}\label{30.05.12.1}
	Let $C, \hat C\ge 1$, $t > 0$ and 
	$S=\Delta(y_0,\dots,y_{\NN })$ an $(\NN ,\frac{t}{C})$-simplex with
	$S \subset B(x,\hat C t)$, $x \in \R^\n$. 
	There exists an orthonormal basis $(o_{1},\dots,o_{\NN})$ of $\sspan(y_1-y_0, \dots, y_{\NN}  - y_0) $ and $\gamma_{l,r} \in \R$ so that
	for all $1 \le l \le \NN $ and $1 \le r \le l$ we have
	\[o_l := \sum_{r=1}^{l} \gamma_{l,r} (y_r-y_0) \ \ \ \ \text{ and } \ \ \ \
		|\gamma_{l,r}| \le (2lC \hat C)^{l} \frac{C}{t} \le  (2\NN  C \hat C)^{\NN } \frac{C}{t}.\]
\end{lem}
\begin{proof}
	We set $z_i:=y_i - y_{0}$ for all $i=0,\dots, \NN$, and $R:=\Delta(z_0,\dots,z_{\NN})= S - y_{0}$. 
	We obtain for all $i\in \{1, \dots,\NN \}$ ($S$ is an $(\NN,\frac{t}{C})$-simplex)
	\begin{align}\label{30.05.12.3}
		d(z_i,\aff(z_0,\dots,z_{i-1})) \ge \height{i}(R) & = \height{i}(S) \ge {\textstyle \frac{t}{C} }.
	\end{align}
	Due to $\height{i}(R) \ge \frac{t}{C}>0$, we have that $(z_1,\dots,z_m)$ are linearly independent. 
	So with the Gram-Schmidt process
	we are able to define some orthonormal basis of the $\NN$-dimensional linear subspace $\sspan(z_1,\dots,z_{\NN} )$
	\begin{align*}
		o_1 & :=\gamma_{1,1} z_1,
		& o_{l+1} := \gamma_{l+1,l+1} z_{l+1}- \gamma_{l+1,l+1}\displaystyle{\sum_{i=1}^{l} \langle z_{l+1},o_i \rangle o_i},
	\end{align*}
	where $\gamma_{1,1}:=\frac{1}{|z_{1}|}$ and $\gamma_{l+1,l+1}:=\frac{1}{d(z_{l+1},\aff(z_0,\dots,z_{l}))}$.
	Furthermore we define recursively 
	\[\gamma_{l+1,r} := -\sum_{i=r}^{l} 
		\gamma_{l+1,l+1} \langle z_{l+1},o_i \rangle  \gamma_{i,r}\]
	for $r\in \{1,\dots,l\}$.
	Now we prove by induction that $\gamma_{l,r}$ fulfil the desired properties.
	We have $o_1=\gamma_{1,1} (y_1-y_0)$ and \eqref{30.05.12.3} implies $|\gamma_{1,1}| \le \frac{C}{t}$.
	Now let $1 \le l \le \NN $. We assume that, for all $i \in \{1,\dots,l \}$, $j \in \{1,\dots,i\}$, we have
	$o_i = \sum_{r=1}^{i} \gamma_{i,r} z_r$ and $|\gamma_{i,j}| \le (2lC \hat C)^{l} \frac{C}{t}$.
	We obtain
	\begin{align*}
		o_{l+1} & = \gamma_{l+1,l+1} z_{l+1} - \sum_{i=1}^{l}  \sum_{r=1}^{i} 
			\gamma_{l+1,l+1}\langle z_{l+1},o_i \rangle  \gamma_{i,r} z_r
		 = \sum_{r=1}^{l+1} \gamma_{l+1,r} z_r.
	\end{align*}
	If $r=l+1$, \eqref{30.05.12.3} implies $|\gamma_{l+1,r}| \le \frac{C}{t}$ and if $1 \le r \le l$, we get
	with $|z_{l+1}| \le 2 \hat C t$ 
	\begin{align*}
		|\gamma_{l+1,r}| & \stackrel{\eqref{30.05.12.3}}{\le} \sum_{i=r}^{l}
			\frac{C}{t}|z_{l+1}|  (2lC \hat C)^l \frac{C}{t}
		 < (2(l+1)C \hat C)^{l+1} \frac{C}{t}.
	\end{align*}
\end{proof}

\begin{lem}\label{21.11.11.2}
	Let $C, \hat C \ge 1$, $t > 0$,
	$0<\sigma \le \left( 10(10^{\NN}+1) \NN C (2\NN C \hat C)^{\NN} \right)^{-1}$, 
	$P_1,P_2 \in \mathcal{P}(\n,\NN)$ and $S=\Delta(y_0,\dots,y_{\NN}) \subset P_1$ 
	an $(\NN,\frac{t}{C})$-simplex with $S \subset B(x,\hat C t)$, $x \in \R^{\n}$
	and $d(y_i,P_2) \le t \sigma$ for all $i \in  \{0,\dots,\NN\}$. 
	It follows that 
	\[\varangle(P_1,P_2)\le 4\NN (10^\NN + 1)\left(2\NN  C (2\NN  C\hat C)^{\NN } \right)\sigma.\] 
\end{lem}
\begin{proof}
	Use Lemma \ref{30.05.12.1}, to get some orthonormal basis of 
	\mbox{$\sspan(y_1-y_0, \dots, y_{\NN} - y_0) $} and $\gamma_{l,r} \in \R$.
	We set $\hat y_{0}:=\pi_{P_{2}}(y_{0})$ and we obtain for $1 \le l\le \NN$
	\begin{align*}
		d(o_l,P_2-\hat y_{0})
		& \le \sum_{r=1}^{l} |\gamma_{l,r}| (d(y_r,P_2)+d(y_{0},P_{2}))
		 \le 2\NN  C(2\NN  C\hat C)^{\NN } \sigma.
	\end{align*}
	Setting $\tilde \sigma = 2\NN C(2\NN C \hat C)^{\NN}\sigma \le \frac{1}{10(10^{\NN}+1)} $ 
	the assertion follows with Corollary \ref{24.10.11.1} ($G_{1}=P_{1}-y_{0}$, $G_{2}=P_{2}-\hat y_{0})$.
\end{proof}

\begin{lem} \label{23.03.2012.1}
	Let $\sigma > 0$, $t \ge 0$, $P_1,P_2 \in \mathcal{P}(\n,\NN)$ with $\varangle(P_1,P_2)\le \sigma$ and 
	assume that 
	there exists $p_1 \in P_1$, 
	$p_2 \in P_2$ with $d(p_1,p_2) \le t\sigma$. Then $d(w,P_2) \le \sigma (d(w,p_1)+t)$
	holds for every $w\in P_1$.
\end{lem}
\begin{proof}
	For $w\in P_1$, set $\tilde w:= w - p_1 \in P_1 - p_1$.
	We obtain
	\begin{align*}
		d(w,P_2)
		& \le |\tilde w| \Bigl|\frac{\tilde w}{|\tilde w|}
			-\pi_{P_2-p_2}\left(\frac{\tilde w}{|\tilde w|}\right)\Bigr| + d(p_1,p_2)
		 \le |\tilde w| \varangle(P_1-p_1,P_2-p_2) + t\sigma.
	\end{align*}
\end{proof}

\setcounter{equation}{0}
\section{Integral Menger curvature and rectifiability}
\subsection{Main result}\label{1.11.2014.2}
Let $\N,\n \in \mathbb{N}$ with $1 \le \N < \n$. We start with some definitions. 
\newcommand{\proper}{proper}
\begin{dfn}[Proper integrand] \label{4.10.12.1} \label{muproper}	
	Let $\K : \left(\R^{\n}\right)^{\N+2} \to [0,\infty)$ and $p>1$. 
	We say that $\K^{\p}$ is a \textit{\proper{} integrand} if it fulfils the following four conditions:
	\begin{itemize}
		\item $\K$ is $\left(\cH^{\N}\right)^{\N+2}$-measurable, where $\left(\cH^{\N}\right)^{\N+2}$
			denotes the $\N+2$-times product measure of $\cH^{\N}$.
		\item There exists some constants $c=c(\N,\K,\p) \ge 1$ and $l=l(\N,\K,\p) \ge 1$ so that, for all
			$t>0$, $C \ge 1$, $x \in \R^{\n}$ and all $(\N,\frac{t}{C})$-simplices $\Delta(x_{0},\dots,x_{\N}) \subset B(x,Ct)$, 
			we have
		\[\left(\frac{d(w,\aff(x_0,\dots,x_{\N}))}{t}\right)^{\p} \le c C^{l} t^{\N(\N+1)} \K^{\p}(x_0,\dots,x_{\N},w)\]
		for all $w \in B(x,Ct)$.
		\item For all $t > 0$, we have
		$t^{\N(\N+1)}\K^{\p}(tx_{0},\dots,tx_{\N+1})=\K^{\p}(x_{0},\dots,x_{\N+1})$.
		\item For every $b \in \R^{\n}$, we have $\K(x_{0}+b,\dots,x_{\N+1}+b)=\K(x_{0},\dots,x_{\N+1})$.
	\end{itemize}

\end{dfn}

\begin{rem}
	If instead of the first condition, we have that $\mathcal{K}$ is $\left(\mu\right)^{\N+2}$-measurable for some
	Borel measure $\mu$ on $\R^{\n}$ we call $\mathcal{K}$ \textit{$\mu$-\proper{}}.
\end{rem}

\begin{dfn}
	(i) We call a Borel set $E \subset \mathbb{R}^{\n}$ \textit{purely $\N$-unrectifiable} if for every 
		Lipschitz continuous \index{purely $\N$-unrectifiable}
		function $\gamma :  \mathbb{R}^{\N} \rightarrow  \mathbb{R}^{\n}$, we have
		$\mathcal{H}^{\N}(E \cap \gamma( \mathbb{R}^{\N})) = 0$.\\
	(ii)    A Borel set $E \subset \mathbb{R}^{\n}$ is \textit{$\N$-rectifiable} if there exists some countable
		family of Lipschitz continuous functions
		\index{$\N$-rectifiable}
		$\gamma_{i} :  \mathbb{R}^{\N} \rightarrow  \mathbb{R}^{\n}$ so that
		$ \mathcal{H}^{\N}(E \setminus \bigcup_{i=1}^{\infty} \gamma_{i}( \mathbb{R}^{\N})) = 0$.
\end{dfn}

\begin{dfn}[Integral Menger curvature]
	Let $E \subset \R^{\n}$ be a Borel set and $\mu$ be a Borel measure on $\R^{\n}$. 
	We define the \textit{integral Menger curvature} of $E$ and $\mu$
	with integrand $\K^{\p}$ by 
	$\M_{\K^{\p}}(E):= \M_{\K^{\p}}(\cH^{\n}\big|_{E})$
	and
	\[\M_{\K^{\p}}(\mu):=\int \dots \int \K^{\p}(x_{0},\dots,x_{\N+1}) \ \dd \mu(x_{0}) \dots \dd \mu(x_{\N+1}).\]
\end{dfn}

Now we can state our main result.

\begin{thm} \label{maintheorem}
	Let $E \subset \mathbb{R}^{\n}$ be a borel set with 
	$\mathcal{M}_{\K^{2}}(E) < \infty$, where $\K^{2}$ is some \proper{} integrand.
	Then $E$ is $\N$-rectifiable.
\end{thm}

\subsection{Examples of admissible integrands}\label{22.10.2014.2}
We start with flat simplices. 
\begin{dfn} \label{30.09.2013.1}
	We define the $(\cH^{\N})^{\N+2}$-measurable set
	\[X_{0}:=\left\{(x_{0},\dots,x_{\N+1}) \in (\R^{\n})^{\N+2} \big|
		\text{Gram}(x_1 - x_0, \dots, x_{\N+1}-x_0)=0\right\}\]
	(the Gram determinant is defined in Definition \ref{30.09.2013.2})
	which is the set of all simplices with $\N+2$ vertices
	in $\R^{\n}$ which 
	span at most an $\N$-dimensional affine subspace.
\end{dfn}

The following lemma is helpful to prove that a given integrand fulfils the second condition of a \proper{} integrand.
\begin{lem} \label{25.09.2013}
	Let $t>0$, $C \ge 1$, $x \in \R^{\n}$, $w \in B(x,Ct)$ and let
	$S=\Delta(x_{0},\dots,x_{\N}) \subset B(x,Ct)$ be some
	$(\N,\frac{t}{C})$-simplex. 
	Setting $S_{w}=\Delta(x_0,\dots,x_{\N},w)$,
	$A(S_{w})$ as the surface area of the simplex $S_{w}$ 
	and choosing
	$i,j \in \{0, \dots, \N \}$ with $j \neq i$ we have the following statements:
	\begin{itemize}
		\item $\frac{t}{C} \le d(x_i,x_{j}) \le \diam(S_{w}) \le 2Ct$,
		\item $d(x_{i},w) \le 2Ct$,
		\item $\frac{t^{\N}}{C^{\N} \N!} \le \cH^{\N}(S) \le \frac{(2C)^{\N}}{\N!}t^{\N}$,
		\item $\cH^{\N}(S) \le A(S_{w}) \le [(\N +1 )2C^2+1]\cH^{\N}(S)$,
		\item $d(w,\aff(x_{0},\dots,x_{\N})) =\N \frac{\cH^{\N+1}(S_{w})}{\cH^{\N}(S)}$.
	\end{itemize}
\end{lem}
\begin{proof}
	Since $S$ is an $(\N,\frac{t}{C})$-simplex, we have
	\begin{align} \label{18.10.12.1}
		\frac{t}{C} 
		&\le \height{i}(S) 
		\le  d(x_{i},x_{j}) \le \diam(S_{w}) 
		=    \max_{l,m \in \{0,\dots,\N\}} \left\{d(x_l,x_m),d(x_l,w) \right\}  
		\le  2Ct
	\end{align}
	and because of $x_i,w \in B(x,Ct)$, we get
	$d(x_{i},w) \le 2Ct$.
	Now, with Remark \ref{23.08.12.2}, we conclude that 
	$ \cH^{\N}(S)=\frac{1}{\N!} \prod_{l=0}^{\N-1} d(x_{l},\aff(x_{l+1},\dots,x_{n}))$ 
	which implies with Remark \ref{25.09.2013.10}
	\begin{align*}
		\frac{t^{\N}}{C^{\N} \N!} 
		\stackrel{\eqref{18.10.12.1}}{\le} \frac{1}{\N!} \prod_{l=0}^{\N-1} \height{l}(S)
		\le \cH^{\N}(S) 
		\le \frac{1}{\N!} \prod_{l=0}^{\N-1} d(x_{l},x_{n}))
		\stackrel{\eqref{18.10.12.1}}{\le} \frac{(2C)^{\N}}{\N!}t^{\N}.
	\end{align*}
	Using Remark \ref{23.08.12.2} and $\height{w}(\face{i}(S_w)) \le d(w,x_j) \le 2C t$, we obtain 
	\begin{align*}
		\cH^{\N}(\face{i}(S_w)) 
		& \stackrel{\ref{23.08.12.2}}{=} \frac{1}{\N} \height{w}(\face{i}(S_w)) \cH^{\N-1}(\face{i,w}(S_w))
		\stackrel{\substack{\hphantom{\ref{23.08.12.2}} \\ \eqref{18.10.12.1}}}{\le} \frac{1}{\N} 2C^2 \height{i}(S) \cH^{\N-1}(\face{i}(S))
		\stackrel{\ref{23.08.12.2}}{=} 2C^2 \cH^{\N}(S),
	\end{align*}
		so that with 
		$A(S_{w})=\sum_{i=0}^{\N} \cH^{\N}(\face{i}S_w) + \cH^{\N}(\face{w}S_w)$
		and $\face{w}(S_w)=S$, we get
	\begin{align*}
		\cH^{\N}(S) \le A(S_w) & \le [(\N +1 )2C^2+1] \cH^{\N}(S).
	\end{align*}
	Finally, with Remark \ref{23.08.12.2} and using that $S=\face{w}(S_w)$, we deduce 
	\begin{align*}
		d(w,\aff(x_0,\dots,x_{\N})) &= \height{w}(S_{w})
		= \frac{\height{w}(S_w) \cdot \cH^{\N}(\face{w}(S_w))}{\cH^{\N}(S)}
		=  \frac{\N \cH^{\N+1}(S_{w})}{\cH^{\N}(S)}.
	\end{align*}
\end{proof}

Now we can state some examples of \proper{} integrands. 
Use the previous lemma to verify the second condition.
We define all following examples to be $0$ on $X_{0}$ and will only give
an explicit definition on $(\R^{\n})^{\N+2} \setminus X_{0}$.
We mention that our main result is only valid for all
integrands which are \proper{} for integrability exponent $\p=2$.

\subsubsection*{Proper Integrands with exponent $2$}
We start with the one used in the introduction of this work.
Let $x_{0}, \dots, x_{\N+1} \in (\R^{\n})^{\N+2} \setminus X_{0} $ and set
\[ \K_{1}(x_{0},\dots,x_{\N+1}):= 
	\displaystyle{\frac{\cH^{\N+1}(\Delta(x_{0},\dots,x_{\N+1}))}{\Pi_{0 \le i < j \le \N+1}d(x_{i},x_{j})}}, \]
then $\K_{1}^{2}$ is \proper{}.
The next \proper{} integrand is used by Lerman and Whitehouse in \cite{MR2848529,MR2558685},
\[ \K_{2}^{2}(x_{0},\dots,x_{\N+1})
	:= \frac{1}{\N+2} \cdot \frac{\textnormal{Vol}_{\N+1}(\Delta(x_{0},\dots,x_{\N+1}))^{2}}{\diam(\Delta(x_{0},\dots,x_{\N+1}))^{\N(\N+1)}}
	\sum_{i=0}^{\N+1} \frac{1}{\prod_{\genfrac{}{}{0pt}{}{j=0}{j\neq i}}^{\N+1}|x_{j}-x_{i} |^{2}},\]
where $\textnormal{Vol}_{\N+1}$ is $(\N+1)!$ times the volume of the simplex
$\Delta(x_{0}, \dots,x_{\N+1})$, which is equal to the volume of the parallelotope spanned by 
this simplex, cf. Definition \ref{30.09.2013.2}.
The following proper integrand, $\K_{3}^{2}$, is mentioned among others in \cite[section 6]{MR2558685}:
\[ \K_{3}(x_{0},\dots,x_{\N+1}):= 
	\displaystyle{\frac{\cH^{\N+1}(\Delta(x_{0},\dots,x_{\N+1}))}{\diam \Delta(x_{0},\dots,x_{\N+1})^{\frac{(\N+1)(\N+2)}{2}}}}.\]

\subsubsection*{Proper Integrands with exponents different from $2$}

Now we present some integrands for integral Menger curvature used in several papers, where the scaling behaviour
implies that our main result can not be applied. Nevertheless, most of our partial results
are valid also for these integrands.
The first integrand we consider was introduced for $\N=2, \n=3$ in \cite{SvdMsurface},
\[ \K_{4}(x_{0},\dots,x_{\N+1}):= 
	\displaystyle{\frac{V(T)}{A(T) (\diam T)^2}},\]
where $V(T)$ is the volume of the simplex $T=\Delta(x_0,\dots,x_{\N+1})$ and $A(T)$ 
is the surface area of $T$. $\K_{4}^{\p}$ is a \proper{} integrand with $p=\N(\N+1)$.
The next one, $\K_{5}^{p}$, is a \proper{} integrand with $p=\N(\N+1)$ and is used, for example, in \cite{MR2921162,MR3061777},
\[ \K_{5}(x_{0},\dots,x_{\N+1}):= 
	\displaystyle{\frac{\cH^{\N+1}(\Delta(x_{0}, \dots, x_{\N+1}))}{\diam (\Delta(x_{0}, \dots, x_{\N+1}))^{\N+2}}}.\]
Finally, L\'{e}ger suggested the following integrand in \cite{Leger} for a higher dimensional analogue of 
his theorem. Unfortunately, we can not confirm his suggestion. This one, $\K_{6}^{p}$, is a \proper{} integrand with $p=(\N+1)$
where
\[ \K_{6}(x_{0},\dots,x_{\N+1}):= 
	\displaystyle{\frac{d(x_{\N+1},\aff(x_0,\dots,x_{\N}))}{d(x_{\N+1},x_0) \dots d(x_{\N+1},x_{\N})}}.\]
Hence our main result does \textit{not} apply for $n \neq 1$.
For $\N=1$ up to a factor of $2$, this integrand gives the inverse of the circumcircle of 
the three points $x_{0}, x_{1}, x_{2}$.

\setcounter{equation}{0}
\section{\texorpdfstring{$\beta$-numbers}{{\ss}-numbers}}\label{beta}
In this chapter, let $C_0 \ge 10$ and $\mu$ a Borel measure on $\R^{\n}$ with compact 
support $F$ that is upper Ahlfors regular, i.e., 
\begin{enumerate} \label{GrundeigenschaftenBetaBeta}
	\renewcommand{\labelenumi}{(\Alph{enumi})}
	\setcounter{enumi}{1}
	\item for every ball $B$ we have $\mu(B) \le C_{0} (\diam B)^{\N}$.
\end{enumerate}
If $B=B(x,r)$ is some ball in $\R^{\n}$ with centre $x$ and radius $r$ and $t \in (0,\infty)$, then
we set $tB:=B(x,tr)$. Distinguish this notation from the case $t\Upsilon=\{tz|z \in \Upsilon\}$ 
where $\Upsilon \subset \R^{\n}$ is some arbitrary set.
Furthermore, in this and the following chapters, we assume that every ball is closed. We need this to apply
Vitali's and Besicovitch's covering theorems. By $C$, we denote a generic constant with a fixed value which
may change from line to line.

\subsection{Measure quotient}
\begin{dfn}[Measure quotient]\label{Definitionvondeltaschlange}
	For a ball $B=B(x,t)$ with centre $x \in \mathbb{R}^{\n}$, radius $t > 0$ and 
	a $\mu$-measurable set $\Upsilon \subset \mathbb{R}^{\n}$, we define the \textit{measure quotient}
	\begin{align*}
		\delta(B \cap \Upsilon) = \delta_{\mu}(B \cap \Upsilon) &:= \frac{\mu(B(x,t) \cap \Upsilon)}{t^{\N}}.
	\end{align*}
	In most instances, we will use the special case $\Upsilon = \R^{\n}$ and write $\delta(B)$ instead of
	$\delta(B \cap \R^{\n})$.
\end{dfn}
This measure quotient compares the amount of the support $F$ contained in a ball with the size of this ball.
The following lemma states that if we have a lower control on the measure quotient of some ball, then we can find
a not too flat simplex contained in this ball, where at each vertex we have a small ball with a lower control 
on its quotient measure.

\begin{lem} \label{lem2.3}
	Let $0 < \lambda \le 2^{\N}$ and $N_{0}=N_{0}(\n)$ be the constant from 
	Besicovitch's covering theorem \cite[1.5.2, Thm. 2]{Evans} depending only on the dimension $\n$. There exist constants 
	$C_{1} := \frac{4\cdot120^\N \N^{\N+1} N_{0}C_{0}}{\lambda}>3$
	and $ C_{2} := \frac{2^{\N+2}N_{0}C_{1}^\N}{\lambda}>1$
	so that for a given 
	ball $B(x,t)$ and some $\mu$-measureable set $\Upsilon$ 
	with $ \delta(B(x,t)\cap \Upsilon) \ge \lambda$, there exists 
	some $T=\Delta(x_0,\dots,x_{\N+1}) \in F \cap B(x,t) \cap \Upsilon$  
	so that 
	$\face{i}(T)$ is an $(\N,10\N \frac{t}{C_1})$-simplex and
	$\mu\left( B\left(x_i,\frac{t}{C_1}\right) \cap B(x,t) \cap \Upsilon\right) \ge  \frac{t^{\N}}{C_{2}}$
	for all $i \in \{0,\dots,\N+1 \}$.
\end{lem}
\begin{proof}
	Let $B(x,t)$ be the ball with $\delta(B(x,t) \cap \Upsilon) \ge \lambda $ and
	$\mathcal{F} := \{ B(y,\frac{t}{C_{1}}) | y \in F \cap B(x,t) \cap \Upsilon \}$.
	With Besicovitch's covering theorem \cite[1.5.2, Thm. 2]{Evans} we get $N_{0}=N_{0}(\N)$ families 
	$\mathcal{B}_{m} \subset \mathcal{F}$, $m=1,...,N_{0}$ of disjoint balls so that
	$ F \cap B(x,t) \cap \Upsilon \subset \bigcup_{m=1}^{N_{0}} \ \dot{\bigcup_{B \in \mathcal{B}_{m}}} B$.
	We have
	\begin{align*}
		\lambda & \le \frac{1}{t^{\N}} \mu \left(\bigcup_{m=1}^{N_{0}}
			 \bigcup_{B \in \mathcal{B}_{m}} (B \cap B(x,t) \cap \Upsilon)\right) 
		 \le \frac{1}{t^{\N}} \sum_{m=1}^{N_{0}} \sum_{B\in \mathcal{B}_{m}} \mu(B \cap B(x,t) \cap \Upsilon)
	\end{align*}
	and hence there exists a family $\mathcal{B}_{m}$ with
	\begin{eqnarray} \label{sterneinss}
	 	\sum_{B \in \mathcal{B}_{m}} \mu(B \cap B(x,t) \cap \Upsilon) \ge \frac{\lambda t^{\N}}{N_{0}}.
	\end{eqnarray}
	We assume that for every $S=\Delta(y_0,\dots,y_{\N +1}) \in F \cap B(x,t) \cap \Upsilon$, 
	there exists some $i \in \{0,\dots,\N +1 \}$ so that either
	$\face{i}(S)$ is no $(\N,10\N\frac{t}{C_1})$-simplex or 
	$\mu( B(y_i,\frac{t}{C_1}) \cap B(x,t) \cap \Upsilon) <  \frac{t^{\N}}{C_{2}}$. 
	We define
	$\mathcal{G} := \left\{ B \in \mathcal{B}_{m} \Big| \mu(B \cap B(x,t)\cap \Upsilon) \ge \frac{t^{\N}}{C_{2}} \right\}$
	and mention that $\mathcal{G}$ is a finite set since Lemma \ref{22.2.2012.1} implies that 
	$\#\mathcal{B}_{m} \le (2C_{1})^{\N}$. 
	With Lemma \ref{29.02.2012.1} (where we set $G$ as the set of centres of balls in $\mathcal{G}$ and $C=10\N\frac{t}{C_1}$),
	we know that there exists some 
	$T_z=\Delta(z_0,\dots,z_\N)$ so that for every ball $B(y,\frac{t}{C_1}) \in \mathcal{G}$, 
	there exists some $i \in \{0,\dots,\N \}$ so that $d(y,\aff(\face{i}(T_z))) \le 20\N\frac{t}{C_1}$.
	We define for $i \in \{0,\dots,\N \}$
	\begin{align*}
		T_i & := \aff(\face{i}(T_z)) \cap B(\pi_{\aff(\face{i}(T_z))}(x),2t), \\
		\mathcal{S}_i & := \left\{ y \in \mathbb{R}^n | d(y,\aff(\face{i}(T_z))) \le {\textstyle \frac{30\N t}{C_1}}, \pi_{\aff(\face{i}(T_z))}(y) \in T_i \right\}
	\end{align*}
	and we know that
	$B \in \mathcal{G}$ implies $B \subset S_i$ for some $i \in \{0,\dots,\N \}$.
	With Lemma \ref{23.2.2012.1} applied to $B(x,r)=T_i$, $s=\frac{4}{C_1}t < 2t=r$ and $m=\N-1$, there exists a family 
	$\mathcal{E} $ of disjoint closed balls
	with $\diam B = \frac{8}{C_1}t$ for all $B \in \mathcal{E}$,
	$T_i \subset \bigcup_{B \in \mathcal{E}} 5B$ and
	$\# \mathcal{E}\le C_1^{\N-1}$.
	Let $y \in S_i$. We have $d(y,\aff(\face{i}(T_z))) \le \frac{30\N}{C_1}t$ and $\pi_{\aff(\face{i}(T_z))}(y) \in T_i$.
	So, there exists some $B=B(z,\frac{4}{C_1}t) \in \mathcal{E}$ with $\pi_{\aff(\face{i}(T))}(y) \in 5B$ and we have
	$d(y,z) \le \frac{30 \N}{C_1}t + 5\frac{4}{C_1}t < \frac{60\N}{C_1}t$.
	This proves $S_i \subset \bigcup_{B \in \mathcal{E}} 15\N B$.
	We therefrom derive with (B) (see page \pageref{GrundeigenschaftenBetaBeta})
	\begin{align} \label{18.2.11.1}
	    \mu(S_i) &\le \sum_{B \in \mathcal{E}} \mu\left(15 \N B\right)
	    \stackrel{\text{(B)}}{\le} \sum_{B \in \mathcal{E}} C_0\left(15\N \diam B\right)^{\N} 
	    \le \# \mathcal{E} C_0 \frac{(120 \N)^\N t^{\N}}{C_1^{\N}} 
	    \le (120 \N)^\N C_0 \frac{t^{\N}}{C_1}.
	\end{align}
	We define for $i \in \{1,\dots, \N \}$
	\begin{align*}
		\mathcal{G}_0 &:= \left\{ B \in \mathcal{G} | B \subset S_0 \right\}, \hspace{9mm} \text{and} \hspace{10mm}
		\mathcal{G}_i := \left\{ B \in \mathcal{G} |  B \subset S_i \text{ and } B \notin {\textstyle \bigcup_{j=0}^{i-1}} \mathcal{G}_i  \right\}
	\end{align*}
	as a partition of $\mathcal{G}$ (compare the remark after the definition of $\mathcal{S}_{i}$).
	Now we have
	\begin{align*}
	 \sum_{B \in  \mathcal{G}} \mu(B \cap B(x,t)\cap \Upsilon) 
	  & \le \sum_{i=0}^{\N}\mu(S_i) \stackrel{\eqref{18.2.11.1}}{\le} \N ( 120 \N)^\N C_0 \frac{t^{\N}}{C_1}.
	 \end{align*}
	 Moreover, we have
	 \[ \sum_{B \in \mathcal{B}_{m} \setminus  \mathcal{G}} \mu(B \cap B(x,t)\cap \Upsilon) < 
		\sum_{B \in \mathcal{B}_{m} \setminus  \mathcal{G}} 
		\frac{t^{\N}}{C_{2}} \stackrel{\#\mathcal{B}_m\le (2C_1)^n}{\le} (2 C_{1})^{n} \frac{t^{\N}}{C_{2}}.\]
	 All in all, we get with (\ref{sterneinss}) and the definition of $C_{1}$ and $C_{2}$
	 \[ \lambda \le N_{0} \frac{1}{t^{\N}}\left(2^{n}t^{\N} \frac{C_{1}^{n}}{C_{2}} + 120^{\N} \N^{\N +1} t^{\N} C_{0} \frac{1}{C_{1}}\right) 
	 	= N_{0}\left(2^{n} \frac{C_{1}^{n}}{C_{2}} 
		+ 120^{\N} \N^{\N +1} C_{0} \frac{1}{C_{1}}\right) \le \frac{\lambda}{2}, \]
	 thus in contradiction to $\lambda > 0$. This completes the proof of Lemma \ref{lem2.3}.
\end{proof}

In most instances, we will use a weaker version of Lemma \ref{lem2.3}:
\begin{kor} \label{04.09.12.1}
	Let $0 < \lambda \le 2^{\N}$. There exist constants $C_{1}=C_{1}(\n,\N,C_{0},\lambda)>3$ and 
	$ C_{2}=C_{2}(\n,\N,C_{0},\lambda)>1$ so that for a given ball $B(x,t)$ and some $\mu$-measurable set 
	$\Upsilon$ with $ \delta(B(x,t)\cap \Upsilon) \ge \lambda$, there exists 
	some $(\N,10\N \frac{t}{C_1})$-simplex $T=\Delta(x_0,\dots,x_{\N}) \in F \cap B(x,t) \cap \Upsilon$  
	so that $\mu\left( B\left(x_i,\frac{t}{C_1}\right) \cap B(x,t) \cap \Upsilon\right) \ge  \frac{t^{\N}}{C_{2}}$
	for all $i \in \{0,\dots, \N \}$.
\end{kor}

\subsection{\texorpdfstring{$\beta$-numbers and integral Menger curvature}{\ss -numbers and integral Menger curvature}}\label{1.11.2014.1}
\begin{dfn}[$\beta$-numbers]
	Let $k > 1$ be some fixed constant, $x \in \mathbb{R}^{\n}$, $t>0$, $B=B(x,t)$, $\p\ge 1$, $ \mathcal{P}(\n,\N) $ the set of all 
	$\N$-dimensional planes in $\mathbb{R}^{\n}$ 
	and $ P \in \mathcal{P}(\n,\N)$. We define
	\begin{align*}
		 \beta_{\p;k}^{P}(B) = \beta_{\p;k}^{P}(x,t) = \beta_{\p;k;\mu}^{P}(x,t) & := \left( \frac{1}{t^{\N}} \int_{B(x,kt)} \left( \frac{d(y,P)}{t} \right)^{\p} 
		 	\dd  \mu (y) \right)^{\frac{1}{\p}}, \\
		 \beta_{\p;k} (B) = \beta_{\p;k} (x,t) = \beta_{\p;k;\mu} (x,t) & := \inf_{P \in \mathcal{P}(\n,\N)} \beta_{\p;k}^{P}(x,t).
	\end{align*}
\end{dfn}

The $\beta$-numbers measure how well the support of the measure $\mu$ can be approximated by some plane. A small 
$\beta$-number of some ball implies either a good approximation of the support by some plane or a low 
measure quotient $\delta$ (cf. Definition \ref{Definitionvondeltaschlange}).
Hence, since we are interested in good approximations by planes, we will use the $\beta$-numbers mainly for balls 
where we have some lower control on the measure quotient.

\begin{dfn}[Local version of $\mathcal{M}_{\mathcal{K}^{\p}}$]
	For $\kappa >1$, $x \in \R^{\n}$, $t>0$, $\p>0$, we define
	\[ \mathcal{M}_{\mathcal{K}^{\p};\kappa}(x,t):= \idotsint_{\mathcal{O}_{\kappa}(x,t)}\mathcal{K}^{\p}(x_0,\dots,x_{\N+1}) 
		\dd\mu(x_0) \dots \dd \mu(x_{\N+1}),\]
	where $\K^{\p}$ is a $\mu$-\proper{} integrand (cf. Definition \ref{muproper} on page \pageref{muproper}) and
	\[ \mathcal{O}_{\kappa}(x,t) := \left\{ (x_0,\dots,x_{\N+1}) \in (B(x,\kappa t))^{\N+2} 
		\Big|  d(a,b) \ge \frac{t}{\kappa}, \forall \ a ,b \in \{x_0,\dots,x_{\N+1} \},a \neq b \right\}. \]
\end{dfn}

\begin{thm} \label{lem2.5}
	Let $\K^{\p}$ be a symmetric $\mu$-\proper{} integrand and let
	$0<\lambda < 2^{\N}$, $k > 2$, $k_{0} \ge 1$.
	There exist constants $k_{1}=k_{1}(\n,\N,C_{0},k,k_{0},\lambda)> 1$ and 
	$C=C(\n,\N,\K,\p,C_{0},k,k_{0},\lambda) \ge 1$
	such that if $x \in \mathbb{R}^{\n}$ and $t > 0$ with $\delta(B(x,t)) \ge \lambda$ for every 
	$y \in B(x,k_{0}t) $, we have
	\[ \beta_{\p;k}(y,t)^{\p} \le C \frac{\mathcal{M}_{\mathcal{K}^{\p};k_{1}}(x,t)}{t^{\N}} 
		\le C \frac{\mathcal{M}_{\mathcal{K}^{\p};k_{1}+k_{0}}(y,t)}{t^{\N}}.\]
\end{thm}
\begin{proof}
	With Lemma \ref{lem2.3} for $\Upsilon=\R^{\n}$, 
	there exists some $T=\Delta(x_0,\dots,x_{\N+1}) \in F \cap B(x,t)$  
	so that $\face{i}(T)$ is an $(\N,10\N \frac{t}{C_1 })$-simplex and
	$\mu\left( B\left(x_i,\frac{t}{C_1 }\right) \cap B(x,t) \right) \ge  \frac{t^{\N}}{C_{2} }$
	for all $i \in \{0,\dots,\N+1 \}$
	where $C_1, C_{2}$ are the constants from Lemma \ref{lem2.3} depending on the present constant $\lambda > 0$,
	the constant $C_0$ determined in (B) on page \pageref{GrundeigenschaftenBetaBeta}, as well as $\n$ and $ \N$.
	We set $B_i:=B\left(x_i,\frac{t}{C_1}\right)$, 
	$k_{1} := \mathrm{max}(C_{1},(2+k + k_{0})) > 1$ \label{19.10.12.10} and go on with some intermediate results.

	I. \, Let $z_i \in B_i$ for all $i \in \{0,\dots,\N+1\}$, 
			$w \in B(x,(k+k_0)t) \setminus \bigcup_{\genfrac{}{}{0pt}{}{l=0}{l\neq j}}^{\N+1} 2B_l$
		or $w \in 2B_{j}$
		for some fixed $j \in \{0,\dots, \N+1\}$.
		Since $\face{i}(T)$ is an $(\N,10\N \frac{t}{C_1 })$-simplex we obtain
		$(z_{0},\dots,\hat z_{j}, \dots, z_{\N+1},w) \in \mathcal{O}_{k_{1}}(x,t)$,
		where $(z_{0},\dots,\hat z_{j}, \dots, z_{\N+1},w)$ denotes the $(\N+2)$-tuple 
		$(z_{0},\dots,z_{j-1},z_{j+1}, \dots, z_{\N+1},w)$.
	
	II. \, Let $z_{i} \in B_{i}=B(x_{i},\frac{t}{C_{1}})$ for all $i \in \{0,\dots,\N+1\}$.
		Then Lemma \ref{17.11.11.2} implies that
		$\face{i}(\Delta(z_0,\dots,z_{\N+1}))$ is an $\left(\N,(9\N-1) \frac{t}{C_1}\right)$-simplex for all 
			 $i \in \{0,\dots,\N+1 \}$.
	
	III. \, Let $z_{i} \in B_{i}=B(x_{i},\frac{t}{C_{1}})$ for all $i \in \{0,\dots,\N+1\}$,
		$w \in B(x,(k+k_{0})t)$.
		Since $\K^{\p}$ is a $\mu$-\proper{} integrand with II. there exists some constant 
		$\tilde C=\tilde C(\n,\N,\K,\p,C_{0},k,k_{0},\lambda)$ so that for all $j\in \{0,\dots,\N+1\}$, we have
		\begin{align*}
			\left( \frac{d(w,\aff(z_{0},\dots,\hat z_{j},\dots,z_{\N+1}))}{t} \right)^{\p}
			& \le \tilde C t^{\N(\N+1)}\K^{\p}(z_{0},\dots,\hat z_{j},\dots,z_{\N+1},w).
		\end{align*}
	
	IV. \, There exist some constant $C=C(\n,\N,\K,\p,C_0,k,k_0,\lambda)$ and $z_i \in F \cap B_i \cap B(x,t)$, 
		$i \in \{0,\dots, \N+1\}$, so that for all $l \in \{0,\dots,\N+1 \}$,
		we have
		\begin{equation} \label{24.08.12.1}
			\int \Eins_{\{ (z_0,\dots,\hat z_l, \dots,z_{\N+1},w) \in \mathcal O _{k_{1}}(x,t)\}} 
				\mathcal{K}^{\p}(z_0,\dots,\hat z_l, \dots,z_{\N+1},w) \dd  \mu(w) 
				\le C \ \frac{\mathcal{M}_{\K^{\p};k_{1}}(x,t)}{t^{(\N+1)\N}}
		\end{equation}
		and with $P_{\N+1}:=\aff(z_0,\dots,z_\N)$
		\begin{equation} \label{24.08.12.2}
			\left( \frac{d(z_{\N+1},P_{\N+1})}{t} \right)^{\p} 
			\le C \ \frac{\mathcal{M}_{\K^{\p};k_{1}}(x,t)}{t^{\N}}.
		\end{equation}
	\begin{proof}
	For $E \subset \R^{\n}$ with $\# E=m+1$,  $E=\{e_0,\dots,e_m \}$, $0 \le m \le \N$, we set
	\begin{multline*} 
		\mathcal{R}(E):= \int_{
		F^{\N-m+1}}\Eins_{\{(e_0,\dots,e_m,w_{m+1},\dots,w_{\N+1}) \in \mathcal{O}_{k_1}(x,t)\}} \\
		\mathcal{K}^{\p}(e_0,\dots,e_m,w_{m+1},\dots,w_{\N+1}) 
	\dd \mu(w_{m+1}) \dots \dd \mu(w_{\N+1}).
	\end{multline*}
	The integrand $\K$ is symmetric, hence
	the value $\mathcal{R}(E)$ is well-defined because it does not depend on the
	numbering of the elements of $E$. In the following part, we use the convention that
	$\{0,\dots,-1\}=\emptyset$ and $\{z_{0},\dots,z_{-1}\}=\emptyset$.
	At first, we show by an inductive construction that, for all $m \in \mathbb{N}$ with $0 \le m \le \N+1$,
	there holds:

	For all $j \in \{0,\dots, m\}$ and $ i \in \{j,\dots,\N+1\}$, there 
	exist constants $C^{(j)}>1$,
	sets $Z_{i}^{j} \subset F \cap B_{i} \cap B(x,t)$ 
	and, for all $l \in \{0,\dots,m-1\}$, there exist $z_{l} \in Z_{l}^{l}$
	with 
	\begin{align} \label{28.08.12.1}
		\mu(Z_{i}^{j}) &> \frac{t^{\N}}{2^{j+1}C_{2}},
	\end{align}
	and, for all $u \in \{0,\dots, m\}$, 
	for all $E \subset\{z_{0},\dots,z_{u-1}\}$ and
	$z \in Z_{r}^{u}$, where $r \in \{u,\dots,n+1\}$, we have
	\begin{align} \label{23.08.12.1}
		\mathcal{R}(E \cup \{z\}) \le C^{(u)} \frac{\mathcal{M}_{\K^{\p};k_{1}}(x,t)}{t^{(\#E+1)\N}}.
	\end{align}

	We start with $m=j=0$ and choose the constant $C^{(0)} := 2 C_{2} $, set $\Upsilon_{i}:=F \cap B_i \cap B(x,t)$
	and define for every $i \in \{0,\dots,\N+1 \}$
	\begin{eqnarray} \label{2.5;1}
	Z_{i}^{0} := \left\{ z \in \Upsilon_{i}  \Big| \mathcal{R}(\{z\}) \le C^{(0)} \frac{\mathcal{M}_{\K^{\p};k_{1}}(x,t)}{t^{\N}}  \right\}.
	\end{eqnarray}
	We have 
	$\mu(Z_{i}^{0}) \ge \mu(\Upsilon_{i}) - \mu(\Upsilon_{i}\setminus Z_{i}^{0})> \frac{t^{\N}}{2C_{2}}$
	because
	$\mu(\Upsilon_{i}) \stackrel{\text{(ii)}}{\ge} \frac{t^{\N}}{C_{2} }$,
	and with \eqref{2.5;1}, Chebyshev's inequality and 
	$\int \mathcal{R}(\{z\}) \dd  \mu (z)=\mathcal{M}_{\K^{\p};k_{1}}(x,t)$ we obtain
	$\mu(\Upsilon_{i}\setminus Z_{i}^{0}) < \frac{t^{\N}}{C^{(0)}}$.
	If $u=0$, $E \subset \{z_{0},\dots,z_{-1}\}=\emptyset$ and  
	$z \in Z_{r}^{0}$, where $r \in \{0,\dots,\N+1\}$, the definition \eqref{2.5;1} implies \eqref{23.08.12.1}
	in this case.
	
	Now let $m \in \{0,\dots,\N \}$ and we assume that 
	for all $j \in \{0,\dots, m\}$ and $ i \in \{j,\dots,\N+1\}$, there 
	exist constants $C^{(j)}>1$,
	sets $Z_{i}^{j} \subset F \cap B_{i} \cap B(x,t)$ 
	and for all $l \in \{0,\dots,m-1\}$ there exist $z_{l} \in Z_{l}^{l}$ with 
	\begin{align} \label{16.7.12.1}
		\mu(Z_{i}^{j}) &> \frac{t^{\N}}{2^{j+1}C_{2}},
	\end{align}
	and for all $u \in \{0,\dots, m\}$, 
	for all $E \subset\{z_{0},\dots,z_{u-1}\}$ and
	$z \in Z_{r}^{u}$ where $r \in \{u,\dots,n+1\}$, we have
	\begin{align} \label{6.5.2013.2}
		\mathcal{R}(E \cup \{z\}) \le C^{(u)} \frac{\mathcal{M}_{\K^{\p};k_{1}}(x,t)}{t^{(\#E+1)\N}}.
	\end{align}

	Next we start with the inductive step.
	From the induction hypothesis, we already have the constants $C^{(j)}$ and the sets $Z_{i}^{j}$
	for $j \in \{0,\dots, m\}$ and $ i \in \{j,\dots,\N+1\}$ as well as $z_{l} \in Z_{l}^{l}$ 
	for $l \in \{0,\dots,m-1\}$. Since $\mu(Z_{m}^{m})>0$, we can choose $z_{m} \in Z_{m}^{m}$.
	We define 
	$C^{(m+1)}:= 2^{2m+2} C^{(m)} C_{2}$
	and, for $ i \in \{m+1, \dots, \N+1 \}$, we define
	\begin{align} \label{2.5;3.1}
		Z_{i}^{m+1} 
		& := \bigcap_{\genfrac{}{}{0cm}{}{E \subset \{z_0,\dots,z_m\}}{z_m \in E}} \underbrace{
		\left\{ z \in Z_{i}^{m} \Big| \mathcal{R}(E \cup \{z \}) \le 
		C^{(m+1)} \frac{\mathcal{M}_{\K^{\p};k_{1}}(x,t)}{t^{(\#E+1)\N}} \right\}}_{=:D_{i,E}^{m}}.
	\end{align}
	We have $\mu(Z_i^{m+1}) \ge \mu(Z_{i}^{m}) - \mu \left(  Z_{i}^{m} \setminus Z_{i}^{m+1} \right) \ge \frac{t^{\N}}{2^{m+2}C_{2}}$ 
	for all $i \in \{m+1,\dots,\N+1 \}$ because if
	$E \subset \{z_{0}, \dots, z_{m}\}$ with $z_{m} \in E$, we get, using \eqref{2.5;3.1}, Chebyshev's inequality,
	$\int \mathcal{R}(E \cup \{ z\}) \dd  \mu (z) =\mathcal{R}((E \setminus \{z_{m}\}) \cup \{z_{m}\})$
	and \eqref{6.5.2013.2} that
	\begin{align*}
		\mu \left(  Z_{i}^{m} \setminus D_{i,E}^{m} \right) 
		& < \left( C^{(m+1)} \frac{\mathcal{M}_{\K^{\p};k_{1}}(x,t)}{t^{(\#E+1)\N}} \right)^{-1} 
			\mathcal{R}((E \setminus \{z_{m}\}) \cup \{z_{m}\})
		= \frac{C^{(m)}}{C^{(m+1)}} t^{\N} 
	\end{align*}
	which implies
	\begin{align*}
		\mu(Z_i^m \setminus Z_i^{m+1}) 
		& \le \sum_{\genfrac{}{}{0cm}{}{E \subset \{z_0,\dots,z_m\}}{z_m \in E}} \mu\left( Z_i^m \setminus D_{i,E}^{m} \right)
		< \frac{1}{2^{m+2}C_{2} }t^{\N}.
	\end{align*}
	Now let $u \in \{0,\dots,m+1\}$ and $E \subset \{z_0,\dots,z_{u-1}\}$ 
	and $z \in Z_{r}^{u}$ where $r \in \{u,\dots,\N+1\}$.
	We have to show that \eqref{23.08.12.1} is valid.
	Due to the induction hypothesis and $z \in Z_{r}^{m+1} \subset Z_{r}^{v}$ for all $v \in \{0,\dots,m+1\}$,
	we only have to consider the case $u=m+1$ and $z_{m} \in E$. Then the inequality follows from \eqref{2.5;3.1}. 
	\hfill
	End of induction.

	Now we construct $z_{\N+1}$.	
	We set $P_{\N+1}:=\aff(z_{0},\dots,z_{\N})$, $\hat C^{(\N+1)} := \tilde C \ C^{(\N)} 2^{\N+3} C_{2}$,
	where $\tilde C$ is the constant from III,
	and define
	\begin{eqnarray} \label{2.5;6}
		\hat Z_{\N+1}^{\N+1} := \left\{ z \in Z_{\N+1}^{\N+1} \Big| 
		\left( \frac{d(z,P_{\N+1})}{t} \right)^{\p} \le 
		\hat C^{(\N+1)} \frac{\mathcal{M}_{\K^{\p};k_{1}}(x,t)}{t^{\N}} \right\}.
	\end{eqnarray}
	Next we show $\mu \left( \hat Z_{\N+1}^{\N+1} \right) \ge \frac{t^{\N}}{2^{\N+3}C_{2} } > 0 $.
	Let $u \in Z_{\N+1}^{\N+1} \setminus \hat Z_{\N+1}^{\N+1} \subset B_{\N+1} \subset B(x,(k+k_{0})t)$. 
	With III applied on $w=u$ and $j=\N+1$, we get
	\begin{align}\label{10.2.11.11}
		\left(\frac{d(u,P_{\N+1})}{t}\right)^{\p} 
			\le \tilde C t^{\N(\N+1)} \K^{\p}(z_0,\dots,z_{\N},u).
	\end{align}
	Now we get with \eqref{2.5;6}, Chebyshev's inequality and \eqref{10.2.11.11} that
	\begin{align*}
		\mu\left( Z_{\N+1}^{\N+1} \setminus \hat Z_{\N+1}^{\N+1} \right) 
		& \le  \left( \hat C^{(\N+1)} \frac{\mathcal{M}_{\K^{\p};k_{1}}(x,t)}{t^{\N}} \right) ^{-1} 
			\tilde C t^{\N(\N+1)}
			\int_{Z_{\N+1}^{\N+1} \setminus \hat Z_{\N+1}^{\N+1}}  \mathcal{K}^{\p}(z_{0},\dots,z_\N,u) \dd  \mu (u).
	\end{align*}
	By using I. we see that the integral on the RHS is equal to $\mathcal{R}(\{z_{0},\dots,z_{\N-1}\} \cup \{z_{\N}\})$.
	Hence with \eqref{28.08.12.1} and \eqref{23.08.12.1} we obtain
	\[ \mu(\hat Z_{\N+1}^{\N+1}) \ge \mu(Z_{\N+1}^{\N+1}) - \mu(Z_{\N+1}^{\N+1} \setminus \hat Z_{\N+1}^{\N+1}) > 0,\]
	and we are able to choose $ z_{\N+1} \in \hat Z_{\N+1}^{\N+1} \subset Z_{\N+1}^{\N+1}$.
	Let $l\in \{0,\dots, \N+1\}$ and $E=\{z_{0},\dots,z_{\N+1}\} \setminus \{z_{l}\}$.
	Set $z:=z_{\N}$ if $l=\N+1$ or $z:=z_{\N+1}$ otherwise. Now set $E^{'}:=E \setminus \{z\}$ and use \eqref{23.08.12.1}
	to obtain 
	$\mathcal{R}(E)=\mathcal{R}(E^{'}\cup \{z\}) \le C^{(\N+1)} \frac{\mathcal{M}_{\K^{\p};k_{1}}(x,t)}{t^{(\N+1)\N}}$
	
	All in all, there exists some constant $C=C(\n,\N,\K,\p,C_0,k,k_0,\lambda)$ such that 
	\begin{align*}
		\int \Eins_{\{ (z_0,\dots,\hat z_l, \dots,z_{\N+1},w) \in \mathcal O _{k_{1}}(x,t)\}} 
			\mathcal{K}^{\p}(z_0,\dots,\hat z_l, \dots,z_{\N+1},w) \dd  \mu(w)
		& = \mathcal{R}(E)
		\le C \frac{\mathcal{M}_{\K^{\p};k_{1}}(x,t)}{t^{(\N+1)\N}}
	\end{align*}
	for all $l \in \{0,\dots,\N+1 \}$. This ends the proof of IV.
	\end{proof}

	With IV, there exist some $z_{i} \in F \cap B_{i} \cap B(x,t)$, $i \in \{0,\dots,\N+1\}$
	fulfilling \eqref{24.08.12.1} and \eqref{24.08.12.2}.
	Let $ w \in \left(F \cap B\left(x,(k + k_{0})t \right) \right) \setminus \bigcup_{j=0}^{\N} 2 B_{j} $.
	Hence we get with III ($P_{\N+1}=\aff(z_0,\dots,z_\N)$), I and \eqref{24.08.12.1}
	\begin{align}\label{2.5;13}
		\int_{B(x,(k+k_{0})t) \setminus \bigcup_{j=0}^{\N}2B_j} 
			\left( \frac{d(w,P_{\N+1})}{t} \right)^{\p} \dd  \mu(w) 
		& < C(\n,\N,\K,\p,C_0,k,k_0,\lambda) \mathcal{M}_{\K^{\p};k_{1}}(x,t). 
	\end{align}

	Now we prove this estimate on the ball $2B_{j}$, where $j \in \{0,\dots,\N\}$.
	We define the plain $P_{j}:=\aff(\{z_0,\dots,z_{\N+1} \} \setminus \{z_j\})$ and get analogously with III, I
	and \eqref{24.08.12.1}
	\begin{align}\label{2.5;8}
		\int_{2B_{j}} \left( \frac{d(w,P_{j})}{t} \right)^{\p} \dd  \mu(w) 
		& < C(\n,\N,\K,\p,C_0,k,k_0,\lambda) \mathcal{M}_{\K^{\p};k_{1}}(x,t). 
	\end{align}
	Now we have an estimate on the ball $2B_{j}$ but with plane $P_{j}$ instead of $P_{\N+1}$. 
	If $z_{\N+1} \in P_{\N+1}$, we have $P_{\N+1}=P_{j}$ for all $j \in \{0,\dots,\N+1\}$
	and hence we get estimate \eqref{2.5;8} for $P_{\N+1}$.
	From now on, we assume that $z_{\N+1} \notin P_{\N+1}$. Let $w \in 2B_{j}$,
	set $w':=\pi_{P_j}(w)$, $w^{''}:=\pi_{P_{\N+1}}(w')$
	and deduce by inserting the point $w'$ with triangle inequality
	\begin{align}\label{2.5;9}
		d(w,P_{\N+1})^{\p} & \le d(w,w^{''})^{\p} 
		\le 2^{\p-1} \left (d(w,P_{j})^{\p} +  d(w',P_{\N+1})^{\p} \right). 
	\end{align}
	If $d(w',P_{\N+1}) >0$, i.e., $w' \notin P_{\N+1}$, we gain
	with Lemma \ref{20.2.2012.4}  ($P_{1}=P_{j}$, $P_{2}=P_{\N+1}$, $a_{1}=w'$, $a_{2}=z_{\N+1}$) 
	where $P_{j,\N+1}:=P_j \cap P_{\N+1}$
	\begin{align}\label{10.2.11.3}
		d(w',P_{\N+1}) 
			& = d(z_{\N+1},P_{\N+1})\frac{d(w',P_{j,\N+1})}{d(z_{\N+1},P_{j,\N+1})}. 
	\end{align}
	With $l \in \{0,\dots,\N \}$, $l\neq j$ ($k_1$ is defined on page \pageref{19.10.12.10}), we get 
	\begin{align*}
		d(w',P_{j,\N+1}) & \le d\bigl(w,P_{j,\N+1}\bigr) \le d(w,x)+d(x,x_l)+d(x_l,z_l) \le k_{1}t.
	\end{align*}
	With II. we get that
	$\face{j}(\Delta(z_{0},\dots,z_{\N +1}))$ is an $(\N,(9\N-1) \frac{t}{C_1})$-simplex and we obtain
	\begin{align}\label{2.5;10}
		\left( \frac{d(w',P_{\N+1})}{t} \right)^{\p} 
		& \stackrel{\eqref{10.2.11.3}}{\le} \left( \frac{d(z_{\N+1},P_{\N+1})}{t} \frac{k_{1}tC_{1}}{(9\N-1) t} \right)^{\p}
		\stackrel{\eqref{24.08.12.2}}{\le} C \frac{\mathcal{M}_{\K^{\p};k_{1}}(x,t)}{t^{\N}} 
	\end{align}
	where $C=C(\n,\N,\K,\p,C_0,k,k_0,\lambda)$.
	If $d(w',P_{\N+1})=0$, this inequality is trivially true.

	Finally, applying \eqref{2.5;8}, \eqref{2.5;8}, \eqref{2.5;10} and
	$\mu(2B_{j}) \stackrel{\text{(B)}}{\le} C_0 (\diam(2B_{j}))^\N \le C_0 \left( \frac{4t}{C_1} \right)^\N $
	((B) from page \pageref{GrundeigenschaftenBetaBeta}), we obtain
	\begin{align*}
		\int_{2B_{j}} \left( \frac{d(w,P_{\N+1})}{t} \right)^{\p} \dd  \mu (w)
		& \le C\left(\n,\N,\K,\p,C_0,k,k_0,\lambda \right) \mathcal{M}_{\K^{\p};k_{1}}(x,t).
	\end{align*}
	Given that $B(y,kt) \subset B(x,(k+k_{0})t)$, it follows with \eqref{2.5;13} that
	\begin{align*}
	\beta_{\p;k}(y,t)^{\p} 
	& \le \frac{1}{t^{\N}} \int_{B(x,(k+k_{0})t)} \left( \frac{d(w,P_{\N+1})}{t} \right)^{\p} \dd  \mu (w)
	 \le C(\n,\N,\K,\p,C_0,k,k_0,\lambda) \frac{\mathcal{M}_{\K^{\p};k_{1}}(x,t)}{t^{\N}}.
	\end{align*}
	To obtain the main result of this theorem, the only thing left to show is
	$ \mathcal{O}_{k_{1}}(x,t) \subset \mathcal{O}_{k_{1}+k_{0}}(y,t)$
	Let $(z_{0},\dots,z_{\N+1}) \in  \mathcal{O}_{k_{1}}(x,t)$. It follows that
	$z_{0},\dots,z_{\N+1} \in B(x,k_{1}t) \subset B(y,(k_{0}+k_{1})t)$ and
	$d(z_{i},z_{j}) \ge \frac{t}{k_{1}} \ge \frac{t}{k_{1}+k_{0}}$ with  $i\neq j$ and $ i,j=0,\dots,\N$.
	Thus $(z_{0},\dots,z_{\N+1}) \in  \mathcal{O}_ {k_{1}+k_{0}}(y,t)$.
\end{proof}

\begin{thm}\label{13.11.2014.4}
	Let $ 0 < \lambda < 2^{\N}$, $k > 2$, $k_{0} \ge 1$ and $\K^{\p}$ be some $\mu$-\proper{} symmetric integrand 
	(see Definition \ref{muproper}). There exists a constant
	$C=C(\n,\N,\K,\p,C_0,k,k_0,\lambda) $ such that
	\[\int  \int_{0}^{\infty} \beta_{\p;k}(x,t)^{\p}\Eins_{\left\{ \tilde{\delta}_{k_{0}}(B(x,t)) \ge \lambda \right\}} \frac{\dd t}{t} 
		 \dd \mu(x) \le C\mathcal{M}_{\K^{\p}}(\mu),\]
	where $\tilde{\delta}_{k_{0}}(B(x,t)):= \sup_{y \in B(x,k_{0}t)}\delta( B(y,t))$.
\end{thm}
\begin{proof}
	At first, we prove some intermediate results.\\
	I. \, 	Let $x \in F$, $t>0$ and $ \tilde{\delta}_{k_{0}}(B(x,t)) \ge \lambda$. 
		There exists some $z \in B(x,k_{0}t)$ with $\delta(B(z,t)) \ge \frac{\lambda}{2}$. Now with 
		Theorem \ref{lem2.5} there exist some constants $k_{1}$ and $C$
		so that with $k_{2}:=k_{1}+k_{0}$, we obtain
		$\beta_{\p;k}(x,t)^{\p} \le C \frac{\mathcal{M}_{\K^{\p};k_{2}}(x,t)}{t^{\N}}$.\\
	II.\,	Let $(x,t) \in \mathcal{D}_\kappa(u_0,\dots,u_{\N+1}):=\{(y,s)\in F \times (0,\infty)| (u_0,\dots,u_{\N+1})
			\in \mathcal{O}_\kappa(y,s)\}$ where $u_0,\dots,u_{\N+1} \in F$.
		We have $(u_0,\dots,u_{\N+1}) \in \mathcal{O}_\kappa(x,t)$ and so 
		$\frac{d(u_0,u_1)}{2\kappa} \le t \le \kappa d(u_0,u_1)$ as well as $x \in B(u_0,\kappa t)$.\\
	III.\,	With Fubini's theorem \cite[1.4, Thm. 1]{Evans} and condition (B) from page 
		\pageref{GrundeigenschaftenBetaBeta} we get
		\begin{align*}
			\int_{F} \int_{0}^{\infty} \chi_{\mathcal{D}_{k_2}(u_0, \dots,u_{\N+1})}(x,t) \frac{1}{t^{\N}} \frac{\dd t}{t} \dd \mu(x)
			& \stackrel{\hidewidth\text{II}\hidewidth}{\le} \int_{\frac{d(u_0,u_1)}{2k_{2}}}^{k_{2}d(u_0,u_1)} \frac{1}{t^{\N}} 
				\int_{B(u_0,k_{2}t)} 1 \ \dd \mu(x) \frac{\dd t}{t}
			\stackrel{\hidewidth\text{(B)} \hidewidth}{=}  C.
		\end{align*}
	Now we deduce with Fubini's theorem \cite[1.4, \mbox{Thm. 1}]{Evans} 
	\begin{align*}
		& \ \ \ \int_{F}  \int_{0}^{\infty} \beta_{\p;k}(x,t)^{\p}\Eins_{\{ \tilde{\delta}_{k_{0}}(B(x,t)) \ge \lambda \}} 
			\frac{\dd t}{t}  \dd \mu(x)\\
		& \stackrel{\text{I}}{\le} C  \int_{F} \int_{0}^{\infty}  \idotsint_{\mathcal{O}_{k_2}(x,t)}
			 \frac{\K^{\p}(u_0,\dots,u_{\N+1})}{t^{\N}} \dd \mu(u_0) \dots \dd \mu(u_{\N+1}) \frac{\dd t}{t} \dd \mu(x)
		 \stackrel{\text{III}}{\le} C \mathcal{M}_{\K^{\p}}(\mu).
	\end{align*}
\end{proof}

\begin{kor} \label{thm2.4}
	Let $ 0 < \lambda < 2^{\N}$, $k > 2$, $k_{0} \ge 1$ and $\K^{\p}$ be some symmetric 
	$\mu$-\proper{} integrand (see Definition \ref{muproper}). 
	There exists a constant
	$C=C(\n,\N,\K,\p,C_0,k,k_0,\lambda) $ such that
	\[\int  \int_{0}^{\infty} \beta_{1;k}(x,t)^{\p}\Eins_{\left\{ \tilde{\delta}_{k_{0}}(B(x,t)) \ge \lambda \right\}} \frac{\dd t}{t} 
		 \dd \mu(x) \le C\mathcal{M}_{\K^{\p}}(\mu).\]
\end{kor}
\begin{proof}
	This is a direct consequence of the previous Theorem and H\"older's inequality.
\end{proof}

\subsection{\texorpdfstring{$\beta$-numbers, approximating planes and angles}{\ss -numbers, approximating planes and angles}}
The following lemma states, that if two balls are close to each other and if 
each part of the  support of $\mu$ contained in those balls is well approximated by
some plane, then these planes have a small angle.
\begin{lem} \label{lem2.6}
	Let $ x,y \in F$, $c\ge1$, $\xi  \ge 1$ and $ t_{x},t_{y}>0$ with $c^{-1}t_{y} \le t_{x} \le c t_{y}$. 
	Furthermore, let $ k \ge 4c$ and $ 0 < \lambda < 2^{\N}$ with
	$ \delta(B(x,t_{x})) \ge \lambda$, $\delta(B(y,t_{y})) \ge \lambda$ and $d(x,y) \le \frac{k}{2c}t_{x}$.
	Then there exist some constants $C_{3}=C_{3}(\n,\N,C_0,\lambda,\xi,c) >1$ 
	and $\varepsilon_{0}=\varepsilon_{0}(\n,\N,C_{0},\lambda, \xi ,c) > 0$ so 
	that for all $\varepsilon < \varepsilon_{0}$ and all planes $P_{1}, P_{2} \in \mathcal{P}(\n,\N)$ with 
	$\beta_{1;k}^{P_{1}}(x,t_{x}) \le \xi\varepsilon$ and $\beta_{1;k}^{P_{2}}(y,t_{y}) \le \xi\varepsilon$
	we get:
	For all $w \in P_{1}$, we have $d(w,P_{2}) \le C_{3}\varepsilon(t_{x} + d(w,x))$, for all 
		$w \in P_{2}$, we have $d(w,P_{1}) \le C_{3} \varepsilon(t_{x} + d(w,x))$ and we have
	$\varangle(P_{1},P_{2}) \le C_{3} \varepsilon$.
\end{lem}
\begin{proof}
	Due to $ \delta(B(x,t_{x})) \ge \lambda$ and Corollary \ref{04.09.12.1}, there exist 
	some constants $C_{1}>3$ and $C_{2}$ depending on $\n,\N,C_0,\lambda$, and some
	simplex $T=\Delta(x_0,\dots,x_{\N}) \in F \cap B(x,t_{x})$  
	so that 
	$T$ is an $(\N,10\N \frac{t_x}{C_1})$-simplex and
	$\mu( B(x_i,\frac{t_x}{C_1}) \cap B(x,t_{x})) \ge \frac{t_x^{\N}}{C_{2}}$
	 for all  $i \in \{0,\dots,\N\}$.
	For $B_i:=B(x_i,\frac{t_x}{C_1})$ and $i \in \{0,\dots,\N \}$, we have 
	$ \mu(B_{i}) \ge  \mu(B_{i} \cap B(x,t_{x})) \ge \frac{t_{x}^{\N}}{C_{2}}  \ge \frac{t_{y}^{\N}}{c^{\N}C_{2}}$.
	Since $ B_{i} \cap B(x,t_{x})  \ne \emptyset$ and $k \ge 4c \ge 4$ we obtain
	$B_{i} \subset B(x,kt_{x})$ and $B_{i} \subset B(y,kt_{y})$.
	Now we see for $i \in \{0,\dots,\N\}$ 
	\begin{align*}
	\frac{1}{\mu(B_{i})} \int_{B_{i}} d(z,P_{1})+d(z,P_{2}) \mathrm d \mu(z) 
	&  = C_{2} t_{x} \beta_{1;k}^{P_{1}}(x,t_{x}) + c^{\N} C_{2} t_{y} \beta_{1;k}^{P_{2}}(y,t_{y})
	  \le  2c^{\N+1}C_{2} xi t_{x} \varepsilon .
	\end{align*}
	With Chebyshev's inequality, there exists $z_{i} \in B_{i}$ 
	so that
	\begin{align}\label{lem2.6;d} 
		d(z_{i},P_{j}) \le d(z_{i},P_{1}) + d(z_{i},P_{2}) \le 2c^{\N+1} C_{2} \xi t_{x} \varepsilon 
	\end{align}
	for $ i\in \{0,\dots,\N\}$ and $j = 1,2$. 
	We set $y_{i}:=\pi_{P_1}(z_i)$ and with 
	\[\varepsilon < \varepsilon_{0}:=\frac{1}{2c^{\N+1} C_{2} \xi}\min {\textstyle \left\{\frac{1}{C_{1}},
		\left(10(10^{\N}+1)\frac{C_1}{6} \left(2\frac{C_{1}}{3} \right)^{\N} \right)^{-1} \right\} }\]
	we deduce
	\begin{align*}
		d(y_i,x_i) & \le d(y_i,z_i) +d(z_i,x_i)
		 \le d(z_i,P_1) + { \textstyle \frac{t_x}{C_1} }
		 \le 2c^{\N+1} C_{2} \xi \  t_{x} \ \varepsilon  + { \textstyle \frac{t_x}{C_1} \le 2\frac{t_x}{C_1}},
	\end{align*}
	so, with Lemma \ref{17.11.11.2}, $S:=\Delta(y_0,\dots,y_\N)$ is an $(\N,6\N\frac{t_x}{C_1})$-simplex and
	$S \subset B(x,\frac{2t_{x}}{C_{1}}+t_{x})\subset B(x,2t_x)$.
	Furthermore, with \eqref{lem2.6;d} we have
	$ d(y_i,P_2) \le d(y_i,z_i) + d(z_i,P_2) \le 2c^{\N+1}C_{2} \xi t_{x} \varepsilon$. 
	Now, with Lemma \ref{21.11.11.2} ($C=\frac{C_1}{6\N}$, $\hat C = 2$, $t=t_{x}$,
		$\sigma = 2c^{\N+1} C_{2} \xi \varepsilon$, $m=\N$) we obtain
	\[\varangle(P_1,P_2) \le 4\N (10^\N + 1)2\frac{C_1}{6}\left(2\frac{C_1}{3}\right)^{\N} 
			2c^{\N+1} C_{2} \xi \varepsilon 
		= C(\n,\N,C_0,\lambda,\xi,c)\varepsilon.\]
	Moreover, we have
	$d(y_{0},\pi_{P_{2}}(z_{0})) \le d(z_{0},P_{1}) + d(z_{0},P_{2}) 
		\stackrel{\eqref{lem2.6;d}}{\le} 2c^{\N+1} C_{2} \xi t_{x} \varepsilon$,
	so finally, with Lemma \ref{23.03.2012.1} ($\sigma=C \varepsilon$, $t=t_{x}$, 
	$p_1=y_0$. $p_2=\pi_{P_2}(z_{0})$), we get for $w \in P_1$ that
	$ d(w,P_2) \le C(d(w,y_0)+t_{x}) \varepsilon \le C(d(w,x) + t_{x}) \varepsilon$
	and for $w \in P_2$ we obtain
	$d(w,P_1) \le C (d(w,\pi_{P_2}(z_0))+t_{x}) \le C(d(w,x) + t_{x}) \varepsilon$, 
	where $C=C(\n,\N,C_0,\lambda,\xi,c)$.
\end{proof}

The next lemma describes the distance from a plane to a ball if the plain approximates the support of $\mu$ contained
in the ball.
\begin{lem} \label{nachlem2.6}
	Let $\sigma>0$, $x \in \R^{\n}$, $t >0$ and $\lambda >0$ with $ \delta (B(x,t)) \ge \lambda$. 
	If $P \in \mathcal{P}(\n,\N)$ with
	$\beta_{1;k}^{P}(x,t) \le \sigma$, there exists some $ y \in B(x,t) \cap F$ so that 
	$d(y,P) \le \frac{t}{\lambda} \sigma$.
	If additionally $\sigma \le \lambda$, we have $B(x,2t) \cap P \neq \emptyset$.
\end{lem}
\begin{proof}
	With the requirements, we get $ \mu(B(x,t)) \ge t^{\N} \lambda$, and so
	\begin{align*}
		\frac{1}{\mu(B(x,t))} \int_{B(x,t)} d(z,P) \dd \mu(z)
		& \le \frac{t}{\lambda} \frac{1}{t^{\N}}\int_{B(x,kt)} \frac{d(z,P)}{t} \dd \mu(z)
		 = \frac{t}{\lambda} \beta_{1;k}^{P}(x,t)
		 \le \frac{t}{\lambda} \sigma.
	\end{align*}
	With Chebyshev's inequality, we get some $y \in B(x,t) \cap F$ with $ d(y,P) \le  \frac{t}{\lambda} \sigma$.
	If $ \sigma \le \lambda$, it follows that $B(x,2t) \cap P \neq \emptyset$.
\end{proof}

\setcounter{equation}{0}
\section{Proof of the main result}\label{27.10.2014.2}
At the end of this section (page \pageref{BeweisdesHptsatzes}), we will give a proof of our 
main result Theorem \ref{maintheorem} under the assumption that
the forthcoming Theorem \ref{16.10.2013.1} is correct. We start with a few lemmas helpful for this proof.

\subsection{Reduction to a symmetric integrand}
\begin{lem} \label{6.5.2013.1}
	Let $\K^{\p}$ be some \proper{} integrand (see Definition \ref{4.10.12.1}). There exists some \proper{} integrand
	$\tilde{\K}^{\p}$, which is symmetric in all components and fulfils 
	$\M_{\K^{p}}(E)=\M_{\tilde{\K}^{\p}}(E)$ for all Borel sets $E$.
\end{lem}
\begin{proof}
	We set
	$\tilde{\K}^{\p}(x_{0},\dots,x_{\N+1}) 
		:= \frac{1}{\# S_{\N+2}} \sum_{\phi \in S_{\N+2}} \K^p(\phi(x_{0},\dots,x_{\N+1}))$,
	where $S_{\N+2}$ is the symmetric group of all permutations of $\N+2$ symbols.
	Due to
	$\K^{\p} \le \# S_{\N+2} \ \tilde{\K}^{\p}$, the integrand
	$\tilde{\K}^{\p}$ fulfils the conditions of a \proper{} integrand.
	Now Fubini's theorem \cite[1.4, Thm. 1]{Evans} implies
	$\M_{\tilde{\K}^{\p}}(E) = \M_{\K^{\p}}(E)$.
\end{proof}

\subsection{Reduction to finite, compact and more regular sets with small curvature}

\begin{lem} \label{19.04.2013.1}
	Let $E$ be a Borel set with $\M_{\K^{\p}}(E)< \infty$, where $\K^{\p}$ is some \proper{} integrand. 
	Then we have $\cH^{\N}(E \cap B) < \infty$ for every ball $B$.
\end{lem}
\begin{proof}
	Let $B$ be some ball and set $F:=E \cap B$. We prove the contraposition so we assume that $\cH^{\N}(F)= \infty$.
	With Lemma \ref{16.04.2013.1}, there exists some constant $C>0$ and some $(\N+1,(\N+3)C)$-simplex 
	$T=\Delta(x_{0},\dots,x_{\N+1}) \in B$ with $\cH^{\N}(B(x_{0},C)\cap F)=\infty$ and
	$\cH^{\N}(B(x_{i},C)\cap F)> 0$ for all $i \in \{1,\dots,\N+1 \}$. With Lemma \ref{17.11.11.2},
	we conclude that $S=\Delta(y_{0},\dots,y_{\N+1})$ is an $(\N+1,C)$-simplex for all
	$y_{i} \in B(x_{i},C)$, $i\in \{0,\dots,\N+1\}$.
	For $t=C\sqrt{\frac{\diam B}{2C}+1}$ and $\bar C =\sqrt{\frac{\diam B}{2C}+1}$, 
	we get $S \in B(x,t \bar C)$, where $x$ is the centre of the ball $B$, 	
	and $S$ is an $(\N+1,\frac{t}{\bar C})$-simplex.
	Hence we are in the right setting for using the second condition of a \proper{} integrand.
	We obtain 
	\[\M_{\K^{\p}}(E)  \ge \int_{B(x_{\N+1},C)\cap F} \dots \int_{B(x_{0},C)\cap F} \K^{\p}(y_{0},\dots,y_{\N+1}) 
		\dd \cH^{\N}(y_{0}) \dots \dd \cH^{\N}(y_{\N+1}) = \infty.\]
\end{proof}

\begin{lem} \label{satz1.1}
	In this lemma, the integrand $\mathcal{K}$ of $\M_{\K^{\p}}$ only needs to be an $(\cH^{\N})^{\N+2}$-integrable 
	function. Let $p>0$, $\N < \n$  and
	$E \subset \mathbb{R}^{\n}$ be a Borel set with $0 < \cH^{\N}(E) < \infty$ and 
	$\mathcal{M}_{\K^{\p}}(E) < \infty$. For all $ \zeta > 0$, there exists some compact $E^{*} \subset E$ with
	\begin{enumerate}
	\renewcommand{\labelenumi}{\textup{(\roman{enumi})}} 
	\item $ \cH^{\N}(E^{*}) > \frac{(\diam E^{*})^\N \vol}{2^{2\N+1}}$,
	\item $ \forall x \in E^{*}, \forall t > 0, \ \cH^{\N}(E^{*} \cap B(x,t)) \le 2\vol t^{\N}$,
	\item $\mathcal{M}_{\K^{\p}}(E^{*}) \le \zeta \ (\diam E^{*})^\N $,
	\end{enumerate}
	where $\vol=\cH^{\N}(B(0,1))$ is the $\N$-dimensional volume of the $\N$-dimensional unit ball.
\end{lem}
\begin{proof}
	Due to $ 0 < \cH^{\N}(E) < \infty$ and \cite[2.3, Thm. 2]{Evans}, for $\cH^{\N}$-almost all $x \in E $ we have
	\begin{equation} \label{evans}
		\frac{1}{2^{\N}} \le \limsup_{t \rightarrow 0^{+}} \frac{\cH^{\N}(E \cap B(x,t))}{\vol t^{\N}} \le 1.
	\end{equation}
	For $ l \in \mathbb{N}$, we define the $\cH^{\N}$-measurable set
	\begin{align} \label{21.2.11.1}
		E_{m} &:= \left\{ x \in E \ \Big| \ \forall t \in \left(0,\frac{1}{m} \right),  
		\cH^{\N}(E \cap B(x,t)) \le 2\vol t^{\N} \right\}.
	\end{align}
	Due to $E_{l} \subset E_{l+1}$, \cite[1.1.1, Thm. 1, (iii)]{Evans} and \eqref{evans} we get that
	\[\lim_{l \rightarrow \infty}  \cH^{\N}(E_{l})
		=\cH^{\N}\left({\textstyle \bigcup_{l=1}^{\infty}} E_{l}\right)=\cH^{\N}(E)\]
	Hence there exists some $m \in \mathbb{N}$ with $ \cH^{\N}(E_{m}) \ge \frac{1}{2} \cH^{\N}(E)$ and
	$\mathcal{M}_{\K^{\p}}(E_{m}) \le \mathcal{M}_{\K^{\p}}(E) < \infty $.
	Define for $ \tau > 0$
	\begin{align} \label{21.2.11.2} 
		\mathcal{I}(\tau) := \int_{A(\tau)} \K^{\p}(x_0, \dots, x_{\N+1}) 
		\dd \cH^{\N}(x_0) \dots \dd \cH^{\N}(x_{\N+1}),
	\end{align}
	where $A(\tau) := \left\{ (x_0, \dots, x_{\N+1}) \in E_{m}^{\N+2} \Big| 
			d(x_0,x_i) < \tau \text{ for all } i \in \{ 1,\dots, \N+1 \} \right\}$.
	Using \eqref{21.2.11.1} we obtain $\left(\cH^{\N}\right)^{\N+2}(A(\tau)) \rightarrow  0$ for $\tau \rightarrow 0$.
	With $ \mathcal{M}_{\K^{\p}}(E_{m}) < \infty $, we conclude 
	$\lim_{\tau \rightarrow 0} \mathcal{I}(\tau) = 0$,
	and so we are able to pick some $ 0 < \tau_{0} \le \frac{1}{2m}$ with
	\begin{align} \label{21.2.11.3}
		\mathcal{I}(2\tau_{0}) \le \frac{\zeta \cH^{\N}(E_{m})}{2\vol \cdot 2^{\N+3}}.
	\end{align}
	We set
	\[ \mathcal{V} := \left\{ B(x,\tau) \Big| x \in E_{m}, 0 < \tau <\tau_{0}, \cH^{\N}(E_{m} \cap B(x,\tau))
		 \ge \frac{\tau^\N \vol}{2^{\N+1}} \right\}. \]
	Since $0 < \cH^{\N}(E_{m}) < \infty$, we get (\ref{evans}) with $E_m$ instead of $E$, \cite[2.3, Thm. 2]{Evans}.
	This implies 
	$\inf \left\{ \tau \big| B(x,\tau) \in \mathcal{V} \right\} = 0$ for $\cH^{\N}$-almost every $x \in E_{m}$. 
	According to \cite[1.3]{Falconer}, $\mathcal{V}$ is a Vitali class.
	For every countable, disjoint subfamily $\{B_{i}\}_{i}$ of $\mathcal{V}$, we have
	$\sum_{i \in \mathbb{N}} (\diam B_{i})^{\N} \le \frac{2^{2\N+1}}{\vol} \cH^{\N}(E_{m})< \infty$.
	Applying Vitali's Covering Theorem \cite[1.3, Thm. 1.10]{Falconer}, we get a countable subfamily of
	$\mathcal{V}$ with disjoint balls $B_{i} = B(x_{i},\tau_{i})$ fulfilling
	$\cH^{\N}\left(E_{m} \setminus  \bigcup_{i \in \mathbb{N}} B_{i}\right) = 0$.
	Therefore, using \eqref{21.2.11.1}, we have 
	$\cH^{\N}(E_{m}) \stackrel{\hphantom{\eqref{21.2.11.1}}}{\le} \sum_{i\in \mathbb{N}} \cH^{\N}(E_{m} \cap B_{i}) 
		\le \sum_{i\in \mathbb{N}} 2\vol \tau_{i}^\N$,
	so that
	\begin{equation} \label{summetaui}
		\sum_{i \in \mathbb{N}} \tau_{i}^\N \ge \frac{\cH^{\N}(E_{m})}{2\vol}.
	\end{equation}
	Furthermore, with $ (B_{i} \cap E_{m})^{\N+2} \subset A(2 \tau_{0}) \cap B_{i}^{\N+2} $, we obtain
	\begin{align}
		\sum_{i \in \mathbb{N}} \mathcal{M}_{\K^{\p}}(B_{i} \cap E_{m}) 
		& \stackrel{\eqref{21.2.11.2}}{\le} \mathcal{I}(2\tau_{0}) 
		\stackrel{\eqref{21.2.11.3}}{\le} \frac{\zeta \cH^{\N}(E_{m})}{2\vol \cdot 2^{\N+3}}. \label{21.2.11.4}
	\end{align}
	We define
	\[ I_{b} := \left\{ i \in \mathbb{N} \Big| \mathcal{M}_{\K^{\p}}(B(x_{i},\tau_{i}) \cap  E_{m}) \ge 
		\zeta {\textstyle \frac{ \tau_{i}^\N}{2^{\N+2}}} \right\}\]
	and so
	\[ \sum _{i \in I_{b}} \mathcal{M}_{\K^{\p}}(B(x_{i},\tau_{i})\cap  E_{m}) 
		\ge \zeta\frac{ \sum_{i \in I_{b}} \tau_{i}^{\N}}{2^{\N+2}}.\]
	We have $ \sum_{i \in I_{b}} \tau_{i}^\N \le \frac{\cH^{\N}(E_{m})}{4\vol}$
	since assuming the converse  would imply
	\begin{align*}
		\sum_{i \in \mathbb{N}} \mathcal{M}_{\K^{\p}}(B(x_{i},\tau_{i}) \cap  E_{m}) 
		& \stackrel{\eqref{21.2.11.4}}{<} \zeta\frac{ \sum_{i \in I_{b}} \tau_{i}^\N}{2^{\N+2}} 
		\stackrel{\hphantom{\eqref{21.2.11.4}}}{\le} \sum_{i \in I_{b}} \mathcal{M}_{\K^{\p}}(B(x_{i},\tau_{i})\cap  E_{m}).
	\end{align*}
	Using \eqref{summetaui}, we obtain $ I_b \neq \mathbb{N} $.
	Now we choose some $i \in \mathbb{N} \setminus I_b$ and the regularity of the Hausdorff measure 
	\cite[1.2, Thm. 1.6]{Falconer} implies the existence of some compact set $E^{*} \subset B(x_{i},\tau_{i}) \cap E_{m}$
	with
	\begin{enumerate}
		\renewcommand{\labelenumi}{\textup{(\roman{enumi})}} 
		\item $\cH^{\N}(E^{*}) > \frac{1}{2}\cH^{\N}(B(x_{i},\tau_{i})\cap E_{m})
			\ge \frac{\tau_{i}^{\N}\vol}{2^{\N+1}} \ge \frac{(\diam E^{*})^{\N} \vol}{2^{2\N+1}}$
		\item $ \forall x \in E^{*}, \forall t > 0$, we have 
			$\cH^{\N}(E^{*} \cap B(x,t)) \le \cH^{\N}(B(x_{i},\tau_{i}) \cap E_{m} \cap B(x,t)) \le 2\vol t^{\N}$ 
			since if $t < \frac{1}{m}$ \eqref{21.2.11.1} implies $\cH^{\N}(E \cap B(x,t))\le 2\vol t^{\N}$
			and if $\tau_{i}<\frac{1}{m}<t$ \eqref{21.2.11.1} implies 
			$\cH^{\N}(B(x_{i},\tau_{i}) \cap E_{m}) \le 2\vol t^{\N}$.
		\item $\mathcal{M}_{\K^{\p}}(E^{*}) \le \zeta \frac{\tau_{i}^{\N}}{2^{\N+2}} \le \zeta (\diam E^{*})^{\N}$ 
			since $i \notin I_{b}$ and for some ball $B$ with $E^{*} \subset B$ and $ \diam B=2\diam E^{*}$ we have
			$\frac{\tau_{i}^{\N}}{2^{\N+2}} \stackrel{\text{(i)}}{\le} \frac{\cH^{\N}(E^{*}\cap B)}{2\vol}
			\stackrel{\text{(ii)}}{\le} (\diam E^{*})^{\N}$.
	\end{enumerate}
\end{proof}

Next, we present the crucial theorem of this work. 

\begin{thm} \label{16.10.2013.1}
	Let $\K : \left(\R^{\n}\right)^{\N+2} \to [0,\infty)$.
	There exists some $k > 2$ such that for every
	$C_{0} \ge 10$, there exists some 
	$ \eta=\eta(\n,\N,\K,C_{0},k) \in (0,\vol2^{-(2\N+2)}]$
	so that if 
	$\mu$ is a Borel measure on $\mathbb{R}^{\n}$ with
	compact support $F$ such that $\K^{2}$ is a symmetric $\mu$-\proper{} integrand (cf. Definition \ref{muproper})
	and $\mu$ fulfils 
	\begin{enumerate}
		\renewcommand{\labelenumi}{\textup{(\Alph{enumi})}}
		\item $\mu(B(0,5)) \ge 1, \ \mu(\mathbb{R}^{\n} \setminus B(0,5)) = 0$,
		\item $\mu(B) \le C_{0} \left(\diam B\right)^{\N}$ for every ball $B$,
		\item $\mathcal{M}_{\K^{2}}(\mu) \le \eta$,
		\item $\beta_{1;k;\mu}^{P_{0}}(0,5) \le \eta$ for some plane $P_{0} \in \mathcal{P}(\n,\N)$
			with $0 \in P_{0}$,
	\end{enumerate}
	then there exists some Lipschitz function $A: P_{0} \to P_{0}^{\perp} \subset \R^{\n}$ so that
	the graph $G(A) \subset \R^{\n}$ fulfils
	$ \mu(G(A)) \ge {\textstyle \frac{99}{100}} \mu(\mathbb{R}^{\n})$.
	($P_0^{\perp}:=\{x \in \R^{\n}|x\cdot v=0 \text{ for all } v \in P_{0}\}$ denotes the 
	orthogonal complement of $P_{0}$.)
\end{thm}

At first, we show that, under the assumption that the previous theorem is correct,
we can prove Theorem \ref{maintheorem}. The remaining proof of Theorem \ref{16.10.2013.1} is then given by 
the following chapters \ref{construction}, \ref{gamma} and \ref{notanullset}.
We will use the notation $sE:=\{x \in \R^{\n}|s^{-1}x \in E\}$ for $s > 0 $ and some set $E \subset \R^{\n}$.
Distinguish this notation from $sB(x,t)=B(x,st)$, where the centre stays unaffected and only the radius is scaled.

\begin{proof}[Proof of Theorem \textup{\ref{maintheorem}}] \label{16.10.2013.2} \label{BeweisdesHptsatzes}
	Let $\K^{2}$ be some \proper{} integrand (see Definition \ref{muproper}), $E \subset \R^{\n}$ 
	some Borel set with $\M_{\K^{2}}(E) < \infty$ and let $C_{0}= 2^{2\N+2}$.
	Furthermore, let $k>2$ and $0< \eta \le \vol 2^{-(2\N+2)}$ be the constants given by Theorem \ref{16.10.2013.1}.
	Using Lemma \ref{6.5.2013.1}, we can assume that $\K$ is symmetric.

	We start with a countable covering of $\R^{\n}$ with balls $B_{i}$ so that 
	$\R^{\n} \subset \bigcup_{i \in \mathbb{N}} B_{i}$. We will show that for all $i\in \mathbb{N}$ 
	the sets $E \cap B_{i}$ are $\N$-rectifiable, which
	implicates that $E$ is $\N$-rectifiable.

	Let $i \in \mathbb{N}$ with $\cH^{\N}(E \cap B_{i})>0$. With Lemma \ref{19.04.2013.1}, we conclude 
	that $\cH^{\N}(E \cap B_{i})< \infty$.
	Then, using \cite[Thm. 3.3.13]{Federer}, we can decompose 
	$E \cap B_{i}=E_{\text{r}}^{i} \ \dot \cup \ E_{\text{u}}^{i}$ into two disjoint subsets,
	where $E_{\text{r}}^{i}$ is $\N$-rectifiable and $E_{\text{u}}^{i}$ is purely $\N$-unrectifiable.
	
	Now we assume that $E \cap B_{i}$ is not $\N$-rectifiable, so $ \cH^{\N}(E_{\text{u}}^{i}) > 0$. 
	The set $E_{\text{u}}^{i}$ is a Borel set and
	fulfils $0 < \cH^{\N}(E_{\text{u}}^{i}) \le \cH^{\n}(E \cap B_{i}) < \infty$ and 
	$\M_{\K^{2}}(E_{\text{u}}^{i}) \le \M_{\K^{2}}(E) < \infty$. 
	Now we apply Lemma \ref{satz1.1} with $\zeta = \eta \frac{1}{\hat C \tilde C}$ where the constants $\hat C$ and
	$\tilde C$ are given in this passage
	and get some compact set $E^{*}\subset E_{\text{u}}^{i}$
	which fulfils condition (i),(ii) and (iii) from Lemma \ref{satz1.1}.
	We set $a:= (\diam E^{*})^{-1}$ and $\tilde \mu = \cH^{\N} \textsf{ L } aE^{*}$.
	Let $\tilde B$ be a ball with $aE^{*}\subset \tilde B$ and $\diam \tilde B = 2$.
	Using (i), we get $\delta_{\tilde \mu}(\tilde B) \ge \frac{\vol}{2^{2\N+1}}$. 
	So, Theorem \ref{lem2.5} ($p=2$, $x=y\mathrel{\hat=}\text{centre of } \tilde B$, $t=1$, 
	$\lambda = \frac{\vol}{2^{3\N+1}}$, $k_{0}=1$)  implies 
	$\beta_{2;k;\tilde \mu}(\tilde B)^{2}
		< \hat C \mathcal{M}_{\K^{2}}(\tilde \mu) \le \eta^{2},$
	for some constant $\hat C= \hat C(\n,\N,\K,C_{0},k) \ge 1$.
	Using H\"older's inequality there exists some $\N$-dimensional plane 
	$\tilde{P_{0}} \in \mathcal{P}(\n,\N) \index{Subsets of $\mathbb{R}^{\n}$ ! $\tilde{P_{0}}$}$ with
	$ \beta_{1;k;\tilde \mu}^{\tilde{P_{0}}}(\tilde B) \le \eta$.
	Now we define a measure $\mu$ by $\mu(\cdot):= \frac{2^{2\N +1}}{\vol}\tilde \mu (\ \cdot \ + \pi_{\tilde P_{0}}(b))$,
	where $b$ is the centre of $\tilde B$.
	This is also a Borel measure with compact support and Lemma \ref{nachlem2.6} 
	($\sigma = \eta$, $B(x,t)=\tilde B$, $\lambda=\frac{\vol}{2^{2\N+1}}$) 
	implies that the support fulfils 
	$F:=aE^{*} -\pi_{\tilde P_{0}}(b) \subset B(0,2)$.
	This measure fulfils condition (D) from Theorem \ref{16.10.2013.1} ($P_{0}=\tilde P_{0}-\pi_{\tilde P_{0}}(b)$)
	and (i) implies condition (A). To get condition (B) for some arbitrary ball,
	cover it by some ball with centre on F, double diameter and apply (ii). Use 
	$\mathcal{M}_{\K^{2}}(\mu) = \tilde C(\N) a^{\N} \ \mathcal{M}_{\K^{2}}(E^{*})$ and (iii) to obtain (C).
	Finally we mention that $\K^{2}$ is $\mu$-proper, since $\mu$ is an adapted version of $\cH^{\N}$.
	Hence we can apply Theorem \ref{16.10.2013.1} and after some scaling and translation we 
	obtain some Lipschitz function which covers a part of positive Hausdorff measure of $E_{u}^{i}$
	which is in contrast to $E_{u}^{i}$ being purely $\N$-unrectifiable. Hence $E \cap B_{i}$ is $\N$-rectifiable.
\end{proof}

\setcounter{equation}{0}
\section{Construction of the Lipschitz graph} \label{construction}
\subsection{\texorpdfstring{Partition of the support of the measure $\mu$}{Partition of the support of the measure u}}\label{04.02.2014.1}
Now we start with the proof of Theorem \ref{16.10.2013.1}.
Let $\K : \left(\R^{\n}\right)^{\N+2} \to [0,\infty)$ 
and let $C_{0} \ge 10$ be some fixed constant.
There is one step in the proof which only works for integrability exponent 
$\p=2$. ($\p=2$ is used in Lemma \ref{25.09.2014.1} so that the results of Theorem \ref{2.9.2014.1} 
and Theorem \ref{thm4.1} fit together.)
Since most of the proof can be given with less constraints to $\p$,
we start with $\p \in (1,\infty)$ and restrict to $\p=2$ only if needed.
Furthermore, let $k>2$, \mbox{$0<\eta \le \vol2^{-(2\N+2)}$}, 
$P_{0} \in \mathcal{P}(\n,\N)$ \index{Subsets of $\mathbb{R}^{\n}$ ! $P_{0}$}
with $0 \in P_{0}$ and $\mu$ be a Borel measure on $\mathbb{R}^{\n}$ with
compact support $F$ \index{Subsets of $\mathbb{R}^{\n}$ ! $F$} \label{Grundeigenschaften} 
such that $\K^{\p}$ is a symmetric $\mu$-\proper{} integrand (cf. Definition \ref{muproper}) and \label{3.12.2013.3}
\begin{enumerate}
	\renewcommand{\labelenumi}{\textup{(\Alph{enumi})}}
	\item $\mu(B(0,5)) \ge 1, \ \mu(\mathbb{R}^{\n} \setminus B(0,5)) = 0$,
	\item $\mu(B) \le C_{0} \left(\diam B\right)^{\N}$ for every ball $B$,
	\item $\mathcal{M}_{\K^{\p}}(\mu) \le \eta$,
	\item $\beta_{1;k;\mu}^{P_{0}}(0,5) \le \eta$.
\end{enumerate}

In this chapter, we will prove that if $k$ is large and $\eta$ is small enough, 
we can construct some function $A: P_{0} \to P_{0}^{\perp}$ which covers some part of the 
support $F$ of $\mu$.
For this purpose, we will
give a partition of the support of $\mu$ in four parts,
$\supp(\mu)=\mathcal{Z} \dot{\cup} F_{1} \dot{\cup} F_{2} \dot{\cup} F_{3}$, and
construct the function $A$ so that the graph of $A$ covers $\mathcal{Z}$, i.e., $\mathcal{Z} \subset G(A)$.

The following chapters \ref{gamma} and \ref{notanullset} will give a proof of
$\mu(F_{1} \cup F_{2} \cup F_{3}) \le {\textstyle\frac{1}{100}}$,
hence with (A) we will obtain $\mu(G(A)) \ge {\textstyle\frac{99}{100}} \mu(\R^{\n})$,
which is the statement of Theorem \ref{16.10.2013.1}.

From now on, we will only work with the fixed measure $\mu$, so we can simplify the expressions by
setting $\beta_{1;k}:=\beta_{1;k;\mu}$ and $ \delta(\cdot):=\delta_{\mu}(\cdot)$.
Furthermore, we fix the constant \label{WahlderConstants}
\begin{align}
	\delta &:= \min\left\{\frac{10^{-10}}{600^{\N}N_{0}},\frac{2}{50^{\N}}\right\}, 
		\index{Constants ! $\delta$} \label{Wahlvondelta}
\end{align}
where $N_{0}=N_0(\n)$ \index{Constants ! $N_{0}$} 
is the constant from Besicovitch's Covering Theorem \cite[1.5.2, Thm. 2]{Evans}.

\begin{dfn}\label{12.07.13.1} \label{Definitionvonh} 
Let $\alpha, \varepsilon > 0$.
We define the set \index{$S_{total}$} 
\[ S_{total}^{\varepsilon,\alpha} := \left\lbrace (x,t) \in F \times (0,50)  \ \
	\begin{array}{|ll}
	  (i) & \delta(B(x,t)) \ge \frac{1}{2}\delta \\
	  (ii) & \beta_{1;k}(x,t) < 2\varepsilon \\
	  (iii) & \exists \ P_{(x,t)}\in \mathcal{P}(\n,\N) \index{Subsets of $\mathbb{R}^{\n}$ ! $P_{(x,t)}$} \ \  \text{s.t.} \ 
	 	\begin{cases}
  			\beta_{1;k}^{P_{(x,t)}}(x,t) \le 2\varepsilon  \\
  			 \ \ \text{and} \\
			\varangle(P_{(x,t)},P_{0}) \le \alpha
		\end{cases} \\
\end{array} 
\right\rbrace. \]
Having in mind that the definition of $S_{total}^{\varepsilon,\alpha}$ depends on the choice of $\varepsilon$
and $\alpha$, we will normally skip these and write $S_{total}$ instead.
In the same manner, we will handle the following definitions of $H,h$ and $S$.
For $x \in F$ we define 
\[H(x):=\Big\lbrace t \in (0,50) \ \Big| \ \exists \ y \in F, \ \exists \ \tau \text{ with } \\ \frac{t}{4} 
	 \le \tau \le \frac{t}{3}, \ d(x,y) < \frac{\tau}{3}  \ \text{and} \  (y,\tau) \notin S_{total}  \Big\rbrace,\]
\[ h(x) := \sup (H(x) \cup \{0\}) \quad \quad \text{ and } \quad \quad S := \left\lbrace (x,t) \in S_{total} \ | \ t \ge h(x) \right\rbrace.\]
Sometimes, we identify a ball $B=B(x,t)$ with the tuple $(x,t)$ and write to simplify matters $B \in S$ instead of
$(x,t) \in S$. In the same manner we use the notation $\beta_{1;k}(B)$.
\end{dfn}

\begin{lem} \label{rem3.1}
	Let $\alpha, \varepsilon > 0$.
	If $\eta \le 2\varepsilon$, we have that $S_{total}\neq \emptyset$ and
	\begin{enumerate}
	\renewcommand{\labelenumi}{(\roman{enumi})} 
	\item $F \times [40,50) \subset \{(x,t) \in F \times (0,50)|t\ge h(x)\}=S$,
	\item If $ (x,t) \in S$ and $t \le t^{'} < 50$, we have $(x,t^{'}) \in S$.
	\end{enumerate}
\end{lem}
\begin{proof}
	(i)\, If $x \in F \subset B(0,5)$ and $10\le t < 50$, we have
	$F \subset B(x,t)$. Using (A),(D) and $P_{(x,t)}:=P_0$ we get $(x,t) \in S_{total}$, which implies that
	$F \times [10,50) \subset S_{total}$. Now if $x \in F$ and $t \in [40,50)$ we deduce for arbitrary 
	$y\in F$ and $\tau \in [\frac{t}{4},\frac{t}{3}]$ that $(y,\tau)\in S_{total}$, 
	which implies that $H(x) \subset (0,40)$, $h(x) \le 40$ and hence the first inclusion.
	For the equality it is enough to prove that the central set is contained in $S$.
	Let $ x \in F $ and $t \in (0,50)$ with $h(x) \le t < 50$.
	Assume that $(x,t) \notin S$. Due to $h(x)\le t $, we obtain $(x,t) \notin S_{total}$,
	which implies that $t < 10$.
	Hence with $y=x$ and $\tau=t$ we get $3t \in H(x)$.
	This implies $h(x) \ge 3t > t$ and hence a contradiction to $t \ge h(x)$. So, we obtain $(x,t) \in S$. \\
	(ii)\, We have $x \in F$ and $h(x) \le t \le t^{'} < 50 $ so with (i) we conclude that $(x,t^{'}) \in S$ .
\end{proof}

Remember that the function $h$ depends on the set $S_{total}$, which depends on 
the choice of $\varepsilon$ and $\alpha$. Hence the sets defined in the following definition 
depend on $\alpha$ and $\varepsilon$ as well.
\begin{dfn}[Partition of $F$] \label{def3.2} \index{Partition of $F$! $\mathcal{Z}$}
	\index{Partition of $F$! $F_{1}$} \index{Partition of $F$! $F_{2}$} \index{Partition of $F$! $F_{3}$}
	Let $\alpha, \varepsilon > 0$. We define
	\[ \mathcal{Z} := \left\lbrace x \in F \ | \ h(x) = 0 \right\rbrace, \]
	\[ F_{1} := \left\lbrace x \in F \setminus \mathcal{Z} \ 
		\begin{array}{|ll}
		  & \exists y \in F, \exists \tau \in \left[\frac{h(x)}{5},\frac{h(x)}{2}\right],
		  	\text{ with } d(x,y) \le \frac{\tau}{2} \\
		  & \ \text{and} \\
		  & \delta(B(y,\tau)) \le \delta \ 
		\end{array} 
		\right\rbrace, \]
	\[ F_{2} := \left\lbrace x \in F \setminus (\mathcal{Z} \cup F_{1}) \ 
		\begin{array}{|ll}
		  & \exists y \in F, \exists \tau \in \left[\frac{h(x)}{5},\frac{h(x)}{2}\right],
		  	\text{ with } d(x,y) \le \frac{\tau}{2}  \\
		  & \ \text{and} \\
		  & \beta_{1;k}(y,\tau) \ge \varepsilon \ 
		\end{array} 
		\right\rbrace, \]
	\[ F_{3} := \left\lbrace x \in F \setminus (\mathcal{Z} \cup F_{1} \cup F_{2}) \ 
		\begin{array}{|ll}
		  & \exists y \in F, \exists \tau \in \left[\frac{h(x)}{5},\frac{h(x)}{2}\right],
		  	\text{ with } d(x,y) \le \frac{\tau}{2}  \\
		  & \text{and for all planes } P \in \mathcal{P}(\n,\N)\text{ with}   \\
		  & \beta_{1;k}^{P}(y,\tau) \le \varepsilon \text{ we have } \varangle(P,P_{0}) \ge \frac{3}{4}\alpha \ 
		\end{array} 
		\right\rbrace. \]
\end{dfn}
In this chapter, we prove that $\mathcal{Z}$ is rectifiable by constructing a function $A$ such that the
graph of $A$ will cover $\mathcal{Z}$. This is done by inverting the orthogonal projection 
$\pi|_{\mathcal{Z}} : \mathcal{Z} \to P_{0}$.
After that, to complete the proof, it remains to show that $\mathcal{Z}$ constitutes the major part of $F$.
Right now, we can prove that $\mu(F_{2}) \le 10^{-6}$ (cf. section \ref{F2issmall}, $F_2$ is small) where
the control of the other sets need some more preparations.

\begin{lem}
	Let $\alpha, \varepsilon > 0$.
	Definition \ref{def3.2} gives a partition of $F$, i.e.
	$F = \mathcal{Z} \ \dot \cup \ F_{1} \ \dot \cup \ F_{2} \ \dot \cup \ F_{3}$.
\end{lem}
\begin{proof}
	From the definition we see that the sets are disjoint. We show
	$F \setminus \mathcal{Z} \subset F_{1} \cup F_{2} \cup F_{3}$.
	Let $x \in  F \setminus \mathcal{Z} $, so we have $h(x) > 0$.
	There exist some sequences $(y_{l})_{l \in \mathbb{N}} \in F^{\mathbb{N}}$, $(t_{l})_{l \in \mathbb{N}}$ and 
	$(\tau_{l})_{l \in \mathbb{N}}$ so that for all $l \in \mathbb{N}$, we have $0 < t_{l} \le h(x)$, 
	$ t_{l} \rightarrow h(x)$, $\frac{t_{l}}{4} \le \tau_{l} \le \frac{t_{l}}{3}$, $ d(x,y_{l}) < \frac{\tau_{l}}{3}$ 
	and $ (y_{l},\tau_{l}) \notin S_{total}$. Due to
	$\tau_{l} \le \frac{t_{l}}{3} \le \frac{h(x)}{3} \le \frac{50}{3}$, we have for every $l \in \mathbb{N}$ either
	$\delta(B(y_{l},\tau_{l}))= \frac{\mu(B(y_{l},\tau_{l}))}{\tau_{l}^{\N}} < \frac{1}{2}\delta$ or
	$\delta(B(y_{l},\tau_{l})) 
				\ge \frac{1}{2} \delta \text{ and } \beta_{1;k}(y_{l},\tau_{l}) \ge 2 \varepsilon$ or 
	$\delta(B(y_{l},\tau_{l})) 
				\ge \frac{1}{2} \delta \text{ and } \beta_{1;k}(y_{l},\tau_{l}) < 2 \varepsilon$,
			and for every plane $P \in \mathcal{P}(\n,\N)$ with 
			\mbox{$\beta_{1;k}^{P}(y_{l},\tau_{l}) \le 2 \varepsilon$}, we have
			$\varangle(P, P_{0}) > \alpha.$

	Choose $l$ so large that $\frac{4h(x)}{5} \le t_{l}$. We obtain
	$\frac{h(x)}{5} \le \frac{t_{l}}{4} \le \tau_{l} \le \frac{t_{l}}{3} \le \frac{h(x)}{2}$.
	Furthermore, we have $y_{l} \in F$ and $d(x,y_{l}) \le \frac{\tau_{l}}{3} < \frac{\tau_{l}}{2}$.
	Since $(y_{l},\tau_{l})$ fulfils one of this tree cases, it follows
	$ x \in F_{1} \cup F_{2} \cup F_{3}$.
\end{proof}

The following lemma is for later use (cf. Lemma \ref{25.09.2014.2} and Lemma \ref{25.09.2014.1}).
\begin{lem} \label{rem3.3}
	Let $\alpha >0$. There exists some constant 
	$ \bar \varepsilon = \bar \varepsilon(\n,\N,C_{0},\alpha)$ so that if $\eta < 2 \bar \varepsilon$
	and $k \ge 2000$, 
	there holds for all $\varepsilon \in [\frac{\eta}{2},\bar \varepsilon)$:
	If $x \in F_{3}$ and $ h(x) \le t \le \min\{100 h(x),49\}$,
	we get $ \varangle(P_{(x,t)},P_{0}) > \frac{1}{2} \alpha$, where $P_{(x,t)}$ is the plane granted 
	since $(x,t) \in S_{total}$ (cf. Definition \ref{12.07.13.1}).
\end{lem}
\begin{proof}
	Let $\alpha > 0$ and $k \ge 400$. We set 
	$\bar \varepsilon:= \min\{\varepsilon_{0},\varepsilon_{0}',\alpha(5C_{3})^{-1}\}$,
	where $\varepsilon_{0}$, $\varepsilon_{0}'$, $C_{3}$ and $C_{3}^{'}$ depend only on $\n,\N$ and $C_{0}$
	will be chosen during this proof. Furthermore, let $\eta \le 2 \varepsilon < 2 \bar \varepsilon$.

	Since $x \in F_{3}$ and $x \notin (F_1 \cap F_2)$, there exists some $y \in F$, 
	$\tau \in  \left[ \frac{h(x)}{5},\frac{h(x)}{2} \right]$ and $\bar P \in \mathcal{P}(\n,\N)$ 
	with $d(x,y) \le \frac{\tau}{2}$, $\beta_{1;k}^{\bar P}(y,\tau) \le \varepsilon$ and
	$\varangle(\bar P,P_{0}) \ge \frac{3}{4}\alpha$.
	Furthermore $h(x) \le t$ implies $(x,t) \in S \subset S_{total}$ and hence $\delta(B(x,t)) \ge \frac{1}{2}\delta$
	and $\beta_{1;k}^{P_{(x,t)}}(x,t) \le 2\varepsilon.$
	Now with Lemma \ref{lem2.6} ($c=500$, $\xi =2$, $t_{x}=t$, $t_{y}=\tau$, $\lambda = \frac{\delta}{2}$), 
	there exist some constants $C_{3}=C_{3}(\n,\N,C_{0}) > 1$ and 
	$\varepsilon_{0}=\varepsilon_{0}(\n,\N,C_{0})>0$ so that 
	$\varangle(\bar P,P_{(x,t)}) \le C_{3} \varepsilon$.
	Due to $\varangle(\bar P,P_{0}) \ge \frac{3}{4}\alpha$ and $\varepsilon < \frac{\alpha}{4C_{3}}$ this gives
	$\varangle(P_{(x,t)},P_{0})  > \frac{1}{2} \alpha$.
\end{proof}

\subsection{The distance to a well approximable ball}
We recall that the set $S$ depends on the choice of $\alpha$ and $\varepsilon $. Hence the functions $d$ and $D$
defined in the next definition depend on $\alpha$ and $\varepsilon$ as well.
We introduce $\pi:=\pi_{P_0}: \mathbb{R}^{\n} \to P_0$ \index{Functions ! $\pi$}, the orthogonal 
projection on $P_0$.

\begin{dfn}[The functions $d$ and $D$] \index{Functions ! $d(x)$} \index{Functions ! $D(x)$}
	Let $\alpha, \varepsilon >0$.
	If $\eta \le 2 \varepsilon$, we get with Lemma \ref{rem3.1} (i) that $S \neq \emptyset$.
	We define $d : \mathbb{R}^{\n} \rightarrow [0,\infty)$ and $D : P_{0} \rightarrow [0,\infty)$ with
	\[ d(x) := \inf_{(X,t) \in S} (d(X,x) + t) \hspace{20mm} D(y) := \inf_{x \in \pi^{-1}(y)} d(x).\]
\end{dfn}

Let us call a ball $B(X,t)$ with $(X,t) \in S$ a good ball.
Then the function $d$ measures the distance from the given point $x$ to the nearest good ball,
using the furthermost point in the ball. This implies that a ball $B(x,d(x))$
always contains some good ball.
The function $D$ does something similar. Consider the projection of all good balls to the plane $P_{0}$.
Then $D$ measures the distance to the nearest projected good ball in the same sense as above
(cf. next lemma).

\begin{lem}\label{remnachdefD}
	Let $\alpha, \varepsilon > 0$. If $\eta \le 2 \varepsilon$ and $y \in P_{0}$ we have
	$D(y) = \inf_{(X,t) \in S}(d(\pi(X),y) + t)$.
\end{lem}
\begin{proof}
	Due to $d(X,x) \ge d(\pi(X),\pi(x))$ we have $D(y) \ge \inf_{(X,t) \in S} (d(\pi(X),y) + t)$.
	Assume that $\lim_{l \rightarrow \infty} (d( \pi(X_{l}),y) + t_{l}) > \inf_{(X,t) \in S} (d(\pi(X),y) + t)$
	for some sequence $(X_{l},t_{l}) \in S$.
	Now there exists some $ l \in \mathbb{N}$ so that 
	\begin{align*}
		D(y) > d\big(\pi(X_{l}) + X_{l} - \pi(X_{l}) ,y + X_{l} - \pi(X_{l})\big) + t_{l} 
		& \ge \inf_{x \in \pi^{-1}(y)} d(X_{l} ,x) + t_{l} \ge D(y)
	\end{align*}
	which is a contradiction.
\end{proof}

\begin{lem}\label{rem3.7}
	The functions $d$ and $D$ are Lipschitz functions with Lipschitz constant 1.
\end{lem}
\begin{proof}
	Let $x,y \in \R^{\n}$. We get with the triangle inequality $d(x) \le d(y) + d(x,y)$ and
	$d(y) \le d(x) + d(x,y)$. This implies $|d(x)-d(y)| \le d(x,y)$. Using the previous lemma,
	we can use the same argument for the function $D$.
\end{proof}

\begin{lem} \label{mengebeschraenkt}
	We have 
		$\left\{ x \in \mathbb{R}^{\n} \big| d(x) < 1 \right\}  \subset B(0,6)$
	and
		$d(x)  \le 60$
	for all $x \in B(0,5)$.
\end{lem}
\begin{proof}
	Let $x \in \mathbb{R}^{\n}$ with $  \inf_{(X,t) \in S} (d(X,x) + t) = d(x) < 1$.
	Hence there exists some $X \in F \subset B(0,5)$ with
	$ d(0,x) \le d(0,X) + d(X,x) < 6.$
	If $x \in B(0,5)$, we have $d(x) \le 10 + 50$.
\end{proof}

\begin{lem} \label{rem3.8} \label{7.2.10;1}
	Let $\alpha, \varepsilon >0$. If $\eta \le 2 \varepsilon$,
	we have $d(x) \le h(x)$ for all $x \in F$ and
	  \[ \mathcal{Z} = \left\{ x \in F | d(x)=0 \right\}, \quad
	  \pi(\mathcal{Z})=\{y \in P_{0} \ | \ D(y)=0\}.\]
	Furthermore, both sets $\mathcal{Z}$ and $\pi(\mathcal{Z})$ are closed. 
	We recall that $\pi$ denotes the orthogonal projection on the plane $P_{0}$. 
\end{lem}
\begin{proof}
	Let $x \in F$. With Lemma \ref{rem3.1} (i), we have $(x,h(x)) \in S$ and hence $d(x) \le h(x)$. 
	This implies $\mathcal{Z} \subset \left\{ x \in F | d(x) = 0 \right\}$.

	Now let $x \in F$ with $h(x) > 0$. We prove $d(x)>0$.
	There exist some sequences $t_{l} \rightarrow h(x)$ and
	some sequence $(X_{i},s_{i}) \in S$ with $d(X_{i},x) + s_{i} \to d(x)$.
	If on the one hand there exists some subsequence with $X_{i} \to x$ we obtain for another subsequence 
	$s_{i} \ge h(X_{i}) \ge t_{i}>0$ for sufficiently large $i$ and hence $d(x)>0$.
	If on the other hand $d(X_{i},x)$ has an positive lower bond, we conclude
	$d(x) \ge \lim_{l \rightarrow \infty}d(X_{l},x) > 0$.
	
	Now we prove the second equality.
	If $y \in \pi(\mathcal{Z})$, there exists some $x_{0} \in \mathcal{Z}$ with $\pi(x_{0})=y$ and $d(x_{0})=0$. 
	Now we get $0 \le D(y) \le d(x_{0})=0.$

	If $y \in P_{0}$ with $D(y)=0$, since $d$ is continuous, we get with Lemma \ref{mengebeschraenkt}
	that there exists some $a \in \pi^{-1}(y)$ with $d(a)=0$. This implies $a \in F$ and hence 
	$a \in \mathcal{Z}$. Thus $y \in \pi(\mathcal{Z})$.

	According to Lemma \ref{rem3.7}, $d$ and $D$ are continuous and hence these sets are closed.
\end{proof}

\begin{lem} \label{lem3.9}
	Let $0<\alpha \le \frac{1}{4}$. There exists some 
	$\bar \varepsilon = \bar \varepsilon (\n,\N,C_{0})$ so that if $\eta < 2 \bar \varepsilon$ and $k \ge 4$
	for all $\varepsilon \in [\frac{\eta}{2},\bar \varepsilon)$, there holds:
	For all $x,y \in F$ we have
	\begin{align*}
		d(x,y) &\le 6(d(x)+d(y)) + 2d(\pi(x),\pi(y)),\\
		d(\pi^\perp(x),\pi^\perp(y)) &\le 6(d(x)+d(y)) + 2\alpha d(\pi(x),\pi(y)).
	\end{align*}
\end{lem}
\begin{proof}
	Let $0< \alpha < \frac{1}{4}$ and $k \ge 4$. During this proof, there occur several smallness conditions on 
	$\varepsilon$. The minimum of those will give us the constant $ \bar \varepsilon$. Let 
	$\eta \le 2 \varepsilon < 2\bar \varepsilon$.

	The first estimate is an immediate consequence of the second estimate. So we focus on this one.
	Due to $F \subset B(0,5)$ the LHS is always less than 10. Hence we can assume that $d(x) + d(y)<2$.
	We choose some arbitrary $r_{x} \in (d(x),d(x)+1)\subset (0,3)$. There exists some $(X,t) \in S$ with
	$ d(x) \le d(X,x) + t < r_{x} $. According to Lemma \ref{rem3.1} (ii), 
	it follows that $(X,r_{x}) \in S$. Analogously, for all $r_{y} \in (d(y),d(y)+1)$, 
	we can choose some $Y \in F$ with $d(Y,y) < r_{y}$ and $(Y,r_{y}) \in S$.
	Now it is enough to prove $d(\pi^{\perp}(x),\pi^{\perp}(y)) \le 6(r_{x}+r_{y}) + 2\alpha d(\pi(x),\pi(y))$
	since $r_{x} \ge d(x)$ and $r_{y} \ge d(y)$ were arbitrarily chosen.
	We can assume $d(X,Y) > 2( r_{x}+r_{y})$ since otherwise $d(x,y) \le d(x,X) + d(X,Y) + d (Y,y)$
	immediately implies the desired estimate.

	We define $B_{1} := B(X, \frac{1}{2}d(X,Y))$ and $ B_{2} := B(Y, \frac{1}{2}d(X,Y))$. 
	With Lemma \ref{rem3.1} (i) we obtain $B_{1}, B_{2} \in S$.
	Let $P_{1} $ and $ P_{2}$ be the associated planes to $B_{1}$ and $B_{2}$
	(see Definition \ref{12.07.13.1}). 
	With Lemma \ref{lem2.6} ($x=X$, $y=Y$, $c=1$, $\xi = 2$, $t_x=t_y=\frac{1}{2}d(X,Y)$, $\lambda = \frac{1}{2}\delta$)
	there exist some constants $C_{3}=C_{3}(\n,\N,C_{0})>1$ and 
	$\varepsilon_{0}=\varepsilon_{0}(\n,\N,C_{0})>0$ so that if $\varepsilon < \varepsilon_{0}$
	for  $w \in P_{1}$, we obtain
	\begin{align} \label{2.12.09;1}
		 d(w,P_{2}) \le C_{3}(\n,\N,C_{0},\delta) \varepsilon\left(\textstyle{\frac{1}{2}}d(X,Y)+d(w,X)\right).
	\end{align}

	Let $B_{1}^{'} := B(X, \textstyle{\frac{1}{2}}\varepsilon^{\frac{1}{2\N}}d(X,Y)+r_{x})$ and 
	$B_{2}^{'} := B(Y, \textstyle{\frac{1}{2}}\varepsilon^{\frac{1}{2\N}}d(X,Y)+r_{y})$. 
	Lemma \ref{rem3.1} (i) implies that these balls are in $S$.
	Now we conclude using $\delta(B_{i}^{'}) \ge \frac{\delta}{2}$, $B_{i}^{'} \subset kB_{i},$ and 
	$\beta_{1;k}^{P_{i}}(B_{i}) \le 2 \varepsilon$ for $i \in \{1,2\}$ that
	\begin{align*}
		\frac{1}{\mu(B_{i}^{'})} \int_{B_{i}^{'}}\frac{d(X^{'},P_{i})}{d(X,Y)} \mathrm d\mu(X^{'})
		& \le \frac{1}{\delta \varepsilon^{\frac{1}{2}}} \frac{1}{\left(\textstyle{\frac{1}{2}}d(X,Y)\right)^{\N}} 
			\int_{kB_{i}} \frac{d(X^{'},P_{i})}{\textstyle{\frac{1}{2}}d(X,Y)} \mathrm d\mu(X^{'}) 
		\le \frac{2}{\delta } \varepsilon^{\frac{1}{2}}.
	\end{align*}
	With Chebyshev's inequality, we deduce that there exists some $X^{'} \in B_{1}^{'}$ and some $Y^{'} \in B_{2}^{'}$ 
	so that 
	$d(X^{'},P_{1}) \le \frac{2}{\delta } \varepsilon^{\frac{1}{2}} d(X,Y)$ and 
	$d(Y^{'},P_{2}) \le \frac{2}{\delta } \varepsilon^{\frac{1}{2}} d(X,Y)$.
	
	Now let $X_{1}^{'}:= \pi_{P_1}(X^{'}) $ be the orthogonal projection of $X^{'}$ on $P_{1}$, 
	$Y_{2}^{'}:=\pi_{P_2}(Y^{'})$ the  
	orthogonal projection of $Y^{'}$ on $P_{2}$, and 
	$X_{12}^{'}:=\pi_{P_2}(X_1^{'})$ the orthogonal projection of $X_{1}^{'}$ on $P_{2}$. 
	If $\varepsilon$ is small enough, we have with $\varrho \in \{\pi,\pi^{\perp} \}$
	\begin{align*}
		d(\varrho(X),\varrho(X^{'}))\le d(X,X^{'}) 
		& \le \textstyle{\frac{1}{2}}\varepsilon^{\frac{1}{2\N}}d(X,Y) + r_{x}, \\
		d(\varrho(Y),\varrho(Y^{'}))\le d(Y,Y^{'}) 
		& \le \textstyle{\frac{1}{2}}\varepsilon^{\frac{1}{2\N}}d(X,Y) + r_{y}, \\ 
		d(\varrho(X^{'}),\varrho(X_{1}^{'}))\le d(X^{'},X_{1}^{'}) 
		& = d(X^{'},P_{1}) \le \frac{2}{\delta }\varepsilon^{\frac{1}{2}}d(X,Y), \\ 
		d(\varrho(Y^{'}),\varrho(Y_{2}^{'}))\le d(Y^{'},Y_{2}^{'}) 
		& = d(Y^{'},P_{2}) \le \frac{2}{\delta }\varepsilon^{\frac{1}{2}}d(X,Y), \\
		d(\varrho(X_{1}^{'}),\varrho(X_{12}^{'}))\le d(X_{1}^{'},X_{12}^{'}) 
		& = d(X_{1}^{'},P_{2}) \stackrel{\eqref{2.12.09;1}}{<} 2C_{3}\varepsilon d(X,Y). 
	\end{align*}
	According to Definition \ref{12.07.13.1}, we have $\varangle(P_{2},P_{0}) \le \alpha$ and 
	we get with Lemma \ref{6.9.11.1} ($X_{12}^{'},Y_{2}^{'} \in P_{2}$) using $\alpha \le \frac{1}{4}$ 
	\begin{align}\label{07.10.2013.4}
		d(X_{12}^{'},Y_{2}^{'}) 
		&\le \frac{1}{1-\alpha}d(\pi(X_{12}^{'}),\pi(Y_{2}^{'})) \le 2d(\pi(X_{12}^{'}),\pi(Y_{2}^{'})), \\
		d(\pi^{\perp}(X_{12}^{'}),\pi^{\perp}(Y_{2}^{'})) 
		& \le \frac{\alpha}{1-\alpha} d(\pi(X_{12}^{'}),\pi(Y_{2}^{'}))
			\le \frac{4}{3}\alpha d(\pi(X_{12}^{'}),\pi(Y_{2}^{'})).\label{13.7.11.1}
	\end{align}
	Inserting the intermediate points $X^{'}$, $X_{1}^{'}$, $X_{12}^{'}$, $Y_{2}^{'}$, $Y^{'}$
	using triangle inequality twice and using the previous inequalities, there exists some constant $C$ so that
	\begin{align*}
		d(X,Y) 
		& \le C\textstyle{\frac{1}{\delta}}\varepsilon^{\frac{1}{2\N}} d(X,Y) + r_{x}+r_{y} + 2d(\pi(X_{12}^{'}),\pi(Y_{2}^{'}))\\
		& \le C\textstyle\frac{1}{\delta}\varepsilon^{\frac{1}{2\N}} d(X,Y) + 3(r_{x}+r_{y}) + 2d(\pi(X),\pi(Y))
	\end{align*}
	and hence if $\varepsilon$ fulfils $C \frac{1}{\delta} \varepsilon^{\frac{1}{2\N}} \le \frac{1}{2}$,
	we get
	\begin{align} \label{07.10.2013.3}
		d(X,Y) \le 6(r_{x}+r_{y}) + 4d(\pi(X),\pi(Y)).
	\end{align}
	As for $d(X,Y)$, we estimate $d\left(\pi^{\perp}(X),\pi^{\perp}(Y)\right)$ by repeated use
	of the triangle inequality and \eqref{13.7.11.1}.
	With \eqref{07.10.2013.3}, we deduce 
	\begin{align*} 
		& \ \ \ \ \ \ \ \ d\big(\pi^{\perp}(X),\pi^{\perp}(Y)\big) \\
		&\stackrel{\hphantom{\eqref{07.10.2013.3}}}{\le} C\textstyle{\frac{1}{\delta}}\varepsilon^{\frac{1}{2\N}} d(X,Y) + 3(r_{x}+r_{y}) + \frac{4}{3}\alpha d(\pi(X),\pi(Y))\\
		&\stackrel{\eqref{07.10.2013.3}}{\le} C\textstyle{\frac{1}{\delta}}\varepsilon^{\frac{1}{2\N}} [6(r_{x}+r_{y}) + 4d(\pi(X),\pi(Y))] + 3(r_{x}+r_{y}) + \frac{4}{3}\alpha d(\pi(X),\pi(Y))\\
		&\stackrel{\hphantom{\eqref{07.10.2013.3}}}{\le} 4(r_{x}+r_{y}) + 2\alpha d(\pi(X),\pi(Y)).
	\end{align*}
	This implies using $d(\pi^{\perp}(x),\pi^{\perp}(X)) \le d(x,X) \le r_{x}$ 
	and $d(\pi^{\perp}(Y),\pi^{\perp}(y)) \le d(Y,y) \le r_{y}$ that
	\begin{align*}
		d(\pi^\perp(x),\pi^\perp(y)) & \le 5(r_{x}+r_{y}) + 2\alpha d(\pi(X),\pi(Y))
		\le 6(r_{x}+r_{y}) + 2\alpha d(\pi(x),\pi(y)).
	\end{align*}
\end{proof}

\subsection{\texorpdfstring{A Whitney-type decomposition of $P_{0} \setminus \pi(\mathcal{Z})$}{A Whitney-type decomposition}}

In this part, 
we show that $P_{0} \setminus \pi(\mathcal{Z})$ can be decomposed as a union of disjoint cubes $R_{i}$,
where the diameter of $R_{i}$ is proportional to $D(x)$ for all $x \in R_{i}$. This result is 
a variant of the Whitney decomposition for open sets in $\R^{\N}$, cf. \cite[Appendix J]{FourierAnalysis}.

\begin{dfn}[Dyadic primitive cells]
	1.\, We set $\mathcal{D}$ to be the set of all dyadic primitive cells on $P_0$. We recall that 
	the plane $P_0$ is an $\N$-dimensional linear subspace of $\R^{\n}$.\\
	2.\, Let $r \in (0,\infty)$ and $Q$ be some cube in $\R^{\n}$.
		By $rQ$, we denote the cube
		with the same centre and orientation as $Q$ but $r$-times the diameter.
\end{dfn}

We mention that the function $D$ depends on the choice of $\alpha$ and $\varepsilon$ because $D$ depends
on the set $S \subset S_{total}^{\varepsilon,\alpha}$. Hence the family of cubes given by the following 
lemma depends on the choice of $\alpha$ and $\varepsilon$ as well.

\begin{lem} \label{inneredisjunkt} \label{Riueberdeckung} \label{rem3.10} \label{lem3.11} 
	Let $\alpha, \varepsilon >0$. If $\eta \le 2 \varepsilon$, then
	there exists a countable family of cubes $\{R_{i}\}_{i \in I} \subset \mathcal{D}$ such that
	\begin{enumerate} \renewcommand{\labelenumi}{(\roman{enumi})} 
	\item $10\diam R_{i} \le D(x) \le 50 \diam R_{i}$ for all $x \in 10 R_{i}$,
	\item $P_{0} \setminus \pi(\mathcal{Z}) = \bigcup_{i \in I} R_{i} = \bigcup_{i \in I} 2R_{i}$ 
		and cubes $R_{i}$ have disjoint interior,
	\item for every $ i,j \in I$ with $10R_{i} \cap 10R_{j} \neq \emptyset$, we have
		$ \frac{1}{5}  \diam R_{j} \le  \diam R_{i} \le 5 \diam R_{j},$
	\item for every $ i \in I$, there are at most $180^{\N}$ cells $R_{j}$ with 
		$10R_{i} \cap 10 R_{j} \neq \emptyset$.
	\end{enumerate}
\end{lem}
\begin{proof}
	For $z \in P_{0}$, $D(z)>0$, we define
	$Q_{z} \in \mathcal{D}$ \index{Subsets of $P_{0}$ ! $R_{i}$}
	as the largest dyadic primitive cell that contains $z$ and fulfils 
	$\diam Q_{z} \le \frac{1}{20} \inf_{u \in Q_{z}} D(u)$.
	For such a given $z$ the cell $Q_{z}$ exists because
	the function $D$ is continuous and $D(z)>0$. Hence if we choose a small enough dyadic primitive 
	cell $Q$ that contains $z$, we get $\diam Q \le \frac{1}{20} \inf_{u \in Q} D(u)$. Due to the 
	dyadic structure, there can only be one largest dyadic primitive cell that contains $z$ and 
	fulfils the upper condition.
	We choose $ R_{i} \in \mathcal{D}$ such that
	$\{R_{i}| i \in I\} = \{Q_{z} \in \mathcal{D} | z \in P_{0}, D(z)>0\}$ and $R_i=R_j$ is equivalent to $i=j$.\\
	(i)\, 
	Let $x \in 10 R_{i}$ and $u \in R_{i}$. We get $20 \diam R_{i} \le D(u) < D(x) + 10 \diam R_{i}$,
	and hence $10 \diam R_{i} \le D(x)$.
	Let $J_{i}\in \mathcal{D}$ be the smallest cell in $\mathcal{D}$ with $R_{i} \subsetneq J_{i}$
	and choose $u \in J_{i}$ so that
	$D(u) < 20 \diam J_{i} = 40 \diam R_{i}$. This is possible because otherwise $R_{i}$ is not maximal
	relating to $\diam R_{i} \le \frac{1}{20} \inf_{v \in R_{i}} D(v)$.
	We obtain $D(x) \le D(u) + d(u,x) < 50 \diam R_{i}$.\\
	(ii) \, 
	If the interior of some cells $R_{i}$ and $R_{j}$ were not disjoint, because of the dyadic structure,
	one cell would be contained in the other. But then one of those would not be the maximal cell. 
	Hence the $R_{i}$'s have disjoint interior.
	For all $x \in 2R_{i}$, we obtain using (i) and Lemma \ref{7.2.10;1} that $x \notin \pi(\mathcal{Z})$. 
	Now let $x \in P_{0} \setminus \pi(\mathcal{Z})$. With Lemma \ref{7.2.10;1}, we get $D(x)>0$.
	So there exists the cube $Q_{x} \in \mathcal{D}$ with $x \in Q_{x}$ and hence 
	$x \in \bigcup_{i \in I} R_{i}$.\\
	(iii) \,
	If $10 R_{i} \cap 10R_{j} \neq \emptyset$ we can apply (i) for some $x \in 10 R_{i} \cap 10R_{j}$ 
	and obtain the assertion.
	(iv)\,
	Let $i \in I$ and $R_{j}$ with $10 R_{i} \cap 10R_{j} \neq \emptyset$. We conclude with (iii) that
	$d(R_{i},R_{j}) \le 30 \diam R_{i}$ and so $R_{j} \subset (1+30+5) R_{i} $. 
	Furthermore, we have $\diam R_{j} \ge \frac{1}{5} \diam R_{i}$. Since the cells $R_{j}$ are disjoint, 
	there exist at most $\frac{\cH^{\N}(36R_{i})}{\cH^{\N}(R_{j})} \le (180)^{\N}$
	cells $R_{j}$ with $10 R_{i} \cap 10R_{j} \neq \emptyset$.
\end{proof}

Now we set \label{17.1.10;1} $ U_{12} := B(0,12) \cap P_{0}$ \index{Subsets of $P_{0}$ ! $U_{12}$}
and $I_{12} := \{ i \in I | R_{i} \cap U_{12} \neq \emptyset \}$ \index{Index sets ! $I_{12}$}.

\begin{lem} \label{vor3.12}
	Let $\alpha, \varepsilon >0$. If $\eta \le 2 \varepsilon$, 
	for every $ i \in I_{12}$, there exists some ball $B_{i}=B(X_{i},t_{i}) $ 
	\index{Subsets of $\mathbb{R}^{\n}$ ! $B_{i}$} 
	with $(X_{i},t_{i}) \in S$, 
	$\diam R_{i} \le \diam B_{i} \le 200 \diam R_{i}$ and
	$d(\pi(B_{i}),R_{i}) \le 100 \diam R_{i}$.
\end{lem}
\begin{proof}
	Let $ i \in I_{12}$ and $x \in R_{i}$. Use Lemma \ref{remnachdefD}, Lemma \ref{7.2.10;1} 
	and Lemma \ref{rem3.10} (i), (ii) to get some $(X,t) \in S$ with 
	$d(\pi(X),x) + t \le 2D(x) \le 100 \diam  R_{i}$.
	Choose $B_{i} :=B(X_{i},t_{i}) := B(X,r)$ with $r = \max \{ t,\frac{\diam  R_{i}}{2} \} \le 100 \diam R_{i}$. 
	Now we have $d(\pi(B_{i}),R_{i}) \le 100 \diam R_{i}$ and
	$\diam R_{i} \le \diam B_{i} \le 200 \diam R_{i}$.
	You can show that $r < 50$ and hence with Lemma \ref{rem3.1} (ii), we get $(X,r) \in S$.
\end{proof}

\subsection{\texorpdfstring{Construction of the function $A$}{Construction of the function A}} \label{17.05.2013.1}

We recall that $\pi:=\pi_{P_0}: \mathbb{R}^{\n} \to P_0$  is
the orthogonal projection on $P_0$
and introduce
$\pi^{\perp}:=\pi_{P_0}^{\perp}: \mathbb{R}^{\n} \to P_0^{\perp}$, \index{Functions ! $\pi^{\perp}$}
the orthogonal projection on $P_0^{\perp}$, where
$P_0^{\perp}:=\{x \in \R^{\n}|x\cdot v=0 \text{ for all } v \in P_{0}\}$  
is the orthogonal complement of $P_{0}$.
To define the function $A$, we want to invert the projection $\pi|_{\mathcal{Z}}$ on $\mathcal{Z}$.

\begin{lem}\label{12.11.2013.1}
	Let $0<\alpha \le \frac{1}{4}$. There exists some 
	$\bar \varepsilon = \bar \varepsilon (\n,\N,C_{0})$ so that if $\eta < 2 \bar \varepsilon$
	and $k \ge 4$
	for all $\varepsilon \in [\frac{\eta}{2},\bar \varepsilon)$, 
	the orthogonal projection
	$\pi|_{\mathcal{Z}} : \mathcal{Z} \rightarrow P_{0}$ is injective.
\end{lem}
\begin{proof}
	The assertion follows directly from Lemma \ref{rem3.8} and Lemma \ref{lem3.9}.
\end{proof}
Since $\pi|_{\mathcal{Z}} : \mathcal{Z} \rightarrow P_{0}$ is injective, we are able to define the desired Lipschitz function 
$A$ on \label{DefvonAonpiZ}
$\pi(\mathcal{Z})$ by \index{Functions ! $A$}
\[A(a):= \displaystyle\pi^{\perp}\left( \pi|_{ \mathcal{Z}}^{-1}(a) \right)\]
where $a \in \pi(\mathcal{Z})$.

\begin{lem} \label{ALipschitz1}
	Under the conditions of the previous lemma,
	the map $A\big|_{\pi(\mathcal{Z})}$ is $2 \alpha$-Lipschitz.
\end{lem}
\begin{proof}
	Due to Lemma \ref{12.11.2013.1} for $ a, b \in \pi(\mathcal{Z})$, there exist distinct 
	$ X,Y \in \mathcal{Z}$ with $\pi(X)=a$ and $ \pi(Y) = b$.
	We have $A(a) = \pi^{\perp}(X)$, $A(b)= \pi^{\perp}(Y)$ and Lemma \ref{rem3.8} implies that $d(X)=d(Y)=0$. 
	So, with Lemma \ref{lem3.9}, we get $d(A(a),A(b)) \le 2\alpha d(a,b)$.
\end{proof}

Now we have a Lipschitz function $A$ defined on $\pi(\mathcal{Z})$. By using 
Kirszbraun's theorem \cite[Thm 2.10.43]{Federer}, we would obtain a Lipschitz extension of $A$ 
defined on $P_{0}$ with the same Lipschitz constant $2\alpha$, where the graph of the extension covers $\mathcal{Z}$.
But until now, we do not know that $\mathcal{Z}$ is a major part of $F$. We cannot even be sure 
that $\mathcal{Z}$ is not a null set. So we do not use Kirszbraun's theorem here, but we will 
extend $A$ by an explicit construction. This will help us to show that the other parts of $F$, 
in particular $F_{1}, F_{2}, F_{3}$, are quite small.

\begin{dfn}\label{3.12.2013.10}
	Let $\alpha, \varepsilon > 0$. If $\eta \le 2 \varepsilon$,
	for all $i \in I_{12}$, we set $P_{i} := P_{(X_{i},t_{i})}$,
	where $P_{(X_{i},t_{i})}$ is the $\N$-dimensional plane, which is,
	in the sense of Definition \ref{12.07.13.1}, associated to the ball 
	$B(X_{i},t_{i})=B_{i}$ given by Lemma \ref{vor3.12}.
\end{dfn}

\begin{lem} \label{AiLipschitz}
	Let $0< \alpha \le \frac{1}{2}$ and $\varepsilon >0$.
	If $\eta \le 2 \varepsilon$, then for all $i \in I_{12}$, there exists 
	some affine map $A_{i}: P_{0} \rightarrow P_{0}^{\perp}$ 
	with graph $G(A_i)=P_i$ and $A_{i}$ is $2 \alpha$-Lipschitz.
\end{lem}
\begin{proof}
	Use $\varangle(P_{i},P_{0})\le \alpha \le \frac{1}{2}$ (cf. definition of $S_{total}$) 
	and apply Corollary \ref{24.04.2012.1}.
\end{proof}

\label{14.12.09;1}

In the following, we use differentiable functions defined on subsets of $P_{0}$. For the definition of the
derivative see section \ref{diff_on_lin_subspace} on page \pageref{diff_on_lin_subspace}.

\begin{lem}\label{15.10.2013.1}
	Let $\alpha, \varepsilon > 0$. If $\eta \le 2 \varepsilon$, then 
	there exists some partition of unity $\phi_{i} \in C^{\infty}(U_{12},\R)$, $i \in I_{12}$,
	with $0 \le \phi_{i} \le 1$ on $U_{12}$, $\phi_{i} \equiv 0$ on the exterior of $3R_{i}$ and
	$\sum_{i \in I_{0}} \phi_{i}(a)= 1$ for all $a \in U_{12}$. Furthermore there exists some constant $C=C(\N)$ with
	$| \partial^{\omega} \phi_{i}(a) | \le \frac{C(\N)}{ (\diam R_{i})^{|\omega|}}$ where $\omega$ 
	is some multi-index with $1 \le |\omega|\le 2$.
\end{lem}
\begin{proof}
	For every $i \in I_{12}$, we choose some function 
	$\tilde{\phi_{i}} \in \mathcal{C}^{\infty}(P_{0},\mathbb{R})$ with \label{Defphii}
	$0 \le \tilde{\phi_{i}} \le 1, \tilde{\phi_{i}} \equiv 1 $ on $2R_{i}$, $\tilde{\phi_{i}} \equiv 0$ 
	on the exterior of $3R_{i}$, $ | \partial^{\omega} \tilde{\phi_{i}} | \le \frac{C}{ \diam R_{i}}$ 
	for all multi-indices $\omega$ with $|\omega|=1$ and 
	$| \partial^{\kappa}\tilde{\phi_{i}} | \le \frac{C}{(\diam R_{i}) ^{2}} $ for all multi-indices $\kappa$ 
	with $|\kappa|=2$. Now on $V := \bigcup_{i \in I_{12}} 2R_{i}$, we can define the partition of unity
	$\phi_{i}(a) := \frac{\tilde{\phi_{i}}(a)}{\sum_{j \in I_{12}} \tilde{\phi_{j}}(a)}$.
	For all $a \in V$, there exists some $i \in I_{12}$ with $a \in 2R_{i}$ and hence 
	$\sum_{j \in I_{12}} \tilde{\phi}_{j}(a) \ge 1$. 
	Moreover, due to Lemma \ref{lem3.11} (iv), there are only finitely many $j \in I_{12}$ such that
	$\tilde \phi_{j}(a) \neq 0$.
	Due to the control we have on the derivatives of $\tilde{\phi_{i}}$, we obtain with
	Lemma \ref{lem3.11} (iv) the desired estimates of the derivatives of $\phi_{i}$.
\end{proof}

\begin{dfn}[Definition of $A$ on $U_{12}$]\index{Functions ! $A$} \label{04.11.2013.1}
	Let $\alpha, \varepsilon >0$. If $\eta \le 2 \varepsilon$ and $k \ge 4$,
	we extend the function $A: \pi(\mathcal{Z}) \to P_{0}^{\perp} \subset \R^{\n}$, 
	$a \mapsto \pi^{\perp}\left( \pi|_{ \mathcal{Z}}^{-1}(a) \right)$ (see page \pageref{DefvonAonpiZ}) to 
	the whole set $U_{12}$ by setting for $a \in U_{12}$ 
	\[A(a):= \begin{cases}
			{\displaystyle\pi^{\perp}\left( \pi|_{ \mathcal{Z}}^{-1}(a) \right)} & ,a \in \pi(\mathcal{Z}) \\ \vspace{-2mm}\\
			{\displaystyle \sum_{i \in I_{12}} \phi_{i}(a) A_{i}(a)} &, a \in U_{12} \cap \bigcup_{i \in I_{12}}2R_{i}.
	         \end{cases}\]
	With $\mathcal{Z} \subset F \subset B(0,5)$, we get $\pi(\mathcal{Z}) \subset U_{12}$ and, 
	with Lemma \ref{lem3.11} (ii), we obtain \\
	$ \bigcup_{i \in I_{12}}2R_{i} \cap \pi(\mathcal{Z}) = \emptyset$, hence
	we have defined $A $ on the whole set \\ $U_{12}=(U_{12} \cap \bigcup_{i \in I_{12}}2R_{i}) \ \dot{\cup} \ \pi(\mathcal{Z})$. 
	\label{Aeindeutigdefiniert}
\end{dfn}

\subsection{\texorpdfstring{$A$ is Lipschitz continuous}{A is Lipschitz continuous}}
In this section, we show that $A$ is Lipschitz continuous.
We start with some useful estimates.

\begin{lem} \label{lem3.12}
	Let $0<\alpha \le \frac{1}{4}$. 
	There exists some $\bar k \ge 4$ and some
	$\bar \varepsilon = \bar \varepsilon (\n,\N,C_{0})$ so that if $k \ge \bar k$ and
	$\eta < 2 \bar \varepsilon$
	for all $\varepsilon \in [\frac{\eta}{2},\bar \varepsilon)$, 
	there exist some constants $C>1$ and $\bar C=\bar C(\n,\N,C_{0})>1$ 
	so that for all $i,j \in I_{12}$ with $i \neq j$ and 
	$10R_{i} \cap 10R_{j} \neq \emptyset$, we get
	\begin{enumerate}
	\renewcommand{\labelenumi}{\textup{(\roman{enumi})}} 
		\item $d(B_{i},B_{j}) \le C \diam R_{j}$,
		\item $d(A_{i}(q),A_{j}(q)) \le \bar C
			\varepsilon \diam R_{j}$ for all $ q \in 100R_{j}$,
		\item the Lipschitz constant of the map $(A_{i}-A_{j}): P_{0} \rightarrow P_{0}^{\perp}$ fulfils
			$ \Lip_{A_{i}-A_{j}} \le \bar C \varepsilon $,
		\item $d(A(u),A_{j}(u)) \le \bar C \varepsilon \diam R_{j}$ for all $u \in 2R_{j} \cap U_{12}$.
	\end{enumerate}
\end{lem}
\begin{proof}
	Let $0<\alpha \le \frac{1}{4}$. We set 
	$\bar \varepsilon = \min \big\{\frac{\delta}{2},\bar \varepsilon',\varepsilon_{0}\big\}$, where 
	$\delta=\delta(\n,\N)$ is defined on page \pageref{Wahlvondelta}, $\bar \varepsilon'$ is the 
	upper bound for $\varepsilon$ given by Lemma \ref{lem3.9} and $\varepsilon_{0}$ is the constant from
	Lemma \ref{lem2.6}.
	Let $\eta < 2 \bar\varepsilon$ and choose $\varepsilon$ such that
	$\eta \le 2\varepsilon < 2\bar\varepsilon$.\\
		(i)\, 
		Let $B_{i}= B(X_{i},t_{i})$ and $B_{j}= B(X_{j},t_{j})$.
		Lemma \ref{lem3.11} and Lemma \ref{vor3.12} imply $d(\pi(X_{i}),\pi(X_{j})) \le C \diam R_{j}$,
		and, using $(X_{l},t_{l}) \in S$ we have
		$d(X_{l}) \le 500 \diam R_{j}$ for $l  \in \{ i,j\}$.
		Now Lemma \ref{lem3.9} implies the assertion.\\
		(ii)\,
		At first, we show for $q \in 100 R_{j}$ that $ d(A_{i}(q) + q, X_{i}) \le C \diam R_{j}$.
		Since $(X_{i},t_{i}) \in S \subset S_{total}$, $\varepsilon \le \frac{\delta}{4}$,  
		and Lemma \ref{nachlem2.6} 
		($\sigma= 2\varepsilon$, $x=X_{i}$, $t=t_{i}$, $\lambda = \frac{1}{2} \delta$, $P=P_{i}$) 
		we get $B(X_{i},2t_{i}) \cap P_{i} \neq \emptyset$. Thus there exists some 
		$a \in P_{0}$ with $A_{i}(a)+a \in B(X_{i},2t_{i}) \cap P_{i}$ and 
		$a \in \pi(2B_{i})$. Since $A_{i}$ is $2\alpha$-Lipschitz and $ \alpha < \frac{1}{2}$, using 
		Lemma \ref{lem3.11} and \ref{vor3.12} we obtain by inserting $A_{i}(a)+a$ with triangle inequality
		\begin{align}\label{13.1.10;1}
			d(A_{i}(q)+q,X_{i})  \le |A_{i}(q)-A_{i}(a)| + d(q,a) + \diam B_{i} 
			\le C \diam R_{j}.
		\end{align}
		With Lemma \ref{lem3.11} and \ref{vor3.12}, there exists some constant $C>2$ so that
		$\frac{1}{C}t_{j} \le t_{i} \le C t_{j}$.
		Moreover, we have $(X_{i},t_{i}),(X_{j},t_{j}) \in S\subset S_{total}$ 
		With $k \ge \bar k :=2C^{2} \ge 4C$, Lemma \ref{lem2.6} 
		($x=X_{j}$, $y=X_{i}$, $c=C$, $\xi =2$, $t_{x}=t_{j}$, $t_{y}=t_{i}$
		$\lambda = \frac{\delta}{2}$) implies that there exists some $\varepsilon_{0} >0 $
		and some constant $C_{3}=C_{3}(\n,\N,C_{0}) >1$ so that,
		for $\varepsilon < \bar \varepsilon \le \varepsilon_{0}$ with the already shown (i),
		\eqref{13.1.10;1} and Lemma \ref{vor3.12},  we get
		\begin{align}\label{9.4.2013.1}
			d(A_{i}(q)+q,P_{j}) \le C_{3} \varepsilon \left(t_{j}+d(A_{i}(q)+q,X_{j}) \right)
			\le C \varepsilon \diam R_{j}.
		\end{align}
		Furthermore, there exists some $o \in P_{0}$ so that $A_{j}(o)+o = \pi_{P_{j}}(A_{i}(q)+q)$. 
		Now, since $A$ is $2\alpha$-Lipschitz, we have 
		$d(A_{j}(o)+o,A_{j}(q)+q) \le 2d(o,q) \le 2d(A_{i}(q)+q,A_{j}(o)+o)$ and hence
		with Lemma \ref{lem3.11} and Lemma \ref{vor3.12} we obtain for some $C=C(\n,\N,C_{0})$
		\[d(A_{i}(q) + q, A_{j}(q)+ q) \le d(A_{i}(q)+q , P_{j}) + d(A_{j}(o)+o, A_{j}(q)+q)
		 \stackrel{\eqref{9.4.2013.1}}{\le}  C \varepsilon \diam R_{j}.\]
		(iii)\, 
		Without loss of generality, we assume $ \diam R_{i} \le \diam R_{j}$. 
		We have $B(y,2\diam R_{i}) \cap P_{0} \subset 20 R_{i} \cap 20 R_{j}$
		for some $y \in 10 R_{i} \cap 10 R_{j} \neq \emptyset$.
		We choose arbitrary $a,b \in B(y,2\diam R_{i})\cap P_{0}$ with $d(a,b) \ge \diam R_{i}$. 
		Now, with (ii), we get
		\[ |(A_{i}-A_{j})(a)-(A_{i}-A_{j})(b) | \le C \varepsilon \diam R_{i} \le C(\n,\N,C_{0}) \varepsilon d(a,b).\]
		Since $A_{i}-A_{j}$ is an affine map, this implies $ \Lip_{A_{i}-A_{j}} \le C(\n,\N,C_{0}) \varepsilon $.\\
		(iv)\,
		We get the estimate using Definition \ref{04.11.2013.1}, $\sum_{l \in I_{12}}\phi_{l}(u)=1$, 
		Lemma \ref{lem3.11} (iv) and (ii) of the current Lemma.
\end{proof}

\begin{lem}\label{4.5.2012.1}
	Let $0<\alpha \le \frac{1}{4}$. 
	There exists some $\bar k \ge 4$ and some
	$\bar \varepsilon = \bar \varepsilon (\n,\N,C_{0},\alpha) < \alpha$ so that if $k \ge \bar k$ and
	$\eta < 2 \bar \varepsilon$
	for all $\varepsilon \in [\frac{\eta}{2},\bar \varepsilon)$, 
	the function $A$ is Lipschitz continuous on $2R_{j} \cap U_{12}$ 
	for all $j \in I_{12}$ with Lipschitz constant $3 \alpha$.
\end{lem}
\begin{proof}
	Let $0 < \alpha \le \frac{1}{4}$. We set 
	$\bar \varepsilon := \min\big\{\bar \varepsilon', \frac{\alpha}{\tilde C}\big\}$,
	where $\bar \varepsilon'$ is the upper bound for $\varepsilon$ given by Lemma \ref{lem3.12} and 
	$\tilde C(\n,\N,C_{0})$ is some constant presented at the end of this proof. Let $\eta < 2 \bar \varepsilon$ and choose
	$\varepsilon >0$ such that $\eta \le 2 \varepsilon < 2 \bar \varepsilon$.
	Let $a,b \in 2R_{j} \cap U_{12}$. We obtain 
	\[|A(a)-A(b)| \le \sum_{i \in I_{12}} \phi_{i}(a) |A_{i}(a)-A_{i}(b)|  
		+ \sum_{i \in I_{12}} | \phi_{i}(a) - \phi_{i}(b) | |A_{i}(b)-A_{j}(b)|. \]
	If $\phi_{i}(a) - \phi_{i}(b) \neq 0 $, we get \label{18.12.09;3}
	$3 R_{i} \cap 2 R_{j} \neq \emptyset$ and so we can apply Lemma \ref{lem3.11} (iii), (iv) and 
	Lemma \ref{lem3.12} (ii). 
	Since $\varepsilon < \bar \varepsilon \le \frac{\alpha}{\tilde C}$, 
	we obtain with Lemma \ref{AiLipschitz} and Lemma \ref{15.10.2013.1} that $A$ is $3 \alpha$ Lipschitz.
\end{proof}

\begin{lem}\label{29.3.10;1}
	Under the conditions of the previous lemma for
	$a,b \in U_{12} \setminus \pi(\mathcal{Z})$ 
	with $[a,b] \subset U_{12} \setminus \pi(\mathcal{Z})$,
	we have that $d(A(a),A(b)) \le 3\alpha d(a,b)$.
\end{lem}
\begin{proof}
	Lemma \ref{Riueberdeckung} (ii) implies that for all $v \in [a,b]$, there exists some
	$j \in I_{12}$ with $v \in R_{j}$ and, with Lemma \ref{rem3.10} (i), we get $D(v) >0$.
	Assume that the set $\tilde{I}_{12} := \left\{i \in I_{12} | R_{i} \cap [a,b] \neq \emptyset \right\}$ 
	is infinite. The cubes $R_{i}$ have disjoint interior, so there exists some sequence 
	$(R_{i_{l}})_{l \in \mathbb{N}}$, $i_{l} \in \tilde{I}_{12}$ with $ \diam R_{i_{l}} \rightarrow 0$.
	Hence there exists some sequence $(v_{l})_{l \in \mathbb{N}}$ with $v_{l} \in R_{i_{l}} \cap [a,b]$ 
	and, with Lemma \ref{rem3.10} (i), we obtain
	$D(v_{l}) \le 50 \diam R_{i_{l}} \rightarrow 0$.
	Let $\overline{v} \in [a,b]$ be an accumulation point of $(v_{l})_{l \in \mathbb{N}}$.
	Since $D$ is continuous (Lemma \ref{rem3.7}), we deduce $D(\overline{v})=0$, which is 
	according to Lemma \ref{rem3.8} equivalent to
	$\overline{v} \in \pi(\mathcal{Z})$. 
	This is in contradiction to $[a,b] \subset P_{0} \setminus \pi(\mathcal{Z})$ and so the set $\tilde{I}_{12}$ has to be finite.
	With Lemma \ref{4.5.2012.1} and $[a,b] \subset \bigcup_{i \in \tilde I_{12}} R_{i}$, we get 
	$d(A(a),A(b)) \le 3 \alpha d(a,b)$.
\end{proof}

Now we show that $A$ is Lipschitz continuous on $U_{12}$ with some large Lipschitz constant. 
After that, using the continuity of $A$,
we are able to prove that $A$ is Lipschitz continuous with Lipschitz constant $3\alpha$.

\begin{lem} \label{AstetigaufU0}
	Let $0<\alpha \le \frac{1}{4}$. 
	There exists some $\bar k \ge 4$ and some
	$\bar \varepsilon = \bar \varepsilon (\n,\N,C_{0},\alpha) < \alpha$ so that if $k \ge \bar k$ and
	$\eta < 2 \bar \varepsilon$
	for all $\varepsilon \in [\frac{\eta}{2},\bar \varepsilon)$, $A$ is Lipschitz continuous on $U_{12}$.
\end{lem}
\begin{proof}
	Let $0 < \alpha \le \frac{1}{4}$, $k \ge \bar k \ge 4$, where $\bar k$ is the constant from Lemma
	\ref{4.5.2012.1}, and let $\bar \varepsilon = \bar \varepsilon(\n,\N,C_{0},\alpha) \le \frac{\delta}{4}$
	be so small that we can apply Lemma \ref{lem3.9}, \ref{ALipschitz1}, \ref{lem3.12} and Lemma \ref{29.3.10;1}.
	Furthermore, let $\varepsilon >0$ such that $\eta \le 2 \varepsilon < 2 \bar \varepsilon$.
	Let $a, b \in U_{12}$ with $a \in \pi(\mathcal{Z})$ and $b \in 2R_{j}$ for some $j\in I_{12}$.
	We estimate $d(A(a),A(b)) \le d(A(a)+a,X_{j}) + d(X_{j},A(b)+b)$
	where $X_{j}$ is the centre of the ball $B_{j}=B(X_{j},t_{j})$ (see Lemma \ref{vor3.12}).

	At first, we consider $d(A(a)+a,X_{j})$.
	Since $A(a)+a \in \mathcal{Z}$, Lemma \ref{rem3.8} implies $d(A(a)+a) = 0$.
	Moreover, with Lemma \ref{vor3.12} and $(X_{j},t_{j}) \in S$, we deduce $d(X_{j}) \le 100 \diam R_{j} $
	and 
	\begin{align*}	
		d(\pi(A(a)+a) , \pi(X_{j})) 
		& \le d(a,b) + d(b,\pi(X_{j})) 
		\le d(a,b) + C \diam R_{j}.
	\end{align*}
	Using those estimates, Lemma \ref{lem3.9} implies
	$d(A(a)+a,X_{j}) \le 2d(a,b) + C\diam R_{j}$.

	Now we consider $d(X_{j},A(b)+b)$.
	We have $(X_{j},t_{j}) \in S \subset S_{total}$ and hence, with Lemma \ref{nachlem2.6} 
	using $\varepsilon < \bar \varepsilon \le \frac{\delta}{4}$,
	there exists some
	$y \in B(X_{j},2t_{j})\cap P_{j}$, where $P_{j}$ is the associated plane to $B_{j}$
	(see Definition \ref{3.12.2013.10}). 
	Since $\varangle(P_{j},P_{0}) \le \alpha \le \frac{1}{4}$, 
	we deduce with Lemma \ref{12.7.11.1}, Lemma \ref{vor3.12} and Lemma \ref{lem3.12} (iv) that
	\[ d(X_{j},A(b)+b) \le d(X_{j},y) + d(y,A_{j}(b)+b) + d(A_{j}(b)+b,A(b)+b)
		\le C (\diam R_{j} + d(a,b)).\]
	With Lemma \ref{lem3.11}, Lemma \ref{7.2.10;1} and using that $D$ is $1$-Lipschitz (Lemma \ref{rem3.7}) we obtain
	$\diam R_{j} \le D(b)-D(a) \le d(a,b)$ and hence $d(A(a),A(b)) \le C d(a,b)$.
	Due to Lemma \ref{ALipschitz1} and Lemma \ref{29.3.10;1} it remains to handle the case were
	$a,b \notin \pi(\mathcal{Z})$ and $[a,b] \cap \pi(\mathcal{Z}) \neq \emptyset$.
	This follows immediately from the just proven case and triangle inequality.
\end{proof}

\begin{lem} \label{29.3.10;2}
	Under the conditions of Lemma \ref{AstetigaufU0} for some
	$a \in \pi(\mathcal{Z})$, $i \in I_{12}$ and $b \in 2R_{j}$, we get
	$d(A(a),A(b)) \le 3 \alpha d(a,b)$.
\end{lem}
\begin{proof}
	We set $c:= \inf_{x \in [a,b]\cap \pi(\mathcal{Z})} d(x,b)$. Due to Lemma \ref{7.2.10;1}, 
	there exists some $v \in [a,b]\cap \pi(\mathcal{Z})$ with $d(v,b) = c$. 
	Furthermore, there exists some sequence $(v_{l})_{l} \subset [v,b]$ 
	with $v_{l} \rightarrow v$ where $l \rightarrow \infty$.
	With Lemma \ref{Riueberdeckung}, we deduce $([v,b] \setminus\{v\}) \subset \bigcup_{j \in I_{12}} 2R_{j}$.
	For every $l \in \mathbb{N}$ we obtain with Lemma \ref{29.3.10;1}
	$d(A(v),A(b)) \le d(A(v),A(v_{l})) + 3\alpha d(v,b)$.
	and, since $A$ is continuous (Lemma \ref{AstetigaufU0}) we conclude with $l \rightarrow \infty$ that
	$d(A(v),A(b)) \le 3 \alpha d(v,b)$.
	The assertion follows since we already know that $A$ is $2\alpha$-Lipschitz on $\pi(\mathcal{Z})$.
\end{proof}

\begin{lem}
	Under the conditions of Lemma \ref{AstetigaufU0} we have 
	$d(A(a),A(b)) \le 3 \alpha d(a,b)$
	for $a, b \in \bigcup_{j \in I_{12}} 2R_{j} \cap U_{12}$.
\end{lem}
\begin{proof}
	This is an immediate consequence of Lemma \ref{4.5.2012.1}, Lemma \ref{29.3.10;1} and Lemma \ref{29.3.10;2}.
\end{proof}

\begin{lem}\label{ALipschitz}
	Under the conditions of Lemma \ref{AstetigaufU0}, the function 
	$A$ is Lipschitz continuous on $U_{12}$ with Lipschitz constant $3 \alpha$.
\end{lem}
\begin{proof}
This follows directly from the previous Lemma and Lemma \ref{ALipschitz1}.
\end{proof}

The following estimate is for later use.
\begin{lem} \label{abschaetzungableitungvonA}
	Let $0<\alpha \le \frac{1}{4}$. 
	There exists some $\bar k \ge 4$ and some
	$\bar \varepsilon = \bar \varepsilon (\n,\N,C_{0})$ so that if $k \ge \bar k$ and
	$\eta < 2 \bar \varepsilon$
	for all $\varepsilon \in [\frac{\eta}{2},\bar \varepsilon)$, 
	there exists some constant $C=C(\n,\N,C_{0})$ so that for all $ j \in I_{12}$, $a \in 2R_{j}$ and
	for all multi-indices $\kappa$ with $|\kappa|=2$ we have
	$\partial^{\kappa} A(a)| \le \frac{C \varepsilon}{\diam R_{j}}$.
\end{lem}
\begin{proof}
	Choose $\bar k$ and $\bar \varepsilon$ as in Lemma \ref{lem3.12}.
	Let $\kappa$ be some multi-index with $|\kappa|=2$.
	For $i \in I_{12}$, the function
	$A_{i}$ is an affine map and hence 
	for some suitable $l_1,l_2 \in \{1,\dots,\N\}$ we have
	\begin{align} \label{15.10.2013.3}
		\partial^{\kappa}  A &= \partial^{\kappa}  \Bigl( \sum_{i \in I_{12}} \phi_{i}A_{i} \Bigr) 
		= \sum_{i \in I_{12}} \left( \partial^{\kappa}  \phi_{i} \right) A_{i} 
			+  \sum_{i \in I_{12}} \left(\partial_{l_1} \phi_{i} \partial_{l_2} A_{i} +
				\partial_{l_2} \phi_{i} \partial_{l_1} A_{i}\right).
	\end{align}
	Let $j \in I_{12}$ and $a \in 2R_{j}$. 
	Lemma \ref{lem3.11} implies that there exist at most $180^{\N}$
	cells $R_{i}$ so that $\partial^{\kappa}  \phi_{i}(a) \neq 0$ or $\partial^{\omega} \phi_{i}(a) \neq 0$,
	where $\omega$ is a multi-index with $|\omega|=1$. 
	So only finite sums occur in the following estimates.
	We have 
	$ \sum_{i \in I_{12}} \partial^{\omega} \phi_{i}= \partial^{\omega} \sum_{i \in I_{12}} \phi_{i} 
		= \partial^{\omega} \ 1 =0$
	so that we get
	\begin{align*}
		 |\partial^{\kappa}  A| 
		 \stackrel{\eqref{15.10.2013.3}}{\le} 
			\sum_{i \in I_{12}} |\partial^{\kappa}  \phi_{i}| \ | A_{i} - A_{j}| 
		+ \sum_{i \in I_{12}}| \partial_{l_1} \phi_{i} | \ |\partial_{l_2} (A_{i} - A_{j})|
		+ \sum_{i \in I_{12}}| \partial_{l_2} \phi_{i} | \ |\partial_{l_1} (A_{i} - A_{j})|.
	\end{align*}
	To estimate these sums, we only have to consider the case when $a$ is in the support of $\phi_{i}$
	for some $i\in I_{12}$. This implies 
	$3R_{i} \cap 2R_{j} \neq \emptyset$.
	Now use Lemma \ref{lem3.12} (ii), (iii), Lemma \ref{15.10.2013.1}, and Lemma \ref{lem3.11} (iii), (iv)
	to obtain the assertion.
\end{proof}

\setcounter{equation}{0}
\section{\texorpdfstring{$\gamma$-functions}{y-functions}} \label{gamma}
\newcommand{\f}{g}
In this chapter, we introduce the $\gamma$-function of some function $\f:P_{0} \to P_{0}^{\perp}$. 
This function measures how well $\f$ 
can be approximated in some ball by some affine function. 
The main results of this chapter are Theorem \ref{2.9.2014.1} on page \pageref{2.9.2014.1} and 
Theorem \ref{thm4.1} on page \pageref{thm4.1}.
We will use these statements in section \ref{F3issmall} to prove that $\mu(F_{3})$ is small.

\begin{dfn}
	Let $U \subset P_{0}$, $ q \in U$ and $ t > 0$ so that $B(q,t) \cap P_{0} \subset U$.
	Furthermore, let $\mathcal{A}=\mathcal{A}(P_0,P_0^{\perp})$ be the set of all affine functions 
	$a:P_{0} \rightarrow P_{0}^{\perp}$ and let $\f:U \to P_{0}^{\perp}$ be some function. 
	We define 
	\begin{align*}
		\gamma_{\f}(q,t)&:= \inf_{a \in \mathcal{A}} \frac{1}{t^{\N}} \int_{B(q,t)\cap P_{0}} 
		\frac{d(\f(u),a(u))}{t} \dd \cH^{\N}(u).
	\end{align*}
\end{dfn}

\begin{lem} \label{bem4.2}
	Let $U \subset P_{0}$, $q \in U$ and $t > 0$ so that $B(q,t) \cap P_{0} \subset U$. 
	Furthermore, let \\
	$g: U \to P_{0}^{\perp}$ be a Lipschitz continuous function such that the Lipschitz constant fulfils
	$60\N(10^{\N}+1)\left(8\N \frac{\vo{\N-1}}{\vol} \right)^{\N+1} \le \Lip_{\f}^{-1},$
	where $\vol$ denotes the $\N$-dimensional volume of the $n$-dimensional unit ball.
	Then we have
	\[\gamma_{\f}(q,t) \le 3 \ \tilde \gamma_{\f}(q,t):=3 \inf_{P \in \mathcal{P}(\n,\N)} \frac{1}{t^{\N}} 
		\int_{B(q,t)\cap P_{0}} \frac{d(u+\f(u),P)}{t} \dd \cH^{\N}(u),\]
	where $\mathcal{P}(\n,\N)$ is the set of all $\N$-dimensional affine planes in $\mathbb{R}^{\n}$.
\end{lem}
\begin{proof}
	Let $\f$ be a Lipschitz continuous function with an appropriate Lipschitz constant.
	By using $a : u \to g(q) \in \mathcal{A}$ as a constant map and by using that g is $1$-Lipschitz,
	we deduce $\gamma_{\f}(q,t) \le \Lip_{\f} \vol.$
	It follows, since for every $a \in \mathcal{A}$ the graph $G(a)$ of $a$ is in $\mathcal{P}(\n,\N)$, that
	$\tilde \gamma_{\f}(q,t) \le \gamma_{\f}(q,t) \le \Lip_{\f} \vol$.
	Let $0<\xi < \Lip_{\f} \vol$ and choose some $P \in \mathcal{P}(\n,\N)$ so that
	\begin{align} \label{5.12.11.1}
	    \frac{1}{t^{\N}} \int_{B(q,t)\cap P_{0}} \frac{d(u+\f(u),P)}{t} \dd \cH^{\N}(u)\le \tilde \gamma_{\f}(q,t)+\xi
		\le 2 \Lip_{\f} \vol. 
	\end{align}
	We set $D_{1}:=\left\{ v \in B(q,t) \cap P_0 | d(v+\f(v),P) \le 4 \Lip_{\f} t \right\}$,
	$D_{2}:=(B(q,t) \cap P_{0}) \setminus D_{1}$ and obtain using Chebyshev's inequality and \eqref{5.12.11.1}
	\begin{align}\label{21.11.11.1}
		\cH^{\N}(D_1) & \ge \vol t^{\N} - \cH^{\N}(D_{2}) \ge \frac{\vol}{2}t^{\N}  
	\end{align}
	Assume that every simplex $\triangle(u_0,\dots,u_{\N}) \in D_{1}$ is
	not an $(\N,H)$-simplex, where $H=\frac{\vol}{4\vo{\N-1}}t $.
	With Lemma \ref{18.11.11.1} ($m=\N$, $D=D_{1}$), there exists some plane $\hat P \in \mathcal{P}(\n,\N-1)$ such that 
	$D_1 \subset U_{H}(\hat P) \cap B(q,t) \cap P_0$.
	We get
	\begin{align*}
		\cH^{\N}(D_1)
		& \le \cH^{\N}(U_{H}(\hat P) \cap B(q,t) \cap P_0)
		 \le 2H \vo{\N-1}t^{\N-1}
		 = \frac{\vol }{2} t^{\N}.
	\end{align*}
	This is in contradiction to \eqref{21.11.11.1}, so there exists some $(\N,H)$-simplex
	$\triangle(u_0,\dots,u_{\N}) \in D_{1}$.
	We set $\hat P_0 := P_0 + \f(u_0)$, $y_i:=u_i+\f(u_0) \in \hat P_0$ for all $i \in \{0,\dots,\N\} $
	and $S:=\Delta(y_0,\dots,y_{\N}) \subset \hat P_0 \cap B(q+\f(u_{0}),t)$. 
	We recall that $P$ is the plane satisfying \eqref{5.12.11.1}.
	We obtain for all $i \in \{0,\dots,\N\}$
	\begin{align*}
		d(y_i,P) & \le d(u_i + \f(u_0), u_i + \f(u_i)) + d(u_i + \f(u_i),P)
		 \le \Lip_{\f} d(u_0,u_i) +4 \Lip_{\f} t
		 \le 6 \Lip_{\f} t.
	\end{align*}
	With Lemma \ref{21.11.11.2}, $C=4\frac{\vo{\N-1}}{\vol} > 1$\footnote{As the volume of the unit
	sphere is strictly monotonously decreasing when the dimension $n \ge 5$ increases, we get 
	$\frac{\vo{\N-1}}{\vol} > 1$ for all $n \ge 6$. With the factor $4$ we have that $4\frac{\vo{\N-1}}{\vol} > 1$
	for all $n \in \mathbb{N}$.}, 
	$\hat C=1$, $m=\N$, $\sigma = 6 \Lip_{\f}$,
	$P_1=\hat P_{0}$, $P_2=P$ and $x=q+\f(u_{0})$, we get $\varangle(P_{0},P)=\varangle(\hat P_{0},P) < \frac{1}{2}$,
	and, with Corollary \ref{24.04.2012.1},
	there exists some affine map $\bar a:P_0 \to P_0^{\perp}$ with graph $G(\bar a)=P$.
	Now we obtain with Lemma \ref{6.9.11.1} ($P_{1}=P$, $P_{2}=P_{0}$), $u,v \in P_0$ and 
	$\varangle(P_{0},P)< \frac{1}{2}$ that
	\begin{align}\label{21.11.11.3}
		d(v +\bar a(v),u+\bar a(u)) 
		 &\le 2d(\pi_{P_0}(v +\bar a(v)),\pi_{P_0}(u +\f(u))). 
	\end{align}
	That yields for $u \in B(q,t) \cap P_{0}$ and some suitable $v \in P_{0}$ with $v+\bar a(v)=\pi_P(u+\f(u))$
	\begin{align*}
		d(\f(u),\bar a(u)) 
		& \stackrel{\hphantom{\eqref{21.11.11.3}}}{\le} d(u+\f(u),P)+d(\pi_P(u+\f(u)),u+\bar a(u))\\
		& \stackrel{\eqref{21.11.11.3}}{\le}
		 d(u+\f(u),P) + 2d(\pi_{P_0}(v+\bar a(v)),\pi_{P_0}(u+\f(u)))
		\stackrel{\hphantom{\eqref{21.11.11.3}}}{=} 3 d(u+\f(u),P).
	\end{align*}
	Finally, using $\bar a \in \mathcal{A}$ and the last estimate, we get 
	$\gamma_{\f}(q,t)\stackrel{\eqref{5.12.11.1}}{\le} 3 (\tilde \gamma_{\f}(q,t)+\xi)$,
	and $0<\xi < \alpha \vol$ was arbitrarily chosen.
\end{proof}

\subsection{\texorpdfstring{$\gamma$-functions and affine approximation of Lipschitz functions}{y-functions and affine approximation of Lipschitz functions}}\label{27.10.2014.3}
In this and the following subsections, we use the notation $U_{l}:=B(0,l)\cap P_{0}$ for $l\in \{6,8,10\}$.
\begin{thm} \label{2.9.2014.1}
	Let $1 < \p < \infty$ and let $\f:P_{0} \to P_{0}^{\perp}$ be a Lipschitz continuous function with Lipschitz constant $\Lip_{\f}$
	and compact support.
	For all $\theta >0$, there exists some set $H_{\theta} \subset U_{6}$ and some constants 
	$C=C(\N,\p)$ and $\hat C=\hat C(\N,\n)$ with
	\[\cH^{\N}(U_{6} \setminus H_{\theta}) 
		\le \frac{C}{\theta^{\p(\N+1)} \Lip_{\f}^{\p}} \int_{U_{10}} \left( \int_{0}^{2} \gamma_{\f}(x,t)^{2} \frac{\dd t}{t} \right)^{\frac{\p}{2}}
		\dd \cH^{\N}(x)\]
	so that, for all $y \in P_{0}$, there exists some affine map $a_{y}:P_{0} \to P_{0}^{\perp}$ 
	so that if $ r \le \theta$ and $ B(y,r) \cap H_{\theta} \neq \emptyset$, we have 
	\[\|\f-a_{y}\|_{L^{\infty}(B(y,r)\cap P_{0},P_{0}^{\perp})} \le \hat C r \theta \Lip_{\f}, \]
	where $\|\cdot\|_{L^{\infty}(E)}$ denotes the essential supremum on  $E \subset P_{0}$
	with respect to the $\cH^{\N}$-measure.
\end{thm}

To prove this theorem, we need the following lemma.
If $\nu$ is some map, we use the notation 
$\nu_{t}(x):=\frac{1}{t^{\N}}\nu \left(\frac{x}{t} \right)$. \index{$\nu_{t}(x)$}

\begin{lem}\label{22.02.2013.01}
	There exists some radial function $\nu \in C_{0}^{\infty}(P_0,\R)$ with
	\begin{enumerate}
	\item	$\supp(\nu) \subset B(0,1)\cap P_{0}$ and $\widehat{\nu}(0)=0$,
	\item for all $x \in P_0 \setminus \{0\}$ and $i \in \{1,\dots, \N \}$, we have
		\begin{align} \label{26.3.10;1}	
			\int_{0}^{\infty}|\widehat{\nu}(t x)|^{2}\frac{\dd t}{t}=1 \ \ \ \ \ \ \text{and} \ \ \ \ \ \
			0 < \int_{0}^{\infty} |\widehat{(\partial_{i} \nu)_{t}}(x)|^{2} \ \frac{\dd t}{t} < \infty,
		\end{align}
	\item for all $i \in \{1,\dots, \N \}$, the function $\partial_{i} \nu$ has mean value zero 
		and, for all $a \in \mathcal{A}(P_{0},P_{0}^{\perp})$ (affine functions), the function $a \nu$
		has mean value zero as well.
	\end{enumerate}
\end{lem}
\begin{proof}
	Let $\nu_{1} : P_{0} \rightarrow \mathbb{R}$ be some non harmonic ($\Delta \nu_{1} \neq 0$), 
	radial $C^{\infty}$ function with support in $B(0,1)\cap P_{0}$. 
	We set $\nu_2 := \Delta \nu_{1} \in C^{\infty}(P_{0}) \cap C_0^{\infty}(B(0,1)\cap P_{0})$ and
	$0 < c_{1}:= \int_{0}^{\infty}|\widehat{\nu_{2}}(te) |^{2}\frac{\dd t}{t}$,
	where $e$ is some normed vector in $P_{0}$. With Lemma \ref{22.02.2013.02}, 
	we get $\nu_2$ is radial as well.
	Using Lemma \ref{10.12.12.2}, we obtain $|\widehat{\nu_{2}}(te)| = 4\pi^2 t^2 |\widehat{\nu_1}(te)|$
	and hence
	\begin{align*}
		0 < c_{1} &= \int_{0}^{\infty}|\widehat{\nu_{2}}(te) |^{2}\frac{\dd t}{t}
		= 16\pi^{4}\int_{0}^{\infty} t^3 |\widehat{\nu_{1}}(te) |^{2} \dd t < \infty
	\end{align*}
	because $\nu_{1}$ is in the Schwarz space and therefore $\widehat \nu_{1}$ 
	as well \cite[2.2.15, 2.2.11 (11)]{FourierAnalysis}.
	The previous equality also implies $\widehat{\nu_{2}}(0) = 0$.
	Now we set $\nu:= \sqrt{\frac{1}{c_{1}}} \nu_{2}$,
	which is a radial $C_{0}^{\infty}(P_{0},\R)$ function that fulfils 1. 
	We have for all $x \in P_{0} \setminus \{0\}$ 
	(use substitution with $t=r\frac{1}{|x|}$ and the fact that $\widehat \nu$ is radial)
	$\int_{0}^{\infty}|\widehat{\nu}(t x)|^{2}\frac{\dd t}{t}
		=\int_{0}^{\infty}|\widehat{\nu}(re)|^{2}\frac{\dd r}{r}=1.$
	In a similar way, we deduce for $i \in \{1,\dots,\N\}$ with Lemma \ref{10.12.12.2} 
	(using $|(\phi^{-1}(tx))^{\kappa}|\le|\phi^{-1}(tx)|=|tx|$ where $\kappa$ is some multi-index
	with $|\kappa|=1$)
	\begin{align*}
		\int_{0}^{\infty} |\widehat{(\partial_{i} \nu)_{t}}(x)|^{2} \ \frac{\dd t}{t}
		& \le |2\pi i|^{2} \int_{0}^{\infty} |tx|^{2}  \left| \widehat{\nu}(tx) \right|^{2} \ \frac{\dd t}{t} 
		 = 4\pi^{2} \int_{0}^{\infty} r \left| \widehat{\nu}\left(r\frac{x}{|x|}\right) \right|^{2} \ \dd r < \infty,
	\end{align*}
	where we use that the Fourier transform of a Schwartz function is a Schwartz function as well
	\cite[2.2.15]{FourierAnalysis}.
	The left hand side of the previous inequality can not be zero, because this would implicate that 
	$\partial_{i}\nu(x)=0$ for all $x \in P_{0}$, which is in contradiction to 
	$0 \neq \nu \in C_{0}^{\infty}(P_{0},\R)$.
	Hence $\nu$ fulfils 2.
	Using partial integration and $\Delta a=0$ for all $a \in \mathcal{A}(P_{0},P_{0}^{\perp})$ implies
	that $\partial_{i} \nu$ and $a \nu$ have mean value zero.
\end{proof}
For some function $f : P_{0} \rightarrow P_{0}^{\perp}$ and $x \in P_{0}$, we define the convolution
	of $\nu_{t}$ and $f$ by
\[ (\nu_{t} * f)(x) := \int_{P_{0}} \nu_{t}(x-y)f(y) \dd \cH^{\N}(y).\]

\begin{lem}[Calder\'on's identity] \label{calderon}
	Let $\nu$ be the function given by Lemma \ref{22.02.2013.01} and let 
	$u \in P_{0} \setminus \{0\}$ and $f\in L^{2}(P_{0},P_{0}^{\perp})$ or let $f \in \mathscr{S}^{'}(P_{0})$ 
	be a tempered distribution and $u \in \mathscr{S}(P_{0})$ (Schwartz space) with $u(0)=0$.
	Then we have
	\begin{align}
		f(u) = \int_{0}^{\infty} (\nu_{t}*\nu_{t}*f)(u) \frac{\dd t}{t}.
	\end{align}
	L\'eger calls  this identity ``Calder\'on's formula'' \cite[p. 862, 5. Calder\'on's formula and the size of $F_3$]{Leger}.
	Grafakos presents a similar version called ``Calder\'on reproducing formula'' 
	\cite[p.371, Exercise 5.2.2]{FourierAnalysis}.
\end{lem}

\begin{proof}
	At first, let $f \in L^2(P_{0},P_{0}^{\perp})$ and $u \in P_{0} \setminus \{0\}$. 
	We have with Lemma \ref{10.12.12.2} that $\widehat{(\nu_{t})}(u)=  \widehat{\nu}(t u)$
	and, with Fubini's theorem and Lemma \ref{FourFal}, we obtain 
	\begin{align*}
		\left( \int_{0}^{\infty} (\nu_{t}*\nu_{t}*f)(u) \frac{\dd t}{t}\right)^{\widehat{}} 
		& = \int_{0}^{\infty} \widehat{(\nu_{t})}(u)\widehat{(\nu_{t})}(u)\widehat{f}(u) \frac{\dd t}{t}
		\stackrel{\eqref{26.3.10;1}}{=} \widehat{f}(u).
	\end{align*}
	The  Fourier inversion holds on $L^{2}(P_{0},P_{0}^{\perp})$ 
	\cite[2.2.4 The Fourier Transform on $L^1+L^2$]{FourierAnalysis}, which gives the statement.
	Use the same idea to get this result for tempered distributions.
\end{proof}

\begin{proof}[Proof of Theorem \ref{2.9.2014.1}]
Let $\f \in C_{0}^{0,1}(P_{0},P_{0}^{\perp})$ and let $\nu$ be the function given by Lemma \ref{22.02.2013.01}.
We define
\begin{align*}
	\f_{1}(u) &:= \int_{2}^{\infty} (\nu_{t}* \nu_{t}*\f)(u) \frac{\dd t}{t}
		 + \int_{0}^{2} (\nu_{t}* (\Eins_{P_{0} \setminus U_{10}}\cdot(\nu_{t}*\f)))(u)
			 \frac{\dd t}{t},\\
	\f_{2}(u) & := \int_{0}^{2} (\nu_{t}* (\Eins_{U_{10}}\cdot(\nu_{t}*\f)))(u)
			 \frac{\dd t}{t}
\end{align*}
and the previous lemma implies that $\f = \f_{1} + \f_{2}$.
We recall the notation $U_{l}=B(0,l)\cap P_{0}$ for $l\in \{6,8,10\}$ and 
continue the proof of Theorem \ref{2.9.2014.1} with several lemmas.

\begin{lem}\label{lem5.2}
	$\f_{1} \in C^{\infty}(U_{8})$ and there exists some constant $C = C(\nu)$ so that
	for all multi-indices $\kappa$ with $|\kappa|\le 2$ we have
	$\| \partial^{\kappa} \f_{1}\|_{L^{\infty}(U_{8},P_{0}^{\perp})} \le C \Lip_{\f}$.

	$\f_{2}$ is Lipschitz continuous on $U_{8}$ with Lipschitz constant $C(\nu)\Lip_{\f}$.
\end{lem}
\begin{proof}
	For $x \in P_0$ we set
	\[ \f_{11}(x) := \int_{2}^{\infty} (\nu_{t} * \nu_{t} * \f)(x) \frac{\dd t}{t}, \ \ \
		\f_{12}(x) := \int_{0}^{2} (\nu_{t} * (\Eins_{P_{0} \setminus U_{10}} \cdot(\nu_{t} * \f)))(x)
			 \frac{\dd t}{t}\]
	so that $\f_{1} = \f_{11} + \f_{12}$ and we set
	$\varphi(x) := \int_{2}^{\infty} (\nu_{t} * \nu_{t})(x) \frac{\dd t}{t}$.

	At first, we look at some intermediate results:
	\begin{enumerate}
		\renewcommand{\labelenumi}{\Roman{enumi}.}
		\item $\f_{12}(x) = 0$ for all $x \in U_{8}$, due to the support of $\nu_{t}$ and 
		$\Eins_{P_{0} \setminus U_{10}}\cdot(\nu_{t} * \f)$.
		\item For every multi-index $\kappa$, there exists some constant $C=C(\N,\nu,\kappa)$ such that
			$|\partial^{\kappa} \varphi| \le C$, where 
			$\partial^{\kappa} \varphi(y) := \int_{2}^{\infty} \partial^{\kappa}(\nu_{t} * \nu_{t})(y) \frac{\dd t}{t}$.
			This is given by 
			$\partial^{\kappa} (\nu_{t}(y)) = \frac{1}{t^{|\kappa|}}(\partial^{\kappa} \nu)_{t}(y)$,
			and $|\partial^{\kappa}(\nu_{t} * \nu_{t})(y)| 
		 \le  \| \nu\|_{L^{\infty}(P_{0},\R)}  \| \partial^{\kappa}\nu\|_{L^{\infty}(P_{0},\R)} \frac{\vol}{t^{\N+|\kappa|}}$.
		\item
		For every multi-index $\kappa$, the function $\partial^{\kappa}\varphi$ has bounded support in $B(0,4) \cap P_{0}$.
		\begin{proof}
		Let $0 <t \le 2$ and $x \in P_0 \setminus B(0,4)$. We have $(\nu_{t} * \nu_{t})(x) = 0$ which implies that
		$\int_{0}^{2}(\nu_{t} * \nu_{t})(x) \frac{\dd t}{t}=0$.
		Now we consider $\varphi$ as a tempered distribution. The convolution with
		$\delta_{0}$, the Dirac mass at the origin, is an identity, hence we get with
		Calder\'on's identity (Lemma \ref{calderon}) for all $\eta \in \mathscr{S}(P_{0})$
		with $\eta(0)=0$
		\begin{align*}
			\varphi(\eta)=\varphi(\eta) - \delta_{0}(\eta) 
				& = \left(\int_{2}^{\infty} (\nu_{t} * \nu_{t}) \frac{\dd t}{t}\right)(\eta) -
					\left(\int_{0}^{\infty}(\nu_{t} * \nu_{t}) \frac{\dd t}{t}\right)(\eta)\\
				& = - \left(\int_{0}^{2}(\nu_{t} * \nu_{t}) \frac{\dd t}{t}\right)(\eta).
		\end{align*}
		Since this holds for arbitrary $\eta \in \mathscr{S}(P_{0})$ with 
		$\supp(\eta)\subset P_{0} \setminus B(0,4)$,
		we conclude that for such $\eta$ we have
		\[\int_{P_{0}}\varphi(x)\eta(x) \dd\cH^{\N}(x) 
			= -\int_{P_{0}} \int_{0}^{2}(\nu_{t} * \nu_{t})(x) \frac{\dd t}{t} \eta(x) \dd \cH^{\N}(x)=0\]
		and hence $\supp( \varphi ) \subset B(0,4) \cap P_{0}$.
		For the same kind of $\eta $, we get using Fubini's theorem and partial integration
		\begin{align*}
			\int_{P_{0}} \partial^{\kappa}\varphi(x)\eta(x) \dd\cH^{\N}(x)
			&= (-1)^{|\kappa|} \int_{2}^{\infty} \int_{P_{0}} (\nu_{t} * \nu_{t})(x)\partial^{\kappa}\eta(x) \dd\cH^{\N}(x)\frac{\dd t}{t}
			= 0
		\end{align*}
		since $\partial^{\kappa}\eta \in \mathscr{S}(P_{0})$ with 
		$\supp(\partial^{\kappa}\eta)\subset P_{0} \setminus B(0,4)$.
		\end{proof}
		\item 
		$\varphi \in C_{0}^{\infty}(P_{0})$
		\begin{proof}
		
		With II. and III. we conclude for every multi-index $\kappa$ that
		$\partial^{\kappa}\varphi \in L^{1}(P_{0},\R)$.
		With Fubini's theorem and partial integration, we see that $\partial^{\kappa} \varphi$ is the weak derivative of $\varphi$
		hence we have
		$\varphi \in W^{l,1}(P_{0})$
		for every $l \in \mathbb{N}$. The Sobolev imbedding theorem \cite[Thm 4.12]{Adams}
		gives us $\varphi \in C^{\infty}(P_{0})$ and, with III., we 
		obtain $\varphi \in C_{0}^{\infty}(P_{0})$.
		\end{proof}
	\end{enumerate}
	Now we have for all $ x \in U_{8}$ with Fubini's theorem \cite[1.4, Thm. 1]{Evans} $ \f_{11}(x) = (\varphi * \f)(x)$.
	We know, that $ \varphi \in C_0^{\infty}(P_{0})$ and $\f \in C_{0}^{0,1}(P_{0},P_{0}^{\perp})$.
	Hence we have $\f_{11} \in C_0^{\infty}(P_0)$, $\f \in W^{1,\infty}(P_{0})$ 
	and for $i,j \in \{1,\dots,\N \}$ we have
	$ \partial_{i} \f_{11} = \varphi * \partial_{i} \f$ and 
	$ \partial_{i}\partial_{j} \f_{11} = \partial_{i} \varphi * \partial_{j} \f$.
	With the Minkowski inequality \cite[Thm. 1.2.10]{FourierAnalysis} and IV., we obtain
	for $i,j \in \{1,\dots,\N \}$
	\begin{align*}
		\| \partial_{i} \f_{1} \|_{L^{\infty}(U_{8},\R)} & \stackrel{\text{I.}}{=} \|  \partial_{i} \f_{11} \|_{L^{\infty}(U_{8},\R)} 
			\le \| \partial_{i} \f \|_{L^{\infty}(U_{8},\R)} \| \varphi\|_{L^{1}(P_0)} \le C(\nu) \Lip_{\f},\\
		\| \partial_{i} \partial_{j} \f_{1} \|_{L^{\infty}(U_{8},\R)} 
			&\stackrel{\text{I.}}{=}\| \partial_{i} \partial_{j} \f_{11} \|_{L^{\infty}(U_{8},\R)} 
			\le \| \partial_{i} \f \|_{L^{\infty}(U_{8},\R)} \| \partial_{j} \varphi\|_{L^{1}(P_0)}\le C(\nu) \Lip_{\f}.
	\end{align*}
	Now it is easy to see that $\f_{2}$ is $C\Lip_{\f}$-Lipschitz on $U_{8}$ because we have $\f_{2}=\f - \f_{1}$ and $\f$ as well as
	$\f_{1}$ are $C\Lip_{\f}$-Lipschitz on $U_{8}$.
\end{proof}

\begin{rem}\label{9.10.2014.2}
Under the assumption that 
\[ \int_{U_{10}} \left( \int_{0}^{2} \gamma_{\f}(x,t)^{2} \frac{\dd t}{t} \right)^{\frac{\p}{2}}
		\dd \cH^{\N}(x) < \infty, \]
the next lemmas will prove that $\f_{2} \in W_{0}^{1,\p}(P_{0},P_{0}^{\perp})$.
We show for this purpose in Lemma \ref{lem5.1} that
\[\partial_{i} \f_{2}(x):=\int_{0}^{2} \partial_{i} (\nu_{t}* (\Eins_{U_{10}}(\nu_{t}*\f)))(x) \frac{\dd t}{t}\]
is in $L^{\p}(P_{0},P_{0}^{\perp})$. 
Using Fubini's theorem \cite[1.4, Thm. 1]{Evans}
and partial integration it turns out that $\partial_{i}g_{2}$ fulfils  the condition of a weak derivative.
\end{rem}

\begin{lem} \label{21.03.2013.1}
	We have $\supp(\f_{2}) \subset B(0,12) \cap P_{0}$ and $\supp(\partial_{i} \f_{2}) \subset B(0,12) \cap P_{0}$
	for all $i \in \{1,\dots,\N \}$.
\end{lem}
\begin{proof}
	If $0 <t<2$ and $x \in P_{0}$, we have
	$ \supp(\nu_{t}(x- \cdot)) \subset B(x,2) \cap P_{0}$
	and
	$\supp (\Eins_{U_{10}}(\nu_{t}*\f)) \subset B(0,10) \cap P_{0}.$
	This implies 
	$\supp (\nu_{t}* (\Eins_{U_{10}}(\nu_{t}*\f))) \subset B(0,12) \cap P_{0}$
	and hence we obtain $\supp(\f_{2})\subset B(0,12)$ and $\supp(\partial_{i} \f_{2}) \subset B(0,12) \cap P_{0}$.
\end{proof}

\begin{lem} \label{3.12.2012.1}
	Let $x \in U_{10}$ and $0<t<2$. We have
	$\left| \frac{(\nu_{t} * \f)(x)}{t} \right|  \le  \| \nu \|_{L^{\infty}(P_{0},\R)} \gamma_{\f}(x,t)$.
\end{lem}
\begin{proof}
	If $a:P_{0} \rightarrow P_{0}^{\perp}$ is an affine function, we get using Lemma \ref{22.02.2013.01} 3. that
	$(\nu_{t}*a)(x) = 0$ and hence, with Lemma \ref{22.02.2013.01} 1.
	\begin{align*}
		\left| \frac{(\nu_{t} * \f)(x)}{t} \right| 
			& = \left| \frac{(\nu_{t} * (\f - a))(x)}{t} \right|
			\le \| \nu \|_{L^{\infty}(P_{0},\R)}  \frac{1}{t^{\N}} \int_{P_{0}\cap B(x,t)} 
				\left|\frac{\f(y)-a(y)}{t}\right| \dd \cH^{\N}(y).
	\end{align*}
	Since $a$ was an arbitrary affine function, this implies the assertion.
\end{proof}

We have $p \in (1,\infty)$ and, for the proof of Theorem \ref{2.9.2014.1}, we can assume that
\[\int_{U_{10}} \left( \int_{0}^{2} \gamma_{\f}(x,t)^{2} \frac{\dd t}{t} \right)^{\frac{\p}{2}}\dd \cH^{\N}(x)< \infty.\]
\begin{lem} \label{lem5.1}
	We have $\f_{2} \in W_{0}^{1,\p}(P_{0},P_{0}^{\perp})$ and there exists some constant $C = C(\N,\p,\nu)$, so that
	for all $i \in \{1,\dots,\N \}$
	\[ \|\partial_{i}\f_{2} \|_{L^{\p}(P_{0},P_{0}^{\perp})}^{\p}
		\le C  \int_{U_{10}} \left( \int_{0}^{2}
		\gamma_{\f}(x,t)^{2} \frac{\dd t}{t}\right)^{\frac{\p}{2}} \dd \cH^{\N}(x),\]
	where $\partial_{i} \f_{2}(x) =  \int_{0}^{2} \partial_{i} (\nu_{t}* (\Eins_{U_{10}}(\nu_{t}*\f)))(x) \frac{\dd t}{t}$.
\end{lem}
\begin{proof}
	We recall that $\partial_{i}\f_{2}$ is the weak derivative of $\f_{2}$ (cf. Remark \ref{9.10.2014.2}).
	Due to \cite[Cor 6.31, An Equivalent Norm for $W_{0}^{m,p}(\Omega)$]{Adams} and 
	Lemma \ref{21.03.2013.1}, 
	we only have to consider $\|\partial_{i}\f_{2}\|_{L^{\p}(P_{0})}$ for all $i \in \{0,\dots,\N\}$
	to get $\f_{2} \in W_{0}^{1,\p}(P_{0},P_{0}^{\perp})$.
	For $x \in P_{0}$, we have
	$ \partial_{i} \nu_{t}(x) = \partial_{i}t^{-\N} \nu\left(\frac{x}{t}\right) 
		= t^{-1} (\partial_{i} \nu)_{t}(x)$
	and hence
	\begin{align*}
		\partial_{i} \f_{2}(x)
		& = \int_{0}^{2} \partial_{i} (\nu_{t}* (\Eins_{U_{10}}(\nu_{t}*\f)))(x)
			 \frac{\dd t}{t}
		 =  \int_{0}^{2} \left((\partial_{i} \nu)_{t}* \left(\Eins_{U_{10}}\left(\frac{\nu_{t}*\f}{t}\right)\right)\right)(x)
			 \frac{\dd t}{t}.
	\end{align*}
	Using duality (cf. \cite[The Normed Dual of $L^{p}(\Omega)$]{Adams}) it suffice to consider the following.
	Let $\frac{1}{\p}+\frac{1}{p'} =1$ and $f \in L^{\p'}(P_{0})$ with $\|f\|_{L^{\p'}(P_{0})}=1$.
	We get with Fubini's theorem \cite[1.4, Thm. 1]{Evans} and H\"older's inequality
	\begin{align*} 
			& \ \ \ \ \left|\int_{P_{0}} \ f(x) \ \partial_{i} \f_{2}(x) \  \dd \cH^{\N}(x) \right| \nonumber \\ 
			& \le \int_{P_{0}} \int_{0}^{2} \left|((\partial_{i} \nu)_{t}*f)(y)\right|  \ \left| 
				\left(\Eins_{U_{10}}\left(\frac{\nu_{t}*\f}{t}\right)\right)(y) \right|
				\ \frac{\dd t}{t} \dd \cH^{\N}(y)  \nonumber \\ 
			& \le  \int_{P_{0}}\left( \int_{0}^{2}  |((\partial_{i} \nu)_{t}*f)(y)|^{2} \ 
				\ \frac{\dd t}{t}\right)^{\frac{1}{2}} 
				\left(\int_{0}^{2} \left|\left(\Eins_{U_{10}}\left(\frac{\nu_{t}*\f}{t}\right)\right)(y)\right|^{2}
				\ \frac{\dd t}{t}\right)^{\frac{1}{2}} \dd \cH^{\N}(y) \nonumber \\
			& \le \left\| \left( \int_{0}^{\infty} |(\partial_{i} \nu)_{t}*f|^{2}
				\ \frac{\dd t}{t} \right)^{\frac{1}{2}} \right\|_{L^{\p'}(P_{0})} 
				\left( \int_{P_{0}} \left(\int_{0}^{2}
				\left|\left(\Eins_{U_{10}}\left(\frac{\nu_{t}*\f}{t}\right)\right)(y)\right|^{2}
				\ \frac{\dd t}{t}\right)^{\frac{\p}{2}}\dd \cH^{\N}(y) \right)^{\frac{1}{\p}}.
	\end{align*}
	There exists some constant $C=C(\N,\nu)$ with
	$|\partial_{i}\nu(x)| + |\nabla \partial_{i}\nu(x)| \le C(1+|x|)^{-\N-1}$
	because $\nu$ is a Schwartz function. 
	Together with Lemma \ref{22.02.2013.01}, all the requirements of Lemma \ref{4.3.2013.1} 
	with $\phi = \partial_{i}\nu$ and $q=\p'$ are fulfilled, which implies, since $\|f\|_{L^{\p}(P_{0})}=1$,
	that the first factor 
	of the RHS of the last estimate is some constant $C(\N,\p,\nu)$ independent of $f$. All in all, we obtain
	\begin{align*} \label{17.1.10;62}
		\| \partial_{i} \f_{2}\|_{L^{\p}(P_{0})} 
		\le C(\N,\p,\nu) \left( \int_{P_{0}} \left(\int_{0}^{2}
				\left|\left(\Eins_{U_{10}}\left(\frac{\nu_{t}*\f}{t}\right)\right)(y)\right|^{2}
				\ \frac{\dd t}{t}\right)^{\frac{\p}{2}}\dd \cH^{\N}(y) \right)^{\frac{1}{\p}},
	\end{align*}
	and with Lemma \ref{3.12.2012.1} the assertion holds.
\end{proof}

\begin{dfn} \label{27.02.2013.1}
	Let $B$ be a ball with centre in $P_{0}$ and $f:P_{0} \rightarrow P_{0}^{\perp}$ be some map. We define 
	the average of $f$ on $B$ and some maximal function for $x \in P_{0}$
	\[ \Avg{B}(f) := \frac{1}{(\diam B)^{\N}} \int_{B \cap P_{0}} f \dd \cH^{\N}, \index{$\operatorname{Avg}(f)$} \ \ \
		N(f)(x) :=
		\sup_{\genfrac{}{}{0pt}{}{t \in (0,\infty), y \in P_{0}}{ \text{with } d(y,x) \le t}} \left\{
		\frac{1}{2t}\Avg{B(y,t)}\left( |f -\Avg{B(y,t)}(f)|\right)  \right\} \index{$N(f)(x)$}. \]
	Moreover we define the oscillation of $f$ on $B$ by
	$\operatorname{osc}_{B}(f) :=  \sup_{x \in B \cap P_{0}} |f(x)-\Avg{B}(f)|$.
	
\end{dfn}

\begin{lem} \label{lem5.3}
	We have $\|N(\f_{2})\|_{L^{\p}(P_0,\R)} \le C\|D\f_{2}\|_{L^{\p}(P_{0},P_{0}^{\perp})}$, where $C=C(\N,\p)$.
\end{lem}
\begin{proof}
	We recall that $\f_{2} \in W_{0}^{1,\p}(P_{0},P_{0}^{\perp})$ (cf. Lemma \ref{3.12.2012.1}) and conclude with
	Poincar\'{e}'s inequality that $ \Avg{}_{B}(|\f_{2} -\Avg{}_{B}(\f_{2})|) = C(\N) \diam B \  \Avg{}_{B}(|D \f_{2}|)$,
	(if $f$ is a Matrix, we denote by $|f|$ a matrix norm)
	and hence we get for $x \in P_{0}$
	\begin{align*}
		N(\f_{2})(x) 
		& \le C(\N) \sup_{\genfrac{}{}{0pt}{}{t \in (0,\infty), y \in P_{0}}{ \text{with } d(y,x) \le t}}
			\Avg{B(y,t)}(|D \f_{2}|)
		= C(\N) M(D\f_{2})(x),
	\end{align*}
	where $M(D\f_{2})$ is the uncentred Hardy-Littlewood maximal function.
	Now, using \cite[Thm. 2.1.6]{FourierAnalysis}, we get the assertion.
\end{proof}

\begin{dfn}
	Let $\theta > 0$. We define
	$H_{\theta} := \left\{ x \in U_{6} | N(\f_{2})(x) \le \theta^{\N+1} \Lip_{\f} \right\}$.
\end{dfn}

\begin{lem}\label{3.9.2014.1}
	Let $\theta > 0$. There exists some constant $C=C(\N,\p,\nu)$ so that
	\[\cH^{\N}(U_{6} \setminus H_{\theta}) 
		\le \frac{C}{\theta^{\p(\N+1)} \Lip_{\f}^{\p}} \int_{U_{10}} \left( \int_{0}^{2} \gamma_{\f}(x,t)^{2} \frac{\dd t}{t} \right)^{\frac{\p}{2}}
		\dd \cH^{\N}(x).\]
\end{lem}
\begin{proof}
	With Lemma \ref{lem5.3}, Lemma \ref{lem5.1} and 
	$\|D\f_{2} \|_{L^{\p}(P_{0},P_{0}^{\perp})}^{\p} \le \N^{\p-1} \sum_{i=1}^{\N} \|\partial_{i} \f_{2}\|_{L^{\p}(P_{0},P_{0}^{\perp})}^{\p}$, 
	there exists some constant $C=C(\N, \p,\nu)$ with
	\[ \|N(\f_{2})\|_{L^{\p}(P_{0},P_{0}^{\perp})}^{\p} 
		\le Csum_{i=1}^{\N} \|\partial_{i} \f_{2}\|_{L^{\p}(P_{0},P_{0}^{\perp})}^{\p}
		\le C\int_{U_{10}} \left( \int_{0}^{2} \gamma_{\f}(x,t)^{2} \frac{\dd t}{t} \right)^{\frac{\p}{2}}
			\dd \cH^{\N}(x).\]
	Hence, using Chebyshev's inequality, we get the assertion.
\end{proof}

\begin{lem} \label{lem5.4}
	Let $B$ be a ball with centre in $P_{0}$.
	If $ (B \cap P_{0}) \subset U_{8}$, then there exists some constant $C=C(\n,\N,\nu)$ with 
	\[\operatorname{osc}_{B} (\f_{2}) \le C \diam B \left( 		
			\frac{1}{\diam B}\Avg{B}\Bigl(|\f_{2} -\Avg{B}(\f_{2})|\Bigr) \right)^{\frac{1}{\N+1}}\Lip_{\f}^{\frac{\N}{\N+1}}.\]
\end{lem}
\begin{proof}
	Let $(B \cap P_{0}) \subset U_{8}$ and $\lambda := \operatorname{osc}_{B} (\f_{2})$.
	The function $\f_{2}$ is Lipschitz continuous on $U_{8}$ with $\Lip_{\f_{2}}=C(\nu)\Lip_{\f}$ 
	(see Lemma \ref{lem5.2} on page \pageref{lem5.2}) and $B \cap P_{0}$ is closed.
	Hence there exists some $y \in B \cap P_{0}$ with
	$\lambda = |\f_{2}(y)-\Avg{}_{B}\f_{2}|$ and
	we get for $x \in B$ with 
	$d(x,y) \le \frac{\lambda}{2\Lip_{\f_{2}}}$ using triangle inequality
	$|\f_{2}(x)-\Avg{B}(\f_{2})| \ge \frac{\lambda}{2}$.
	Furthermore, using that $\f_{2}$ is continuous on $U_{8}$ 
	for all $l \in \{1,\dots,\n\}$, there exists some $z_{l} \in B \cap P_{0}$, with 
	$\f_{2}^{l}(z_{l}) =\Avg{B}(\f_{2}^{l})$ (where $\f_{2}^{l}(z_{l}) \in \R$ means the $l$-th component
	of $\f_{2}(z_{l}) \in \R^{\n}$).
	With $\f_{2}^{l}(y)-\Avg{B}(\f_{2}^{l}) \le \Lip_{\f_{2}} d(y,z_{l})$ for all $l \in \{1,\dots, \n\}$ we get
	$\lambda^{2} \le \n \left(\Lip_{\f_{2}} \diam B \right)^{2}$,
	which implies $\frac{\lambda}{\sqrt{\n}\Lip_{\f_{2}}} \le \diam B$.
	Since $y \in B$, there exists some ball $\hat B \subset B \cap B\left(y,\frac{\lambda}{2\Lip_{\f_{2}}}\right)$
	with $ \diam \hat B \ge \frac{\lambda}{2\sqrt{\n}\Lip_{\f_{2}}} $
	and hence with $|\f_{2}(x) -\Avg{}_{B}(\f_{2})| \ge \frac{\lambda}{2}$ we obtain
	\begin{align*}
		(\diam B)^{n} \Avg{B} \ |\f_{2}(x) -\Avg{B}(\f_{2})| 
			& \ge \vol \left(\textstyle{\frac{\lambda}{4\sqrt{\n}\Lip_{\f_{2}}}}\right)^{\N} \frac{\lambda}{2} .
	\end{align*}
	Using $\Lip_{\f_{2}}=C(\nu)\Lip_{\f}$, this implies the assertion.
\end{proof}

\begin{lem} \label{lem5.5}
	Let $\theta > 0$ and $y \in P_{0}$. There exists some constant $C=C(\n,\N,\nu)$ and some affine map $a_{y}:P_{0} \to P_{0}^{\perp}$ 
	so that if $ r \le \theta$ and $ B(y,r) \cap H_{\theta} \neq \emptyset$, we have \\
	$\|\f-a_{y}\|_{L^{\infty}(B(y,r)\cap P_{0},P_{0}^{\perp})} \le C r \theta \Lip_{\f}$. 
\end{lem}
\begin{proof}
	Let $y \in P_{0}$.
	If $\theta \ge 1$, we can choose $a_y(y'):=\f(y)$ as a constant and get the desired result directly from the
	Lipschitz condition.
	Now let $0 < \theta < 1$ and $y' \in B(y,r) \cap P_{0}$. We set $a_{y}(y'):=\f(y) + D \f_{1}(y)\phi^{-1}(y'-y)$.
	We have $d(y',U_{6}) \le d(y',H_{\theta}) \le d(y',y) + d(y,H_{\theta}) \le 2$. So we get $y',y \in U_{8}$.
	Using Taylor's theorem and Lemma \ref{lem5.2} we obtain
	\[|\f_{1}(y')-[\f_{1}(y) + D \f_{1}(y)\phi^{-1}(y'-y) ]| 
		\le \sum_{|\kappa|=2}\|\partial^{\kappa} \f_{1}\|_{L^{\infty}(U_{8})} |y'-y|^{2}
		\le C(\N,\nu) \Lip_{\f}r^{2}\]

	Since $r \le \theta <1$, $B(y,r) \cap H_{\theta} \neq \emptyset$ and $H_{\theta} \subset U_{6}$, 
	we obtain $B(y,r)\cap P_{0} \subset U_{8}$ and we can apply Lemma \ref{lem5.4}.
	Together with the definition of $H_{\theta}$ this leads to
	\[\operatorname{osc}_{B(y,r)}\f_{2} + \Lip_{\f}r^{2} \le C(\n,\N,\nu)r \theta \Lip_{\f}.\]
	Now by using $g=g_{1}+g_{2}$ and $|\f_{2}(y') - \f_{2}(y)| \le 2 \operatorname{osc}_{B(y,r)}\f_{2}$
	we get for every $y' \in B(y,r) \cap P_{0}$
	\[|\f(y')-[\f(y) + D \f_{1}(y)\phi^{-1}(y'-y) ]| \le C(\n,\N,\nu)r \theta \Lip_{\f}. \]
\end{proof}

Lemma \ref{3.9.2014.1} and Lemma \ref{lem5.5} complete the proof of Theorem \ref{2.9.2014.1}.
\end{proof}

\subsection{\texorpdfstring{The $\gamma$-function of $A$ and integral Menger curvature}{The y-function of A and integral Menger curvature}}\label{3.11.2014.2}
In this section, we prove the following Theorem \ref{thm4.1}.
It states that we get a similar control on the $\gamma$-functions applied to our function $A$
as we get in Corollary \ref{thm2.4} on the $\beta$-numbers.

For $\alpha, \varepsilon > 0$, $\eta \le 2 \varepsilon $ and $k \ge 4$, we defined $A$ on $U_{12}$
(cf. Definition \ref{04.11.2013.1} on page \pageref{04.11.2013.1}).
Since in this section we only apply the $\gamma$-functions to $A$, we set
$\gamma(q,t):=\gamma_{A}(q,t)$ and we recall the notation $U_{10}:=B(0,10)\cap P_{0}$.

\begin{thm} \label{thm4.1}
	There exists some $\tilde k \ge 4$ and some $\tilde \alpha =\tilde \alpha (\N) > 0$
	so that, for all $\alpha$ with $0 < \alpha \le \tilde \alpha$, there exists some
	$\tilde \varepsilon = \tilde \varepsilon(\n,\N,C_{0},\alpha)$ 
	so that, if $k \ge \tilde k$ and $\eta \le \tilde \varepsilon^{\p}$, we have for all
	$\varepsilon \in [\eta^{\frac{1}{\p}},\tilde \varepsilon]$ 
	that there exists some constant $C= C(\n,\N,\K,\p,C_{0},k)$ so that
	\begin{align*}
		\int_{U_{10}} \int_{0}^{2} \gamma(q,t)^{\p} \frac{\dd t}{t} \dd\cH^{\N}(q) 
		&\le C\varepsilon^{\p} + C\mathcal{M}_{\K^{\p}}(\mu)
		 \le C\varepsilon^{\p}.
	\end{align*}
\end{thm}
\begin{proof}
Let $\bar k \ge 4$ be the maximum of all thresholds for $k$ given in chapter \ref{construction} and let 
$\tilde \alpha=\tilde \alpha(\N) \le \frac{1}{4}$ \label{alphafuergamma}
be the upper bound for the Lipschitz constant given by Lemma \ref{bem4.2}. 
We set $\tilde k :=\max\{\bar k,\tilde C+1,\hat C\}$ \label{22.11.2013.1}
where the constants $\tilde C$ and $\hat C$ are fixed constants which will be set during this section\footnote{
$\tilde C$ is given in Lemma \ref{2.1.10;10}, $\hat C$ is given in Lemma \ref{lem4.4} V}.
Let $0 \le \alpha \le \tilde \alpha$. Let $\bar \varepsilon=\varepsilon(\n,\N,C_{0},\alpha) \le \alpha$ 
be the minimum of all thresholds for $\varepsilon$ given in chapter \ref{construction}.
We set $\tilde \varepsilon := \min\{\bar \varepsilon,(2C^{'}C_{1})^{-1}\}<1$\footnote{
$C^{'},C_{1}$ are given in Lemma \ref{lem4.3}}
\label{epsilonfuergamma} and 
assume that $k \ge \tilde k$ and $\eta \le \tilde \varepsilon^{\p}$. Now let $\varepsilon > 0$ with 
$\eta \le \varepsilon^{\p} \le \tilde \varepsilon^{\p}$. \label{3.12.2013.4}
For the rest of this section, we fix the parameters $k,\eta,\alpha,\varepsilon$ and mention that
they meet all requirements of the lemmas in Chapter \ref{construction}.

We start the proof of Theorem \ref{thm4.1} with several lemmas.
At first, we prove
\begin{lem}\label{24.2.10;17} There exists some constant $C=C(\n,\N,\p,C_{0})$ so that
	\[\sum_{i \in I_{12}}\int_{R_{i}\cap U_{10}} \int_{0}^{\frac{\diam R_{i}}{2}}  \gamma(q,t)^{\p}\frac{\dd t}{t} 
		\dd \cH^{\N}(q) \le C \varepsilon^{\p}.\]
\end{lem}
\begin{proof}
	Let $i \in I_{12}$, $q \in R_{i}$, $0 < t < \frac{\diam R_{i}}{2}$ and $u \in B(q,t) \cap P_{0} \subset 2R_i$.
	The function $A$ is in $C^{\infty}(2R_{i},P_{0}^{\perp})$ 
	(see definition of $A$ on page \pageref{04.11.2013.1}).
	Taylor's theorem implies 
	\[\inf_{a \in \mathcal{A}}d(A(u),a(u)) \le t^{2}\frac{C(\n,\N,C_{0}) \varepsilon}{\diam R_{i}}\]
	since the infimum over all affine functions cancels out the 
	linear part and the second order derivatives of the remainder can be estimated using 
	Lemma \ref{abschaetzungableitungvonA}.
	Now we have 
	\[\gamma(q,t)  \le \frac{\vol}{t} \sup_{u \in B(q,t)\cap P_{0}} \inf_{a \in \mathcal{A}} d(A(u),a(u))
		\le t \frac{C(\n,\N,C_{0})\varepsilon}{\diam R_{i}}.\]
	Hence, Lemma \ref{inneredisjunkt} (ii) implies the statement.
\end{proof}

The previous lemma implies that, due to Lemma \ref{inneredisjunkt} (ii), it remains to handle 
the two terms in the following sum to prove Theorem \ref{thm4.1}.
If $q_{1} \in R_{i}$, we get with Lemma \ref{rem3.10} that $\frac{D(q_{1})}{100} \le\frac{\diam R_{i}}{2}$
and, if $q_{2} \in \pi(\mathcal{Z})$, we obtain with Lemma \ref{rem3.8} $D(q_{2})= 0$.
Hence we conclude using Lemma \ref{inneredisjunkt} (ii)
\begin{align}
	& \ \ \  \sum_{i \in I_{12}}  \int_{R_{i} \cap U_{10}} \int_{\frac{\diam R_{i}}{2}}^{2}
	\gamma(q,t)^{\p} \frac{\dd t}{t} \dd \cH^{\N}(q) +
	\int_{\pi(\mathcal{Z})\cap U_{10}} \int_{0}^{2}  \gamma(q,t)^{\p} 
	 \frac{\dd t}{t} \dd \cH^{\N}(q)
	\nonumber \\
	& =  \int_{U_{10}}\int_{\frac{D(q)}{100}}^{2}
	\gamma(q,t)^{\p} \frac{\dd t}{t}  \dd \cH^{\N}(q). \label{10.1.2010;b}
\end{align}

In the following, we prove some estimate for $\gamma(q,t)$ 
where $q \in U_{10}$ and $ \frac{D(q)}{100} < t < 2$. To get this estimate, we need
the next lemma.
\begin{lem} \label{28.11.2013.1}
	For all $q \in U_{10}$ and for all $t$ with $\frac{D(q)}{100} < t < 2$,
	there exists some $\tilde X=\tilde X(q) \in F$ and some $T=T(t)> 0$ with
	\begin{align} \label{2.1.10;27}
		(\tilde X,T) &\in S,&
		d(\pi(\tilde X),q) &\le T &
		&\text{ and }&
		20t &\le T \le 200t.
	\end{align}
\end{lem}
\begin{proof}
	We have
	$ D(q) = \inf_{(X,s) \in S} (d(\pi(X),q) + s)$,
	and hence there exists some $(\tilde X, \tilde s) \in S$ \label{DefvonXSchlange} with
	$ d(\pi(\tilde X),q) + \tilde s \le D(q)+100t \le 200t$.
	We set $T:= \min\{40,200t\}$ which fulfils $20t \le T \le 200t$ as $t < 2$. 
	Using Lemma \ref{rem3.1} (i), (ii) and $200t \ge \tilde s$, we obtain $(\tilde X,T) \in S$.

	With $d(\pi(\tilde X),q) \le d(\pi(\tilde X),0) + d(0,q) \le 5 + 10$ we get $d(\pi(\tilde X),q) \le T$.
\end{proof}

Now let $q,t, \tilde X$ and $T$ \label{2.1.10;25} as in Lemma \ref{28.11.2013.1}.
Furthermore, let $X \in B(\tilde X,200t)\cap F \label{16.3.10;1}$. 
\newcommand{\Pqtx}{\hat P}
We choose some $\N$-dimensional plane named $\Pqtx=\Pqtx(q,t,X)$ with 
\begin{align} \label{2.1.10;30}
	\beta_{1;k}^{\Pqtx}(X,t) \le 2 \beta_{1;k}(X,t)
\end{align}
and define
\[ \mathcal{I}(q,t):= \left\{ i \in I_{12} \big| R_{i} \cap B(q,t) \neq \emptyset \right\}.\label{2.1.10;21}\]
With Lemma \ref{Riueberdeckung}, we have 
$(B(q,t)\cap P_{0}) \subset U_{12} \subset \pi(\mathcal{Z}) \cup \bigcup_{i \in I_{12}}R_{i}$. We set
\[K_{0}:= \int_{B(q,t)\cap \pi(\mathcal{Z})} \frac{d(u+A(u),\Pqtx)}{t^{\N+1}} \dd \cH^{\N}(u), \ \ \
	K_{i}:= \int_{B(q,t)\cap R_{i}} \frac{d(u+A(u),\Pqtx)}{t^{\N+1}} \dd \cH^{\N}(u)\]
and get with Lemma \ref{bem4.2} that
\begin{align}
	\gamma(q,t)
	& \le 3 \ K_{0} + 3 \sum_{i \in \mathcal{I}(q,t)} K_{i}. \label{2.1.10;1}
\end{align}
At first, we consider $K_{0}$.
\begin{lem} \label{2.1.10;10}
	There exists some constant $\tilde C>1$ so that
	\[ \int_{B(q,t) \cap \pi(\mathcal{Z})} d(u + A(u),\Pqtx) \dd \cH^{\N}(u)
		\le \int_{B(X,\tilde Ct)\cap \mathcal{Z}} d(x,\Pqtx) \dd \cH^{\N}(x).\]
\end{lem}
\begin{proof}
	Let $g: \pi(\mathcal{Z}) \rightarrow \mathcal{Z}, u \mapsto u+A(u)$. 
	This function is bijective, continuous ($A$ is $2 \alpha$-Lipschitz on $\pi(Z)$) and 
	$g^{-1}=\pi|_{\mathcal{Z}}$ is Lipschitz continuous with Lipschitz constant 1. 
	With $f(x) = d(x,\Pqtx)$ and $s=n$,
	we apply \cite[Lem. A.1]{Sch2012} and get
	\[\int_{B(q,t) \cap \pi(\mathcal{Z})} d(u + A(u),\Pqtx) \dd \cH^{\N}(u)
		 \le \int_{g(B(q,t) \cap \pi(\mathcal{Z}))} d(x,\Pqtx) \dd \cH^{\N}(x).\]
	Now it remains to show that there exists some constant $C$ so that
	$g(B(q,t) \cap \pi(\mathcal{Z})) \subset B(X,Ct)\cap \mathcal{Z}$. 
	Let $ x \in g(B(q,t) \cap \pi(\mathcal{Z}))$. This implies $x \in \mathcal{Z}$
	and so, using Lemma \ref{7.2.10;1}, we get $d(x)=0$.
	With \eqref{2.1.10;27}, we conclude 
	$d(\tilde{X}) \le d(\tilde X, \tilde X) + T \le 200t$,
	and we obtain with \eqref{2.1.10;27}
	$d(\pi(x),\pi(\tilde X)) \le 201t$.
	So, with Lemma \ref{lem3.9}, we have
	$d(x,\tilde X) \le 1602t$.
	We deduce with $\tilde C=1802$ that
	$d(x,X) \le d(x, \tilde X) + d(\tilde X,X) \le \tilde Ct$
	and so
	$g(B(q,t) \cap \pi(\mathcal{Z})) \subset B(X,\tilde Ct)\cap \mathcal{Z}$.
\end{proof}

\begin{lem} \label{2.1.10;11}
	There exists some constant $C=C(\n,\N,C_{0})>1$ so that
	\[ \int_{B(X,\tilde Ct)\cap \mathcal{Z}}d(x,\Pqtx)\dd \cH^{\N}(x)
		\le C \int_{B(X,(\tilde C +1)t)}d(x,\Pqtx)\dd \mu(x).\]
\end{lem}
\begin{proof}
	At first, we prove for an arbitrary ball $B$ with centre in $\mathcal{Z}$
	\begin{align} \label{23.2.10;a}
		\cH^{\N}(\mathcal{Z} \cap B) \le C(\n,\N,C_{0}) \mu(B).
	\end{align}
	With \cite[Dfn. 2.1]{Evans}, we get
	$\cH^{\N}(\mathcal{Z} \cap B) = \lim_{\tau \to 0}
		\cH^{\N}_{\tau}(\mathcal{Z} \cap B)$.
	Let $0<\tau_{0} < \min\left\{\frac{\diam B}{2},50\right\}$.
	We define
	$\mathcal{F}:= \left\{ B(x,s) | x \in \mathcal{Z}\cap B, s \le \tau_{0} \right\}$.
	With Besicovitch's covering theorem \cite[1.5.2, Thm. 2]{Evans}, there exist $N_{0}=N_{0}(\n)$ 
	countable families $\mathcal{F}_{j} \subset \mathcal{F}$, $j=1,...,N_{0}$, of disjoint balls 
	where the union of all those balls covers $\mathcal{Z} \cap B$.
	For every ball $\tilde B = B(x,s) \in \mathcal{F}_{j}$, we have $x \in \mathcal{Z}$ and hence, 
	using the definition of $\mathcal{Z}$ (see page \pageref{def3.2}), we deduce $h(x)=0$.
	With $h(x)=0 < s < 50$ and Lemma \ref{rem3.1} (i), we get
	$(x,s) \in S \subset S_{total}$ and so
	$\left(\frac{\diam \tilde B}{2}\right)^{\N} \le 2\frac{\mu(\tilde B)}{\delta}$.
	The centre of $B$ is also in $\mathcal{Z}$ and hence, analogously, we conclude
	$\left(\frac{\diam  B}{2}\right)^{\N} \le 2\frac{\mu(B)}{\delta}$.
	With (B) from page \pageref{Grundeigenschaften}, we get
	$\mu(2B) \le4^{\N} C_{0} \frac{2}{\delta} \mu(B)$.
	Since $x \in B$ and $s \le \tau_{0} < \frac{\diam B}{2} $, we obtain
	$\tilde B = B(x,s) \subset 2B$.
	Now, by definition of $\cH^{\N}_{\tau_{0}}$ \cite[Dfn. 2.1]{Evans} and 
	because $\delta =\delta(\n,\N)$
	(see \eqref{Wahlvondelta} on page \pageref{Wahlvondelta}),
	we deduce
	\begin{align*}
		\cH^{\N}_{\tau_{0}}(\mathcal{Z}\cap B) 
		& \le	 2\sum_{j=1}^{N_{0}} \sum_{\tilde B \in \mathcal{F}_{j}} \vol\frac{\mu(\tilde B)}{\delta} 
		 \le 2\frac{\vol}{\delta} \sum_{j=1}^{N_{0}} \mu(2B) 
		 \le C(\n,\N,C_{0}) \mu(B).
	\end{align*}
	So, with $\tau_{0} \rightarrow 0$, the inequality \eqref{23.2.10;a} is proven.
	
	Let $\tilde C$ be the constant from Lemma \ref{2.1.10;10}. For an arbitrary $0 < \sigma \le t$, we define
	\[ \mathcal{G}_{\sigma}:= \left\{ B(x,s) \Big| x \in \mathcal{Z} \cap B(X,\tilde Ct), s \le \sigma \right\}.\]
	With Besicovitch's covering theorem \cite[1.5.2, Thm. 2]{Evans}, there exist $N_{0}=N_{0}(\n)$ families  
	$\mathcal{G}_{\sigma, j} \subset \mathcal{G}_{\sigma}$ of disjoint balls, where $j= 1,..,N_{0}$ and those
	balls cover $\mathcal{Z} \cap B(X,\tilde Ct)$.
	We denote by $p_{B}$ the centre of the ball $B$ and conclude
	\begin{align*}
		& \hphantom{\stackrel{\hphantom{\eqref{23.2.10;a}}}{=}}
			 \int_{\mathcal{Z} \cap B(X,\tilde Ct)} d(x,\Pqtx)\dd \cH^{\N}(x) \\
		& \stackrel{\hphantom{\eqref{23.2.10;a}}}{\le}
			 \sum_{j=1}^{N_{0}} \sum_{B \in \mathcal{G}_{\sigma, j}} \int_{\mathcal{Z} \cap B} 
			\sigma + d(p_{B},\Pqtx)
			\dd \cH^{\N}(x)\\ \displaybreak[1]
		& \stackrel{\eqref{23.2.10;a}}{\le} 
			C(\n,\N,C_{0}) \sum_{j=1}^{N_{0}} \sum_{B \in \mathcal{G}_{\sigma, j}} \int_{B}
			\left(\sigma + d(p_{B},\Pqtx)\right) 
			\dd \mu(x)\\ \displaybreak[1]
		& \stackrel{\hphantom{\eqref{23.2.10;a}}}{\le}
			 C(\n,\N,C_{0}) \left( \mu(B(X,(\tilde C +1)t)) 2\sigma 
			 	+ \int_{B(X,(\tilde C +1)t)} d(x,\Pqtx) \dd \mu(x)\right). 
	\end{align*}
	With $ \sigma \rightarrow 0$, the assertion holds.
\end{proof}
With Lemma \ref{2.1.10;10} and Lemma \ref{2.1.10;11}, we get for $K_{0}$ using that 
$k \ge \tilde k \ge \tilde C+1$, where $\tilde k$ is defined on page \pageref{22.11.2013.1}
\begin{align}
	K_{0} 
	& \stackrel{\hphantom{\eqref{2.1.10;30}}}{\le} C(\n,\N,C_{0}) \ \beta_{1;k}^{\Pqtx}(X,t)
	 \stackrel{\eqref{2.1.10;30}}{\le} C(\n,\N,C_{0}) \ \beta_{1;k}(X,t).  \label{2.1.10;2}
\end{align}

To estimate $K_{i}$, we need the following lemma.
\begin{lem} \label{29.12.09;1}
	There exists some constant $C_{4}=C_{4}(\n,\N,C_{0}) > 1$ so that, for all $i \in I_{12}$ 
	and $ u \in R_{i}$, we have 
	$d(\pi_{P_{i}}(u+ A(u)),B_{i}) \le C_{4} \diam R_{i}$.
	We recall that $P_{i}$ is the $\N$-dimensional plane, which is,
	in the sense of Definition \ref{12.07.13.1}, associated to the ball 
	$B(X_{i},t_{i})=B_{i}$ given by Lemma \ref{vor3.12} (cf. Definition \ref{3.12.2013.10}).
\end{lem}
\begin{proof}
	For every $i \in I_{12}\subset I$, we have with Lemma \ref{vor3.12} that $B_i=B(X_i,t_i)$ 
	and $(X_i,t_i) \in S \subset S_{total}$. Hence we can use Lemma \ref{nachlem2.6} 
	($\sigma=2 \varepsilon$, $x=X_{i}$, $t=t_{i}$, $\lambda = \frac{\delta}{2}$, $P=P_{i}$)
	to get some $y \in 2B_i \cap P_{i}$, where $P_i=P_{(X_i,t_i)}$.
	We obtain with Lemma \ref{6.9.11.1} ($P_{1}=P_{j}$, $P_{2}=P_{0}$), $\alpha \le \tilde \alpha < \frac{1}{2}$ 
	($\tilde \alpha$ is defined on page \pageref{alphafuergamma})
	and Lemma \ref{vor3.12}
	\begin{align*}
		d(u+A_{i}(u),y) 
		& \le \frac{1}{1-\alpha}d(u,\pi(y))
		 < 2[d(u, \pi(X_i))+ d(\pi(X_i),\pi(y))]
		 \le C \diam R_i.
	\end{align*}
	Moreover, with Lemma \ref{lem3.12} (iv) and $\varepsilon \le \tilde \varepsilon \le 1$ 
	($\tilde \varepsilon$ is defined on page \pageref{epsilonfuergamma}), we get
	\[d(\pi_{P_i}(u+A(u)),u+A_i(u)) \le d(u+A(u),u+A_i(u)) \le C\diam R_i\]
	for some $C=C(\n,\N,C_{0})$. Using these estimates, $u+A_i(u)=\pi_{P_i}(u+A_i(u))$ and triangle inequality,
	we obtain the assertion.
\end{proof}

Now, with Lemma \ref{29.12.09;1} and $K_{i}$ from \eqref{2.1.10;1}, we obtain for $i \in \mathcal{I}(q,t) \subset I_{12}$
\begin{align}
	K_{i} 
	& \stackrel{\hphantom{\text{L. \ref{29.12.09;1}}}}{\le} 
		\frac{1}{t^\N}  \int_{B(q,t)\cap R_{i}} \frac{d(u+A(u),P_{i})}{t}\dd \cH^{\N}(u)\nonumber  \\ 
	&\hphantom{\stackrel{\text{L. \ref{29.12.09;1}}}{\le}}
		 + \frac{1}{t^\N} \sup \left\{ \frac{d(\pi_{P_{i}}(v + A(v)),\Pqtx)}{t} \Big| v \in B(q,t) \cap R_{i} \right\}
		\cH^{\N}(B(q,t)\cap R_{i})\nonumber \displaybreak[1] \\ 
	& \stackrel{\text{L. \ref{29.12.09;1}}}{\le} \frac{1}{t^\N}  
		\int_{B(q,t)\cap R_{i}} \frac{d(u+A(u),P_{i})}{t}\dd \cH^{\N}(u)\nonumber  \\
	& \hphantom{\stackrel{\text{L. \ref{29.12.09;1}}}{\le}}
		 + \vol \left( \frac{\diam R_{i}}{t}\right)^\N \sup \left\{ \frac{d(w,\Pqtx)}{t} 
		\Big| w \in P_{i}, d(w,B_{i}) \le C_{4} \diam R_{i} \right\}.
		\label{2.1.10;3}
\end{align}
Since $P_{i}$ is the graph of $A_{i}$, we get for any $u \in B(q,t)\cap R_{i}$ with Lemma \ref{lem3.12} (iv) 
that there exists some $\bar C= \bar C(\n,\N,C_{0})$ with
\begin{align*}
	d(u+ A(u),P_{i}) & \le d(u + A(u),u+A_{i}(u))
	 = d(A(u),A_{i}(u)) 
	 \le \bar C\varepsilon \diam R_{i},
\end{align*}
and so, using Lemma \ref{24.10.12.1},
\begin{align}
	\frac{1}{t^\N}  \int_{B(q,t)\cap R_{i}} \frac{d(u+A(u),P_{i})}{t}\dd \cH^{\N}(u) 
		& \le \varepsilon \ C(\n,\N,C_{0}) \left( \frac{\diam R_{i}}{t} \right)^{\N+1}. \label{2.1.10;4}
\end{align}

\begin{lem} \label{lem4.3}
	There exists some constant $C=C(\n,\N,C_{0})$ so that for all $i \in \mathcal{I}(q,t)$
	\begin{align*}
		 & \ \ \ \sup \left\{ \frac{d(w,\Pqtx)}{t} \Big| w \in P_{i}, d(w,B_{i}) 
			\le C_{4} \diam R_{i}  \right\} \\
		 &\le C \varepsilon \frac{\diam R_{i}}{t} + C \frac{1}{t} \left( \frac{1}{(\diam R_{i})^\N}
		 \int_{2B_{i}} d(z,\Pqtx)^{\frac{1}{3}} \dd \mu(z) \right)^{3}.
	\end{align*}
\end{lem}
\begin{proof} 
	Let $i \in \mathcal{I}(q,t)$.
	Due to the construction of $B_{i}=B(X_{i},t_{i})$ (see Lemma \ref{vor3.12}), we have 
	$(X_{i},t_{i}) \in S \subset S_{total}$
	and so
	$\delta(X_{i},t_{i}) \ge \frac{\delta}{2}$.
	With Corollary \ref{04.09.12.1} ($\lambda=\frac{\delta}{2}$, $B(x,t)=B(X_{i},t_{i})$, $\Upsilon=\R^{\n}$), 
	there exist constants $C_{1}=C_{1}(\n,\N,C_{0})>3$, $C_{2}=C_{2}(\n,\N,C_{0})>1$ 
	and some $(\N,10\N\frac{t_i}{C_1})$-simplex $T=\Delta(x_0,\dots,x_{\N}) \in F \cap B_i$
	with
	\begin{align}\label{11.03.2013.1}
		\mu\left( B\left(x_\kappa,\textstyle{\frac{t_i}{C_1}}\right) \cap B_{i}\right) 
		\ge \textstyle{\frac{t_{i}^{\N}}{C_{2}}} \ \ \text{and} \ \
		B\left(x_{\kappa},\frac{t_i}{C_1}\right) \subset 2B_{i} \subset kB_{i}=B(X_{i},kt_{i}).
	\end{align}
	for all $\kappa = 0, \dots, \N$ and we used that $C_{1} > 3$ and $k \ge \tilde k \ge 2$ 
	($\tilde k$ is defined on page \pageref{22.11.2013.1})., we have
	We set $C^{'} := 400C_{2}$,
	$\tilde B_{\kappa}:=B\left(x_{\kappa},\frac{t_i}{C_1}\right)$
	and define for all $\kappa = 0,\dots,\N$
	\begin{align} \label{29.12.09;e}
		Z_{\kappa}:= \left\{ z \in \tilde B_{\kappa} \cap F \big| d(z,P_{i}) \le C^{'} \varepsilon \diam R_{i} \right\}.
	\end{align}
	We have $(X_{i},t_{i}) \in S_{total}$ and hence $\beta_{1;k}^{P_{i}}(X_i,t_i) \le 2 \varepsilon$. Using this
	and Lemma \ref{vor3.12}, we obtain with Chebyshev's inequality
	\[\mu(\tilde B_{\kappa}\setminus Z_{\kappa}) 
		 < \frac{t_{i}^{\N+1}}{C^{'}\varepsilon \diam R_{i}} \beta_{1;k}^{P_{i}}(X_i,t_i)
		\le \frac{t_{i}^{\N+1} \ 100}{C^{'}\varepsilon t_{i}} 2\varepsilon
		 = \frac{t_{i}^\N}{2C_{2}}.\]
	Using Lemma \ref{vor3.12} again, we get
	\begin{align}
		\mu(Z_{\kappa}) 
		&\stackrel{\hphantom{\eqref{11.03.2013.1}}}{\ge}
			 \mu(\tilde B_{\kappa}) - \mu(\tilde B_{\kappa}\setminus Z_{\kappa})  
		 \stackrel{\eqref{11.03.2013.1}}{\ge} \frac{t_i^{\N}}{C_{2}} 
			- \frac{t_{i}^{\N}}{2C_{2}} 
		 \stackrel{\hphantom{\eqref{11.03.2013.1}}}{=} \frac{t_{i}^{\N}}{2C_{2}} 
		 \stackrel{\hphantom{\eqref{11.03.2013.1}}}{\ge} \frac{\diam R_{i}^{\N}}{2^{\N+1}C_{2}} >0. \label{29.12.09;g}
	\end{align}
	For all $\kappa \in \{0,\dots,\N\}$, let $z_{\kappa} \in Z_{\kappa} \subset \tilde B_{\kappa}$ 
	and set $y_{\kappa}:=\pi_{P_i}(z_{\kappa})$. 
	Since $\varepsilon \le \tilde \varepsilon \le \frac{1}{2C^{'} C_{1}}$ ($\tilde \varepsilon$ was chosen on
	page \pageref{epsilonfuergamma}), we deduce 
	\begin{align*} 
		d(y_{\kappa},x_{\kappa}) 
		& \le d(y_{\kappa},z_{\kappa}) + d(z_{\kappa},x_{\kappa})
		 \le d(z_{\kappa},P_i) + \frac{t_i}{C_1}
		 \stackrel{\eqref{29.12.09;e}}{\le} C^{'}\varepsilon \diam R_i + \frac{t_i}{C_1}
		 \le 2\frac{t_i}{C_1}.
	\end{align*}
	Due to Lemma \ref{17.11.11.2},
	the simplex $S=\Delta(y_0, \dots,y_{\N})$ is an $(\N,6\N\frac{t_i}{C_1})$-simplex and, 
	using the triangle inequality, we obtain $S \subset 2B_i$.
	Now, with Lemma \ref{30.05.12.1}, ($C=\frac{C_1}{6\N}$, $\hat C=2$, $t=t_{i}$, $m=n$, $x=X_{i}$) 
	there exists some orthonormal basis $(o_1, \dots, o_\N)$ of 
	$P_{i}-y_{0} $ and there exists $\gamma_{l,r} \in \R$ with
	$o_l = \sum_{r=1}^{l} \gamma_{l,r} (y_r-y_0)$
	and
	$|\gamma_{l,r}| \le \left(\frac{2C_1}{3}\right)^{\N} \frac{C_1}{6\N t_i}$
	for all $1 \le l \le \N$ and $1 \le r \le l$.

	Now let $ w \in P_{i}$ with $d(w,B_{i}) \le C_{4} \diam R_{i}$. 
	We obtain
	\begin{align} \label{10.12.2013.1}
		w -y_0 
		& = \sum_{\kappa = 1}^{\N} \langle w-y_0,o_{\kappa} \rangle o_{\kappa} 
		 = \sum_{\kappa = 1}^{\N} \langle w-y_0,o_{\kappa} \rangle  \sum_{r=1}^{\kappa} \gamma_{\kappa,r} (y_r-y_0)
	\end{align}
	and so, with Remark \ref{12.4.11} ($b=w$, $P=\Pqtx$)
	and $|w-y_{0}|\le d(w,B_i) +  \diam B_i + d(B_i,y_0) \le Ct_i$, we get
	\begin{align}
		d(w,\Pqtx) 
		& \stackrel{\eqref{10.12.2013.1}}{\le} 
			\N C C_{1}^{\N+1} \sum_{r=1}^{\N} \left(d(y_r,z_r) + d(z_r,\Pqtx) \right) \nonumber \\\displaybreak[1]
		& \stackrel{\eqref{29.12.09;e}}{\le}  
			\N^{2} CC_1^{\N+1}C^{'} \varepsilon \diam R_{i}
			+ \N CC_1^{\N+1} \sum_{r=0}^{\N} d(z_{r},\Pqtx). \label{29.12.09;f}
	\end{align}
	The previous results are valid for arbitrary $z_{\kappa} \in Z_{\kappa}$, hence we get
	\begin{align*}
		& \hphantom{\eqref{29.12.09;g} \eqref{11.03.2013.1} }
			d(w,\Pqtx) -\N^{2} CC_1^{\N+1}C^{'} \varepsilon \diam R_{i}\\
		& \stackrel{\substack{\hphantom{\eqref{29.12.09;g} \eqref{11.03.2013.1}} \\ \eqref{29.12.09;f}}}{\le} 
			\left(\frac{1}{\prod_{r=0}^{\N} \mu(Z_r)}  \int_{Z_{0}} \dots \int_{Z_{\N}} 
			\left(\N CC_1^{\N+1} \sum_{r=0}^{\N} d(z_{r},\Pqtx) \right)^{\frac{1}{3}} \
			\dd \mu(z_{\N}) \dots \dd \mu(z_{0})\right)^{3} \\ \displaybreak[1]
		& \stackrel{\hphantom{\eqref{29.12.09;g} \eqref{11.03.2013.1}}}{\le} \N CC_1^{\N+1} 
			\left( \sum_{r=0}^{\N} \frac{1}{\mu(Z_r)}  \int_{Z_{r}} 
			 d(z_{r},\Pqtx)^{\frac{1}{3}} \
			\dd \mu(z_{r})\right)^{3} \\ \displaybreak[1]
		& \stackrel{\eqref{29.12.09;g} \eqref{11.03.2013.1}}{\le} \N CC_1^{\N+1}
			\left(  \frac{2^{\N+1}C_{2}}{\diam R_{i}^{\N}} \int_{2B_{i}}
			d(z,\Pqtx)^{\frac{1}{3}} \dd \mu(z) \right)^{3},
	\end{align*}
	where we used that the sets $Z_{r}$ are disjoint.
	Since $w \in P_{i}$ was arbitrarily chosen with $d(w,B_{i}) \le C_{4} \diam R_{i}$, we get the statement.
\end{proof}
\begin{lem} \label{lem4.4}
	There exists some constant $C=C(\N,C_{0})$ so that
	\begin{align*}
		\sum_{i \in \mathcal{I}(q,t)} \left(\frac{\diam R_{i}}{t}\right)^{\N} \frac{1}{t} 
			\left( \frac{1}{(\diam R_{i})^{\N}} \int_{2B_{i}}
			d(z,\Pqtx)^{\frac{1}{3}} \dd \mu(z) \right)^{3} 
		& \le C \beta_{1;k}(X,t).
	\end{align*}
\end{lem}
\begin{proof}
	Let $i \in \mathcal{I}(q,t)$ ($\mathcal{I}(q,t)$ is defined on page \pageref{2.1.10;21})
	and $x \in 2B_{i}$. 
	We define \label{28.2.10;2}
	\[ J(i) := \left\{ j \in \mathcal{I}(q,t) \big| \diam B_{j} \le 
			\diam B_{i}, 2B_{i} \cap 2B_{j} \neq \emptyset \right\}, \ \ \text{and} \ \
		\Xi_{i}(x) := \sum_{j \in J(i)} \chi_{_{2B_{j}}}(x).\]
	At first, we prove some intermediate results:\\
	I.\,	For all $i \in \mathcal{I}(q,t)$, we have
		$\int_{2B_{i}}\Xi_{i}(x) \dd \mu(x) \le C(\N,C_{0}) (\diam R_{i})^{\N}$.
		This implies that $\Xi_{i}(x)< \infty$ for $\mu$-almost all $x \in 2B_{i}$.
		\begin{proof}
			Let $i \in \mathcal{I}(q,t)$ and $j \in J(i)$. With Lemma \ref{vor3.12} applied to $j$ 
			and the definition of $J(i)$, we deduce
			$\diam R_{j} \le 200 \diam R_{i}$. Using Lemma \ref{vor3.12} and $j \in J(i)$, we get
 			$ d(R_{i},R_{j}) \le C \diam R_{i}$.
			This implies for some large enough constant $C>1$ that $R_{j} \subset C R_{i}$.
			Since the cubes $ \mathring{R}_{j}$ are disjoint
			(see Lemma \ref{inneredisjunkt} (ii)), 
			we get with Lemma \ref{24.10.12.1}
			\begin{align*}
				\sum_{j\in J(i)} (\diam R_{j})^{\N} 
				& = \sum_{j\in J(i)} (\sqrt{\N})^{\N} \cH^{\N}(R_{j})
				 \le (\sqrt{\N})^{\N} \cH^{\N}\big( C R_{i}\big)
				 = C(\N) (\diam R_{i})^{\N}.
			\end{align*}
			In the following, we apply Fatou's Lemma \cite[1.3, Thm.1]{Evans} to 
			interchange the integration with the summation.
			With (B) from page \pageref{Grundeigenschaften}
			and Lemma \ref{vor3.12}, we obtain
			\[\int_{2B_{i}}\Xi_{i}(x) \dd \mu(x) \le  \sum_{j \in J(i)} \mu(2B_{j})
				 \stackrel{\text{(B)}}{\le} C(\N,C_{0}) \sum_{j\in J(i)} (\diam R_{j})^{\N}
				 \le C(\N,C_{0}) (\diam R_{i})^{\N}.\]
		\end{proof} \noindent
	II.\, Let $x \in \mathbb{R}^{\n}$ and $m \in \mathbb{N}$. There exists some $C=C(\N)>1$ with
		$\sum_{\genfrac{}{}{0pt}{}{i \in  \mathcal{I}(q,t) }{ \Xi_{i}(x)=m}} \chi_{_{2B_{i}}}(x) \le C$.
		\begin{proof}
			Let $l,o \in \mathcal{I}(q,t)$ with $x \in 2B_{l} \cap 2B_{o}$ and
			$\Xi_{l}(x) = m = \Xi_{o}(x)$.
			Without loss of generality, we have $\diam B_{l} \le \diam B_{o}$.

			Assume that $ \diam B_{l} < \diam B_{o}$.
			We define
			$J(l,x) := \left\{ \iota \in  J(l) \big|  x \in 2B_{\iota}\right\}$.
			Let $j \in J(l,x)$. By definition of $J(l)$, we get
			$ \diam B_{j} \le \diam B_{l}  < \diam B_{o}$ 
			and $x \in 2B_{j}$. 
			Since $x \in 2B_{o}$, it follows $2B_{o} \cap 2B_{j} \neq \emptyset$ and,
			because $\diam B_{j} < \diam B_{o}$, 
			we get $j \in J(o,x)$. 
			Furthermore, we have $o \in J(o,x)$, but $o \notin J(l,x)$ because by our assumption
			we have $ \diam B_{l} < \diam B_{o}$. So we get
			$J(l,x) \subsetneq J(o,x)$.
			Now we obtain a contradiction
			\[ m=\Xi_{l}(x)= \sum_{j \in J(l)} \chi_{_{2B_{j}}}(x)=\sum_{j \in J(l,x)} \chi_{_{2B_{j}}}(x)<
				\sum_{j \in J(o,x)} \chi_{_{2B_{j}}}(x)=\Xi_{o}(x)=m.\]
			Hence there exists some $\lambda=\lambda(x,m) \in (0,\infty)$ so that,
			for $l \in \mathcal{I}(q,t)$ with $x \in 2B_{l}$ and
			$\Xi_{l}(x) = m$, we have 
			$\diam B_{l} = \lambda$, 
			and, we obtain with Lemma \ref{vor3.12} that
			$\lambda \le 200 \diam R_{l} \le 200 \lambda$ and $d(R_{l},\pi(B_{l})) \le 100 \lambda$.
			Using
			$d(R_{l},\pi(x)) \le d(R_{l},\pi(B_{l}))+2\diam B_{l} \le 102 \lambda$,
			we get  $R_{l} \subset B(\pi(x),103 \lambda) \cap P_{0}$.
			With Lemma \ref{24.10.12.1}, we have 
			$\cH^{\N}(R_{l}) \ge (\sqrt{\N})^{-\N}(\textstyle{\frac{1}{200}}\lambda)^{\N}$
			and, according to Lemma \ref{inneredisjunkt} (ii) the cubes $R_{l}$ 
			have disjoint interior. This implies that there exists some constant $C(\N)$ so that
			there are at most $C(\N)$ indices $l \in  \mathcal{I}(q,t)$ with $\Xi_{l}(x)=m$ and $x \in 2B_{l}$.
			This implies the assertion.
		\end{proof} \noindent
	III.\, We have $i \in J(i)$ and so $\Xi_{i}(x) \neq 0$ for all $x \in 2B_{i}$. 
		Hence, with $x \in \mathbb{R}^{\n}$, the term
		\[ \chi_{_{2B_{i}}}(x)\Xi_{i}(x)^{-2}:= \begin{cases}
								\Xi_{i}(x)^{-2} &\text{if } x \in 2B_{i}\\
								0		& \text{otherwise}
							\end{cases} \]
		is well-defined.
		Now there exists some constant $C(\N)$ so that, for all $x \in \mathbb{R}^{\n}$, we get
		\begin{align*}
				\sum_{i \in \mathcal{I}(q,t)} \chi_{_{2B_{i}}}(x)\Xi_{i}(x)^{-2} 
				& = \sum_{m=1}^{\infty} \sum_{
					\genfrac{}{}{0pt}{}{i \in  \mathcal{I}(q,t) }{ \Xi_{i}(x)=m}} \chi_{_{2B_{i}}}(x)
				\frac{1}{m^{2}} \stackrel{\text{II}}{\le} C(\N).
		\end{align*} \noindent
	IV.\, Let $i \in \mathcal{I}(q,t)$. Since $i \in J(i)$, we have $\Xi_{i}(x) \neq 0$ for $x \in 2B_{i}$. 
		We obtain with H\"older's inequality
		\begin{align*}
			& \ \ \ \left( \frac{1}{(\diam R_{i})^{\N}} \int_{2B_{i}} d(z,\Pqtx)^{\frac{1}{3}} \Xi_{i}(z)^{\frac{-2}{3}}
				\Xi_{i}(z)^{\frac{2}{3}} \dd \mu(z) \right)^{3}\\ \displaybreak[1]
			& \stackrel{\text{I}}{\le} C(\N,C_{0}) \frac{1}{(\diam R_{i})^{\N}} 
				\int_{2B_{i}} d(z,\Pqtx) \Xi_{i}(z)^{-2}  \dd \mu(z).
		\end{align*}\noindent
	V.\, We have
		\[ \frac{1}{t^{\N+1}} \int_{\bigcup_{i\in \mathcal{I}(q,t)}2B_{i}} d(z,\Pqtx) 
			\dd \mu(z) \le 2 \beta_{1;k}(X,t),\]
		where $X \in B(\tilde X(q),200t)$ (cf. page \pageref{2.1.10;25}).
		\begin{proof}
			At first, we prove that there exists some constant $\hat C>1$ so that 
			for $i \in \mathcal{I}(q,t)$ we have
			$ 2B_{i} \subset B(X,\hat Ct)$. 
			Let $i \in \mathcal{I}(q,t)$. 
			By definition of $\mathcal{I}(q,t)$ (see page \pageref{2.1.10;21}), we obtain
			$R_{i} \cap B(q,t) \neq \emptyset$.
			Let $\tilde{u} \in R_{i} \cap B(q,t)$.
			Since $ \frac{D(q)}{100} < t$ (see page \pageref{2.1.10;25}), we get,
			using the triangle inequality,
			$D(\tilde{u})  \le D(q)+ d(q,\tilde{u}) < 101t$.
			It follows with Lemma \ref{lem3.11} (i) that
			\begin{align} \label{2.1.10;c}
				\diam R_{i} \le \textstyle{\frac{1}{10}} D(\tilde u) < 11t.
			\end{align}
			With Lemma \ref{vor3.12} and \eqref{2.1.10;27} from page \pageref{2.1.10;27}, we
			get ($X \in B(\tilde{X},200t)$, see page \pageref{16.3.10;1})
			\begin{align} \label{2.1.10;d}
				d(\pi(B_{i}),\pi(X)) &\stackrel{\phantom{\eqref{2.1.10;27}}}{\le} d(\pi(B_{i}),\tilde u) + d(\tilde u,q)+ 
					d(q,\pi(\tilde X)) + d(\pi(\tilde X),\pi(X))\nonumber  \\ 
				& \stackrel{\eqref{2.1.10;27}}{\le} 
					d(\pi(B_{i}),R_{i}) + \diam R_{i} + t + 200t + d(\tilde X,X) \stackrel{\eqref{2.1.10;c}}{\le} C t.
			\end{align}
			Now let $x \in 2B_{i}=B(X_{i},2t_{i})$. Since $(X_{i},t_{i}) \in S$,
			using Lemma \ref{vor3.12} and \eqref{2.1.10;c}, we get
			$d(x) < 4400t$.
			Due to $X \in B(\tilde X,200t) \cap F$ and \eqref{2.1.10;27}, we deduce $d(X) \le 400 t$.
			With Lemma \ref{vor3.12} and estimates \eqref{2.1.10;c} and \eqref{2.1.10;d}, we obtain
			with triangle inequality $d(\pi(x),\pi(X)) \le Ct$.
			Now there exists some constant $\hat C > 1$ so that, we get with Lemma \ref{lem3.9}
			$d(x,X) \le \hat C t$.
			All in all we have proven that, for all $i \in \mathcal{I}(q,t)$, we have
			$2B_{i} \subset B(X,\hat C t)$.
			Since $k \ge \tilde k \ge \hat C$ (see page \pageref{22.11.2013.1}), we get the assertion
			with condition \eqref{2.1.10;30} from page \pageref{2.1.10;30}.
		\end{proof}
	Now, Lemma \ref{lem4.4} can be proven by applying IV, III, and V and using  
	the monotone convergence theorem \cite[1.3, Thm. 2]{Evans} to interchange the summation and the integration
\end{proof}

Now we can give some estimate for $\gamma(q,t)$, where $q \in U_{10}$ and $\frac{D(q)}{100} < t < 2.$
Using the inequalities \eqref{2.1.10;1}, \eqref{2.1.10;2}, \eqref{2.1.10;3}, \eqref{2.1.10;4}, 
Lemma \ref{lem4.3} and Lemma \ref{lem4.4}, we get using $T \le 200t$ 
(cf. Lemma \ref{28.11.2013.1}) for every 
$X \in B(\tilde X,T) \cap F \subset B(\tilde X,200t) \cap F$ 
\[\gamma(q,t) \le C(\n,\N,C_{0})\ \beta_{1;k}(X,t)  + C(\n,\N,C_{0}) \
		 \varepsilon \sum_{i \in \mathcal{I}(q,t)}  \left( \frac{\diam R_{i}}{t} \right)^{\N+1}.\] 
With Lemma \ref{28.11.2013.1}, we get $(\tilde X, T) \in S \subset S_{total}$ and $20t \le T \le 200t$. Using this,
the previous estimate, the definition of $\delta=\delta(\N)$ on page \pageref{Wahlvondelta} and
(B) from page \pageref{Grundeigenschaften}, we get
\begin{align*}
	\gamma(q,t)^{\p} 
	& \stackrel{\hphantom{\text{(B)}}}{\le} \frac{2}{\delta T^{\N}}\int_{B(\tilde X,T)} \gamma(q,t)^{\p} \dd  \mu(X)\\ \displaybreak[1]
	& \stackrel{\hphantom{\text{(B)}}}{\le} C\frac{1}{t^{\N}}\int_{B(\tilde X,200t)} 
	\beta_{1;k}(X,t)^{\p}\dd  \mu(X) 
	 +C\varepsilon^{\p}\left( \sum_{i \in \mathcal{I}(q,t)} \left( \frac{\diam R_{i}}{t} \right)^{\N+1} \right)^{\p} ,
\end{align*}
where $C=C(\n,\N,\p,C_{0})$.
We recall that for every $q \in U_{10}$ there exists some $\tilde X=\tilde X(q)$ (cf. Lemma \ref{28.11.2013.1}) such that
the previous inequality is valid. This implies
\begin{align}
	\int_{U_{10}} \int_{\frac{D(q)}{100}}^{2} \gamma(q,t)^{\p} \frac{\dd t}{t} \dd \cH^{\N}(q) 
	& \le C(\n,\N,\p,C_{0}) \ a + C(\n,\N,\p,C_{0}) \ \varepsilon^{\p} \ b, \label{10.1.2010;c}
\end{align}
where
\[a:= \int_{U_{10}} \int_{\frac{D(q)}{100}}^{2} \frac{1}{t^{\N}}\int_{B(\tilde X(q),200t)} 
		\beta_{1;k}(X,t)^{\p}\dd \mu(X) \frac{\dd t}{t} \dd \cH^{\N}(q),\]
\[b:= \int_{U_{10}} \int_{\frac{D(q)}{100}}^{2} \left( \sum_{i \in \mathcal{I}(q,t)} 
		\left( \frac{\diam R_{i}}{t} \right)^{\N+1} \right)^{\p} \frac{\dd t}{t} \dd \cH^{\N}(q).\]
To estimate $a$ and $b$, we need the following lemma.
\begin{lem}\label{06.12.2013.1}
	Let $q \in U_{10}$, $\frac{D(q)}{100} \le t \le 2$ and $X \in B(\tilde X(q),200t) \cap F$,
	where $\tilde X(q)$ is given by Lemma \ref{28.11.2013.1} on page \pageref{DefvonXSchlange}.
	Then 
	$d(\pi(X),q) \le 400t $
	and there exists some $\tilde \lambda  = \tilde \lambda(\n,\N,C_{0})>0$ so that, with $k_{0}=401$, we have
	$\tilde{\delta}_{k_{0}} (B(X,t))=\sup_{y \in B(X,k_{0}t)} \frac{\mu(B(y,t))}{t^{\N}}  \ge \tilde \lambda$,
	where $\tilde{\delta}_{k_{0}} (B(X,t))$ was defined on page \pageref{Definitionvondeltaschlange}. 
	Furthermore, there holds for all $i \in \mathcal{I}(q,t)$ that
	\begin{align}\label{06.12.2013.2}
		d(q,R_{i}) &\le t, &\diam R_{i} &< 11t,
	\end{align}
	and there exists some constant $C=C(\N)$ with
	\begin{align}
		\sum_{i \in \mathcal{I}(q,t)}\left( \frac{\diam R_{i}}{t} \right)^{\N+1} &\le C, \ \ \ \ \ \ \ \ 
		\sum_{i \in I_{12}} (\diam R_{i})^{\N} \le C. \label{2.1.10;101}
	\end{align}
\end{lem}
\begin{proof}
	Let $q \in U_{10}$, $\frac{D(q)}{100} \le t \le 2$ and $X \in B(\tilde X(q),200t) \cap F$.
	We have $d(X, \tilde X(q)) \le 200t$ and, with \eqref{2.1.10;27}, we get $d(\pi(\tilde X(q)),q) \le 200t$. 
	This implies $d(\pi(X),q) \le 400t$ by using triangle inequality.
	With \eqref{2.1.10;27}, we obtain
	$(\tilde X(q), T) \in S \subset S_{total}$ and, by definition of $S_{total}$, we conclude
	$\delta(B(\tilde X(q), T)) \ge \frac{\delta}{2}$. We have
	$B(\tilde X(q), T) \subset B(X,400t)$ and hence with \eqref{2.1.10;27} we get
	$\delta(B(X,400t)) \ge \frac{\delta}{2\cdot20^{\N}}$.
	Applying Corollary \ref{04.09.12.1} (ii) with $\lambda = \frac{\delta}{2\cdot 20 ^{\N}}$ 
	on $B(X,400t)$, we get constants $C_{1}=C_{1}(\n,\N,C_{0})$, $C_{2}=C_{2}(\n,\N,C_{0})$ and
	in particular one ball $B(x,s)$ with $ s = \frac{400t}{C_{1}}$ and
	\begin{align} \label{24.2.10;10}
		\mu(B(x,s) \cap B(X,400t)) \ge {\textstyle \frac{(400t)^{\N}}{C_{2}}}.
	\end{align}
	We have $\delta \le \frac{2}{50^{\N}}$ (cf. \eqref{Wahlvondelta} on page \pageref{Wahlvondelta}), 
	and Lemma \ref{lem2.3} gives us $C_{1} > 400.$
	That yields $s < t$. From \eqref{24.2.10;10}, we get $B(x,s) \cap B(X,400t) \neq \emptyset$
	which implies $d(x,X) < 401t$ and with \eqref{24.2.10;10} we get
	$\sup_{y \in B(X,401t)} \delta(B(y,t)) \ge \frac{400^{\N}}{C_{2}}=:\tilde \lambda$.
	Let $i \in \mathcal{I}(q,t)$. 
	Due to the definition of $\mathcal{I}(q,t)$ (see page \pageref{2.1.10;21}), we have $d(q,R_{i}) \le t$
	and we can choose some $\tilde u \in R_{i} \cap B(q,t)$.
	With Lemma \ref{lem3.11} (i), we obtain $ 10\diam R_{i} \le (D(q) + d(q,\tilde u)) < 11t$.
	The intervals $ R_{i}$ have disjoint interior (see Lemma \ref{inneredisjunkt} (ii)) and, from
	$R_{i} \cap B(q,t) \neq \emptyset$ for all $i \in \mathcal{I}(q,t)$, we get 
	$R_i \subset B(q,12t)$. With Lemma \ref{24.10.12.1}, this implies
	\[\sum_{i \in \mathcal{I}(q,t)} \left( \frac{\diam R_{i}}{t} \right)^{\N+1} 
		\stackrel{\eqref{06.12.2013.2}}{<} \frac{11}{t^{\N}} \sum_{i \in \mathcal{I}(q,t)} (\diam R_{i})^\N 
		= \frac{11}{t^{\N}} \sum_{i \in \mathcal{I}(q,t)} (\sqrt{\N})^{\N} \cH^{\N}(R_i)\\ \nonumber
		=C(\N).\]
	Now let $i \in I_{12}$. We have $R_{i} \cap B(0,12) \neq \emptyset$. 
	If $(Y,r) \in S \subset S_{total}$, we get $Y \in F \subset B(0,5)$ (cf. (A) on page \pageref{3.12.2013.3})
	and hence we obtain $d(\pi(Y),0) \le 5$ as well as $r\le 50$.
	With $\tilde v \in R_{i} \cap B(0,12)$ and
	Lemma \ref{lem3.11} (i), we get
	\[ \diam R_{i} \le \frac{1}{10} D(\tilde v) 
		= \frac{1}{10} \inf_{(Y,r) \in S}(d(\pi(Y),\tilde v) + r) 
		\le \frac{1}{10} (5+12 + 50) < 7. \]
	Hence, for all $i \in I_{12}$, we have $R_i \subset B(0,19)$ and the cubes $R_{i}$ have disjoint interior
	(cf. Lemma \ref{inneredisjunkt} (ii)). With Lemma \ref{24.10.12.1}, we deduce
	$ \sum_{ i \in I_{12}} (\diam R_{i})^{\N} = C(\N)$.
\end{proof}
To control the terms $a$ and $b$ we will use 
Fubini's Theorem \cite[1.4, Thm. 1]{Evans}, in the following abbreviated by (F).
Now, using Lemma \ref{06.12.2013.1} and Corollary \ref{thm2.4} ($\lambda = \tilde \lambda$, $k_{0}=401$), we conclude
\begin{align*}
	a & \stackrel{\text{(F)}}{\le} 
		\int_{F} \int_{0}^{2} \frac{1}{t^{\N}} \int_{U_{10}} 
		\Eins_{\left\{d(\pi(X),q)\le 400t\right\}} \dd \cH^{\N}(q) \
		\Eins_{\left\{ \tilde{\delta}_{k_{0}}(B(X,t)) \ge \tilde \lambda \right\}} 
		\beta_{1;k}(X,t)^{\p} \ \frac{\dd t}{t} \dd \mu(X) \\
	& \stackrel{\hphantom{\text{(F)}}}{\le} C(\n,\N,\K,\p,C_0,k) \ \mathcal{M}_{\K^{\p}}(\mu).
\end{align*}
Now we consider the integral $b$. 
We get using Fatou's Lemma \cite[1.3, Thm.1]{Evans} to interchange the summation with the integration
\begin{align*}
	b \ \ & \stackrel{\eqref{2.1.10;101} \eqref{06.12.2013.2}}{\le} C \ \int_{U_{10}} \int_{\frac{D(q)}{100}}^{2} \sum_{i \in I_{12}} 
		\Eins_{\left\{t > \frac{\diam R_{i}}{11},d(q,R_{i}) \le t \right\}}
		\left( \frac{\diam R_{i}}{t} \right)^{\N+1} \frac{\dd t}{t}\dd \cH^{\N}(q) \\ \displaybreak[1]
	& \stackrel{\substack{\hphantom{\eqref{2.1.10;101} \eqref{06.12.2013.2}}\\(F)}}{\le} C \ \sum_{i \in I_{12}} (\diam R_{i})^{\N+1} \int_{\frac{\diam R_{i}}{11}}^{\infty} 
		\int_{U_{10}} \Eins_{\left\{d(q,R_{i}) \le t \right\}} \dd \cH^{\N}(q) 
		 \frac{\dd t}{t^{\N+2}}
	\stackrel{\eqref{2.1.10;101}}{\le} C(\N,\p).
\end{align*}
Due to Lemma \ref{inneredisjunkt} (ii) the proof of Theorem \ref{thm4.1} is completed by applying 
Lemma \ref{24.2.10;17}, \eqref{10.1.2010;b}
and with (C) from page \pageref{Grundeigenschaften}
because $ \mathcal{M}_{\K^{\p}}(\mu) \stackrel{\text{(C)}}{\le} \eta < \varepsilon^{\p}$ 
(see page \pageref{3.12.2013.3} and page \pageref{3.12.2013.4}).
\end{proof}

\setcounter{equation}{0}
\section{\texorpdfstring{$\mathcal{Z}$ Is not too Small}{Z Is not too Small}} \label{notanullset}
Our aim is to prove Theorem \ref{16.10.2013.1}. 
In Definition \ref{def3.2}, we defined a partition of the support $F$ of our measure $\mu$ in four parts, namely 
$\mathcal{Z}$, $F_{1}$, $F_{2}$, $F_{3}$. Then, in section \ref{17.05.2013.1}, we constructed some function $A$, 
the graph $\Gamma$ of which
covers the set $\mathcal{Z}$. To get our main result, we need to know that we covered a major part of $F$.
In this last part of the proof of Theorem \ref{16.10.2013.1}, we show that the $\mu$-measure of 
$F_{1}$, $F_{2}$, $F_{3}$
is quite small. In particular, we deduce $\mu(F_{1} \cup F_{2} \cup F_{3}) \le \frac{1}{100}$. As stated at the
beginning of section \ref{04.02.2014.1} on page \pageref{04.02.2014.1}, this completes the proof of 
Theorem \ref{16.10.2013.1}.

\subsection{\texorpdfstring{Most of $F$ is close to the graph of $A$}{Most of F lies near the graph of A}}
With $K := 2 \left(104 \cdot 10 \cdot 6 + 214 \right)$, \label{WahlvonK}
we define the set $G$ by
\begin{align*}
	\left\{ x \in F \setminus \mathcal{Z} \ | \ 
		\forall i \in I_{12} \text{ with } \pi(x) \in 3R_{i} \text{, we have } x \notin KB_{i} \right\} 
	      \cup 
	\left\{ x \in F \setminus \mathcal{Z} \ | \ 
		\pi(x) \in \pi (\mathcal{Z}) \right\}. 
\end{align*}
At first, we show that the $\mu$-measure of $G$ is small.
\begin{lem} \label{lem3.14}
	Let $0< \alpha \le \frac{1}{280}$. There exist some 
	$ \tilde \varepsilon=\tilde \varepsilon(\n,\N,C_{0},\alpha)$ so that, if 
	$\eta < 2 \tilde \varepsilon$ and  $k \ge 4$, 
	there exists some constant $C=C(\n,\N,\K,\p,C_{0})$ so that,
	for all $\varepsilon \in [\frac{\eta}{2},\tilde \varepsilon)$, we have
	\[\mu(G) \le C \mathcal{M}_{\K^{\p}}(\mu) \stackrel{\text(C)}{\le} C \eta,\]
	where the condition (C) was given on page \pageref{Grundeigenschaften}.
\end{lem}
\begin{proof}
	Let $0< \alpha \le \frac{1}{280}$ and
	$\tilde \varepsilon := \min\{\bar \varepsilon,\frac{\alpha}{\bar C}\} $ where 
	$\bar \varepsilon$ is given by Lemma \ref{lem3.9} and $\bar C = \bar C(\n,\N,C_{0})$ is a fixed 
	constant defined 
	in this proof on page \pageref{12.06.2014.1}. Furthermore let $ \eta < 2 \tilde \varepsilon$, $k \ge 4$ and
	$\eta \le 2 \varepsilon < 2 \tilde \varepsilon$.

	Let $x \in G$. 
	If $x \in G\setminus \pi^{-1}(\pi(\mathcal{Z})) \subset F \subset B(0,5)$,
	with Lemma \ref{Riueberdeckung} (ii), there exists some $i \in I_{12}$ with $\pi(x) \in R_{i} \subset 2R_{i}$.
	Let $X_{i}$ be the centre of $B_{i}$ (cf. Lemma \ref{vor3.12}).
	We set
	\[X(x):= \begin{cases}
			X_i  & \text{if } x \in G \setminus \pi^{-1}(\pi(\mathcal{Z}))\\
			\pi(x)+A(\pi(x)) & \text{if } x \in G \cap \pi^{-1}(\pi(\mathcal{Z})).
		\end{cases}\]
	\claim{1} For all $x \in G$ and $X=X(x)$ defined as above, we have
	\begin{align}
		d(x,X) &< 7 d(x), \label{23.10.12.1}
		&d(\pi(x),\pi(X)) \le \textstyle{\frac{d(x)}{10}}, 
		& &\textstyle{\frac{d(x)}{2}} \le d(X,x), 
		& &\left(X,\textstyle{\frac{d(x)}{10}}\right) \in S. 
	\end{align}
	\proofofclaim{1}

	1. Case: $x \in G \setminus \pi^{-1}(\pi(\mathcal{Z}))$.\\
	Due to the definition of $G$ and $\pi(x) \in 2R_{i} \subset 3R_{i}$, we have $ x \notin KB_{i}$.
	By adding some $q \in R_{i}$ with triangle inequality and using Lemma \ref{vor3.12}  
	we get $d(\pi(x),\pi(X_{i})) \le 104 \diam B_{i} \label{31.1.10;41}$.
	With Lemma \ref{vor3.12}, we know $\bigl( X_{i}, \frac{\diam B_{i}}{2} \bigr) \in S$ and hence we get
	$d(X_{i}) < \diam B_{i}$.
	Using $x \notin KB_{i}$ and Lemma \ref{lem3.9}, we get
	$K \cdot \frac{\diam B_i}{2} < d(x,X_i) < 6d(x)+ 214 \diam B_i $
	which yields by definition of $K$ (cf.~the beginning of this subsection) $104\diam B_{i} < \frac{d(x)}{10}$.
	From the previous two estimates, we get
	$d(x,X_i) < 7 d(x)$, i.e., the first inequality holds in this case.
	Furthermore, we have the second one since $d(\pi(x),\pi(X_{i})) \le 104 \diam B_{i} < \frac{d(x)}{10}$.
	We have $\bigl( X_{i}, \frac{\diam B_{i}}{2} \bigr) \in S$, so we get
	$d(x) \le d(X_{i},x) + \frac{\diam B_{i}}{2} < d(X_{i},x) + \frac{d(x)}{2}$,
	and hence the third inequality holds in this case.
	Due to Lemma \ref{mengebeschraenkt}, we have 
	$\frac{\diam B_{i}}{2} < \frac{d(x)}{10} < \frac{60}{10}<50$ so that 
	with Lemma \ref{rem3.1} (ii) we deduce $\bigl(X,\frac{d(x)}{10}\bigr) \in S$.

	2. Case: $x \in G \cap \pi^{-1}(\pi(\mathcal{Z}))$.\\
	We have $\pi(x) \in \pi(\mathcal{Z})$ and hence $X=\pi(x)+A(\pi(x)) \in \mathcal{Z}$ 
	(cf. Definition \ref{04.11.2013.1}).
	By definition of  $\mathcal{Z}$ and Lemma \ref{rem3.1} (i), 
	we obtain $(X,\sigma) \in S$ for all $\sigma \in (0,50)$ and hence
	$\frac{d(x)}{2} \le d(X,x)+\sigma$, which establishes the third estimate.
	Moreover, we have
	$ d(\pi(X),\pi(x)) = d(\pi(x),\pi(x))=0$.
	Using Lemma \ref{rem3.8}, we obtain $d(X)=0$ and hence
	we get with Lemma \ref{lem3.9} $ d(x,X) \le 6d(x)$.
	Furthermore, we have with Lemma \ref{mengebeschraenkt} that 
	$\frac{d(x)}{10} \le 6 < 50$
	so that by definition of $\mathcal{Z}$, we get
	$\bigl(X,\frac{d(x)}{10}\bigr) \in S$.
	\proofofclaimend{1} \indent
	Let $P_{x}:=P_{\bigl(X,\frac{d(x)}{10}\bigr)}$ be the plane associated to $B(X,\frac{d(x)}{10})$ 
	(cf. Definition \ref{12.07.13.1}). 
	We define the set
	\begin{align}\label{31.1.10;46}
		\Upsilon := \left\{ u \in B\left(X,\textstyle{\frac{d(x)}{10}} \right) \Big| d(u,P_{x}) 
		\le \textstyle{\frac{8}{\delta}} \textstyle{\frac{d(x)}{10}} \varepsilon  \right\}.
	\end{align}
	Due to Definition \ref{12.07.13.1} we have 
	$\beta_{1;k}^{P_{x}}\bigl(X,\frac{d(x)}{10}\bigr) \le 2\varepsilon$
	and hence we get using Chebyshev's inequality
	\begin{align*}
		\mu\left(B\left(X,\textstyle{\frac{d(x)}{10}} \right) \setminus \Upsilon \right) 
		\le \textstyle{\frac{\delta}{8 \varepsilon}} \left(\textstyle{\frac{d(x)}{10}}\right)^{\N} 
			 \beta_{1;k}^{P_{x}}\left(X, \textstyle{\frac{d(x)}{10}}\right) 
		\le \textstyle{\frac{\delta}{4}} \left(\textstyle{\frac{d(x)}{10}}\right)^{\N}
	\end{align*}	
	Since $\Upsilon \subset B\left(X,\frac{d(x)}{10}\right)$ and 
	$\delta \big(B\big(X,\frac{d(x)}{10}\big)\big) \ge \frac{1}{2}\delta$ (cf.~Definition \ref{12.07.13.1} of 
	$S_{total}$), we obtain
	\begin{align*}
		\mu\left(B\left(X,\textstyle{\frac{d(x)}{10}}\right) \cap \Upsilon\right)
		 \ge \mu\left(B\left(X,\textstyle{\frac{d(x)}{10}} \right)\right) 
			- \mu\left(B\left(X,\textstyle{\frac{d(x)}{10}} \right) \setminus \Upsilon \right) 
		 \ge \textstyle{\frac{\delta}{4}} \left(\textstyle{\frac{d(x)}{10}}\right)^{\N}.
	\end{align*}
	With Corollary \ref{04.09.12.1} ($\lambda=\frac{\delta}{4}$, $t=\frac{d(x)}{10}$), there exist constants
	$C_{1}=C_{1}(\n,\N,C_{0})$, $C_{2}=C_{2}(\n,\N,C_{0})$ and
	an $\left(\N,10\N\frac{d(x)}{10 C_1}\right)$-simplex 
	$T=\Delta(x_0,\dots,x_{\N}) \in F \cap B\left(X,\frac{d(x)}{10}\right) \cap \Upsilon$
	so that for all $j \in \{0,\dots,\N \}$
	\begin{align} \label{21.12.09;5}
		\mu\left( B\left(x_{j},\textstyle{\frac{d(x)}{10C_1}}\right) \cap B\left(X,\textstyle{\frac{d(x)}{10}}\right) \cap \Upsilon\right) 
		& \ge \left(\textstyle{\frac{d(x)}{10}}\right)^{\N} \textstyle{\frac{1}{C_{2}}}.
	\end{align}
	Let $y_j \in B\left(x_{j},\frac{d(x)}{10C_1}\right)\cap \Upsilon$ for all $j \in \{0,\dots,\N \}$.
	By applying Lemma \ref{17.11.11.2} $(\N+1)$ times, we find that 
	$\Delta(y_0,\dots,y_{\N})$ is an $\Bigl(\N,8\N\frac{d(x)}{10C_{1}}\Bigr)$-simplex.

	\claim{2}
	For all $x \in G$, we have $d(x,\aff(y_{0},\dots,y_{\N})) \ge \frac{d(x)}{4}$.

	\proofofclaim{2}
		Let $P_{y}:=\aff(y_{0},\dots,y_{\N})$ be the plane through $y_0, \dots,y_\N$.
		Applying Lemma \ref{21.11.11.2} ($C=\frac{C_1}{8\N}$, $\hat C=1$, $t=\frac{d(x)}{10}$,
		$\sigma = \frac{8}{\delta}\varepsilon$, $P_{1}=P_{y}$, $P_{2}=P_{x}$,
		$S=\Delta(y_{0}, \dots,y_{\N})$, $x=X$, $m=\N$) yields $\varangle(P_{y},P_{x}) \le \alpha$,
		where we use that $\varepsilon \le \tilde \varepsilon \le \frac{\alpha}{\bar C}$ and $\bar C$ is
		given by Lemma \ref{21.11.11.2} \label{12.06.2014.1}.
		So, with Definition \ref{12.07.13.1}, we obtain $\varangle(P_{y},P_{0}) \le 2\alpha$.
		Let $\hat P_y \in \mathcal{P}(\n,\N)$ be the $\N$-dimensional plane parallel to $P_y$ 
		with $X \in \hat P_y$, and 
		$\hat P_0 \in \mathcal{P}(\n,\N)$ be the plane parallel to $P_0$ with $X \in \hat P_0$.
		We have $\alpha \le \frac{1}{280}$ and hence
		\[ d(\pi_{\hat P_{y}}(x) , \pi_{\hat P_{0}}(x)) 
			 = |\pi_{\hat P_{y}-X}(x-X) - \pi_{\hat P_{0}-X}(x-X)|
			 \le d(x,X) \ \varangle(\hat P_{y},\hat P_{0})
			 \stackrel{\eqref{23.10.12.1}}{<} \frac{d(x)}{20}.\]
		Furthermore, with \eqref{23.10.12.1}, we get
		$d(\pi_{\hat P_{0}}(x),X) = d(\pi(x),\pi(X)) \le \frac{d(x)}{10}$.
		Using triangle inequality, the previous two estimates imply
			$d(\pi_{\hat P_{y}}(x),X) 
			 \le \frac{d(x)}{20} + \frac{d(x)}{10}$.
		Since $y_{0} \in \Upsilon \subset B(X,\frac{d(x)}{10})$ we have 
		$d(P_{y},\hat P_{y})=d(X,P_{y})\le d(X,y_{0})\le \frac{d(x)}{10}$ and hence
		\begin{align*}
			\frac{d(x)}{2} 
			&\stackrel{\eqref{23.10.12.1}}{\le} d(x,P_{y})+ d(P_{y},\hat P_{y})+ d(\pi_{\hat P_{y}}(x),X)
			\le d(x,P_{y})+ \frac{d(x)}{4}
		\end{align*}
		and gain $d(x,P_{y}) \ge \frac{d(x)}{4}$.
	\proofofclaimend{2} \indent
	With \eqref{23.10.12.1} and
	$ d(y_{j},X) \le d(y_{j},x_{j}) + d(x_{j},X) \le \frac{d(x)}{10 C_{1}} + \frac{d(x)}{10}$,
	we obtain $y_{0}, \dots y_{\N},x \in B(X,7d(x))$ which is a subset of $B(X, \frac{C_{1}}{8\N}  \frac{d(x)}{10})$,
	where we used the explicit characterisation of $C_{1}$ given in Lemma \ref{lem2.3}.
	Due to the second property of a $\mu$-\proper{} integrand (see Definition \ref{muproper}),
	there exists some $\tilde C = \tilde C(\n,\N,\K,\p,C_{0}) \ge 1$ so that we get with Claim 2
	\begin{align*}
		\mathcal{K}^{\p}(y_{0},\dots,y_{\N},x) 
		&\ge \frac{1}{\left( \frac{d(x)}{10} \right)^{\N(\N+1)} \tilde C} 
		\left(\frac{d(x,\aff(y_{0},\dots,y_{\N}))}{\frac{d(x)}{10}}\right)^{\p}
		> \tilde C^{-1} \left( \frac{10}{d(x)} \right)^{\N(\N+1)}.
	\end{align*}
	This estimate holds for all $y_{i} \in B\big( x_{i},\frac{d(x)}{10C_{1}}\big) \cap \Upsilon$.
	By restricting the integration to the balls $B\big( x_{i},\frac{d(x)}{10C_{1}}\big)$ and using
	the previous estimate as well as estimate \eqref{21.12.09;5}, we get 
	\begin{align*}
		\int \dots \int \mathcal{K}^{\p}(y_0,\dots,y_{\N},x) \dd \mu(y_0) \dots \dd\mu(y_{\N})
		& \ge \tilde C^{-1} C_{2}^{-(\N+1)}. 
	\end{align*}
	We have proven the previous inequality for all $x \in G$, so
	finally we deduce with (C) from page \pageref{Grundeigenschaften} that there exists some constant
	$C=C(\n,\N,\K,\p,C_{0})$ so that
	\[\mu(G) \le  \tilde C C_{2}^{(\N+1)}
			\int_{G} \int \dots \int \mathcal{K}^{\p}(y_0,\dots,y_{\N},x) \dd \mu(y_0) \dots \dd\mu(y_{\N})\dd  \mu(x) 
		 \stackrel{\text{(C)}}{\le} C \eta.\]
\end{proof}

\begin{lem} \label{lem3.15}
	Let $\alpha, \varepsilon >0$. If $\eta \le 2\varepsilon$, 
	we have 
	$(20K)^{-1}d(x) \le D(\pi(x)) \le d(x)$ for all $x \in F \setminus G$,
	where $K$ is the constant defined on page \pageref{WahlvonK} at the beginning of this subsection.
\end{lem}
\begin{proof}
	Let $x \in F \setminus G$.
	We have $D(\pi(x)) = \inf_{y \in \pi^{-1}(\pi(x))}d(y) \le d(x).$
	If $x \in \mathcal{Z}$, Lemma \ref{rem3.8} implies $d(x)=0$, so the statement is trivial.
	Now we assume $x \notin \mathcal{Z}$.
	Since $x \notin G \cup \mathcal{Z}$, by definition of $G$, 
	there exists some $i \in I_{12}$ with $\pi(x) \in 3R_{i}$ and 
	$x \in KB_{i}$. 
	We have $B_{i} = B(X_{i},t_{i})$ where $(X_{i},t_{i}) \in S$ (see Lemma \ref{vor3.12}) 
	and $K > 1$ (see page \pageref{WahlvonK}) so we obtain
	$d(x) \le d(X_{i},x)+t_{i} \le < K \diam B_{i}$.
	Now, with Lemma \ref{rem3.10} (i) and \ref{vor3.12}, we deduce
	$ D(\pi(x)) \ge \frac{1}{20K}d(x)$.
\end{proof}

\begin{lem} \label{lem3.16}
	Let $0<\alpha \le \frac{1}{4}$. 
	There exists some 
	$\bar \varepsilon = \bar \varepsilon (\n,\N,C_{0})$ and some $\tilde k \ge 4$ so that, if
	$\eta < 2 \bar \varepsilon$ and $k \ge \tilde k$,
	for all $\varepsilon \in [\frac{\eta}{2}, \bar \varepsilon)$ we have that the following is true.
	There exists some constant $C=C(\N)$ so that, for all $x \in F$ with $t \ge \frac{d(x)}{10}$, we have
	\[ \int_{B(x,t)\setminus G} d\big(u,\pi(u)+A(\pi(u))\big) \dd \mu(u) \le C \varepsilon t^{\N+1}.\]
\end{lem}
\begin{proof}
	Let $0<\alpha \le \frac{1}{4}$. We choose some $\varepsilon$ with 
	$\eta \le 2 \varepsilon < 2 \bar \varepsilon$ and some $k \ge \tilde k := \max\{\bar k, \tilde C\}$, 
	where $\bar \varepsilon$ and $\bar k$ are given by
	Lemma \ref{lem3.12} and $\tilde C$ is a fixed constant introduced in step VI of this proof.
	Let $x \in F$ and $t \ge \frac{d(x)}{10}$. We define
	\[I(x,t):= \left\{ i \in I_{12} | (3R_{i} \times P_{0}^{\perp}) \cap B(x,t) \cap (F \setminus G) \neq \emptyset  \right\}\]
	where $3R_{i} \times P_{0}^{\perp}:=\{ x \in \R^{\n} | \pi(x) \in 3R_{i}\}$.
	At first, we prove some intermediate results:\\
	I.\,    Due to the definition of $G$ we have
		$(B(x,t)\cap F)\setminus (G \cup \mathcal{Z}) \subset 
		\bigcup_{i \in I(x,t)}(3R_{i} \times P_{0}^{\perp}) \cap KB_{i}$.\\
	II.\,   Let $u \in 3R_{i} \times P_{0}^{\perp}$. Using Lemma \ref{lem3.11} (iv) implies that 
		$\sum_{j\in I_{12}} \phi_{j}(\pi(u))$ is a finite sum.\\
	III.\,  Let $i \in I(x,t)$ and $j \in I_{12}$.
		We define $\phi_{i,j}$ to be $0$ if $3R_{i}$ and $3R_{j}$ are disjoint and $1$ if they are not
		disjoint. We have $\phi_{j}(\pi(u)) \le 1 = \phi_{i,j}$ for all 
		$u \in (3R_{i} \times P_{0}^{\perp}) \cap KB_{i}$, since
		if $\phi_{j}(\pi(u)) \neq 0$ the definition of $\phi_{j}$ (see page \pageref{Defphii}) 
		gives us $\pi(u) \in 3R_{j}$ and, because $\pi(u) \in 3R_{i}$, we deduce
		$ 3R_{i}\cap 3R_{j} \neq \emptyset$. \\
	IV.\,   If $\phi_{i,j} \neq 0$, we can apply Lemma \ref{rem3.10} (iii) and Lemma \ref{lem3.12} (i).
		Hence, using Lemma \ref{vor3.12}, the size of $B_{i}$ as well as the distance of 
		$B_{i}$ to $B_{j}$ are comparable to the size of $B_{j}$.
		Consequently, there exists some constant $\tilde C$ so that 
		$KB_{i} \subset \tilde C B_{j} \subset k B_{j}$.\\
	V.\,	If $u \in kB_{j}$, we have
		$ |\pi^{\perp}(u)- A_{j}(\pi(u))| < 2d(u,P_{j})$.
		We recall that $P_{j}$ is the graph of the affine map $A_{j}$ 
		(cf. Definition \ref{3.12.2013.10} and Lemma \ref{AiLipschitz}).
		\begin{proof}
			We set $\hat P_{0}:= P_{0} + A_{j}(\pi(u))$ and 
			$v:= \pi(u) + A_{j}(\pi(u)) = \pi_{\hat P_{0}}(u)$.
			Remark \ref{24.04.2013.1} implies
			\[|\pi_{P_j}(u) - v| = |\pi_{P_j-v}(u-v)- \pi_{\hat P_{0}-v}(u-v)|
				 \le |u-v| \ \varangle(P_{j},P_{0}). \]
			Using this and $\varangle(P_{j},P_{0}) \le \alpha < \frac{1}{2}$ 
			(cf. Definition \ref{3.12.2013.10})
			we obtain
			$|u-v| < d(u,P_{j}) + \frac{1}{2}|u-v|$
			and hence
			$|\pi^{\perp}(u) - A_{j}(\pi(u))|=|u-v| < 2 d(u,P_{j})$. 
		\end{proof}
	If $ u \in \mathcal{Z}$, the definition of $A$ (see page \pageref{04.11.2013.1})
	yields $d(u,\pi(u)+ A(\pi(u)))=0$.
	Using Lemma \ref{15.10.2013.1} and Definition \ref{04.11.2013.1}, we get
	\[\int_{B(x,t)\setminus G} d(u,\pi(u)+ A(\pi(u))) \dd \mu(u) 
		\le \int_{B(x,t)\setminus (G \cup \mathcal{Z})} \sum_{j\in I_{12}}\phi_{j}(\pi(u)) 
		\left|\pi^{\perp}(u) - A_{j}(\pi(u))\right| \dd \mu(u). \]
	Using I to V we obtain
	\[\int_{B(x,t)\setminus G} d(u,\pi(u)+ A(\pi(u))) \dd \mu(u) \le
		2 \sum_{i \in I(x,t)}\sum_{j\in I_{12}}  \phi_{i,j}  t_{j}^{\N+1} 
			\frac{1}{t_{j}^{\N}} \int_{kB_{j}}
		 	\frac{d\left(u,P_{j}\right)}{t_{j}} \dd \mu(u).\]
	Now we get the statement by using the following five results.\\
	VI.\,   Lemma \ref{lem3.12} and the definition of $S_{total}$ imply
		$\beta_{1;k}^{P_{j}}(B_{j}) \le 2\varepsilon$.\\
	VII.\,  Let $i \in I(x,t)$ and $j \in I_{12}$. If $\phi_{i,j} \neq 0$, we conclude that 
		$ 3R_{i}\cap3R_{j} \neq \emptyset$. 
		Hence, with Lemma \ref{lem3.11} (iii) and Lemma \ref{vor3.12}, we deduce
		$2t_{j}=\diam B_{j} \le 1000 \diam R_{i}$.\\
	VIII.\, For $i \in I(x,t)$, we have with Lemma \ref{lem3.11} (iv) that 
		$ \sum_{j\in I_{12}} \phi_{i,j} \le (180)^{\N}$.\\
	IX.\,   For $i \in I(x,t)$, there exists some $y \in B(x,t) \cap (F\setminus G)$ with 
		$ \pi(y) \in 3R_{i}$. We obtain with Lemma
		\ref{rem3.10}, Lemma \ref{lem3.15} and our assumption $t \ge \frac{d(x)}{10}$ that
		$10\diam R_{i} \le d(x)+d(x,y) \le 11t$.\\
	X.\,    Let $ i \in I(x,t)$. With XI we obtain 
		$ \diam R_{i}< 2t$ and, because $(3R_{i} \times P_{0}^{\perp}) \cap B(x,t) \neq \emptyset$,
		we get $ R_{i} \subset B(\pi(x),t+\diam 3R_{i})\cap P_{0}\subset B(\pi(x),7t)\cap P_{0}$.
		Moreover, with Lemma \ref{inneredisjunkt} (ii), the primitive cells $R_{i}$ have disjoint interior
		and hence we get with Lemma \ref{24.10.12.1}
		(we recall that $\vol$ denotes the volume of the $n$-dimensional unit sphere)
		\begin{align*}\sum_{i \in I(x,t)} (
			\diam R_{i})^{\N} 
			& \le \sqrt{\N}^{\N} \cH^{\N}(B(\pi(x),7t)\cap P_{0})
			 = \sqrt{\N}^{\N} \vol (7t)^{\N}.
		\end{align*}
\end{proof}

\begin{dfn} \label{DefinitionvonFtilde}
	We define
	$\tilde{F} := \big\{ x \in F \setminus G \ | \ d(x,\pi(x) + A(\pi(x))) \le \varepsilon^{\frac{1}{2}} d(x) \big\}$.
\end{dfn}

\begin{thm} \label{thm3.18}
	Let $0<\alpha \le \frac{1}{4}$. 
	There exists some 
	$\hat \varepsilon = \hat \varepsilon (\n,\N,C_{0})\le \frac{1}{4}$ and some $\tilde k \ge 4$ so that, if
	$\eta < 2 \hat \varepsilon$ and $k \ge \tilde k$,
	there exists some constant $C_{5} = C_{5}(\n,\N,\K,\p,C_{0})$ so that,
	for all $\varepsilon \in [\frac{\eta}{2},\hat \varepsilon)$, we have
	$ \mu(F \setminus \tilde{F}) \le C_{5} \varepsilon^{\frac{1}{2}}$. 
\end{thm}
\begin{proof}
	Let $0<\alpha \le \frac{1}{4}$. We choose some $\varepsilon$ with 
	$\eta \le 2 \varepsilon < 2 \hat \varepsilon:= \min\{2\tilde \varepsilon, 2\bar \varepsilon,\frac{1}{2}\}$ 
	and some $k \ge \tilde k$ 
	where $\tilde \varepsilon$ is given by Lemma \ref{lem3.14} and 
	$\bar \varepsilon$ and $\tilde k$ are given by Lemma \ref{lem3.16}.

	At first, we prove some intermediate results:\\
	I.\,	We have $\mathcal{Z} \subset \tilde F$ because for $ x \in \mathcal{Z}$ 
		the definition of $A$ on $\mathcal{Z}$ (see Definition \pageref{04.11.2013.1}) implies that
		$ d(x,\pi(x)+A(\pi(x)))=d(x,x)=0$.\\
	II.\,	If $x \in F \setminus(\tilde{F} \cup G)$, we conclude with I that $x \notin \mathcal{Z}$ 
		and, with Lemma \ref{rem3.8}, we deduce
		$d(x) \neq 0$. So 
		$\mathcal{G} = \left\{ B\left(x,\frac{d(x)}{10}\right) \Big| x \in F \setminus(\tilde{F} \cup G) \right\}$
		is a set of nondegenerate balls.
		For $x \in F \subset B(0,5)$, we have $ d(x) \le 60 $ (see Lemma \ref{mengebeschraenkt})
		so that we can apply the Besicovitch's covering theorem \cite[1.5.2, Thm. 2]{Evans} to $\mathcal{G}$
		and get $N_{0}=N_{0}(\n)$ families 
		$\mathcal{B}_{m} \subset \mathcal{G}, m=1,...,N_{0}$
		of disjoint balls with
		\[ F \setminus(\tilde{F} \cup G) \subset \bigcup_{m=1}^{N_{0}} \bigcup_{B \in \mathcal{B}_{m}}B.\]
	III.\,	Since $d$ is 1-Lipschitz (Lemma \ref{rem3.7}), for all $u \in B\big(x,\frac{d(x)}{10}\big)$
		$d(x)-d(u) \le d(x,u) \le \frac{d(x)}{10}$
		and hence
		$\frac{1}{d(u)} \le \frac{10}{9}\frac{1}{d(x)} < \frac{2}{d(x)}$.\\
	IV.\,	Let $1 \le m \le N_{0}$ and
		let $B_{x} = B\left(x,\frac{d(x)}{10}\right)$ and $B_{y} = B\left(y,\frac{d(y)}{10}\right)$ 
		be two balls in 
		$\mathcal{B}_{m}$. Then we either have
		\begin{enumerate}[a)]
			\item $\pi\left(\frac{1}{40K}B_{x}\right) \cap \pi\left(\frac{1}{40K}B_{y}\right) = \emptyset $ 
			or
			\item if $2d(x) \ge d(y)$, we have $B_{y} \subset 200 B_{x}$ 
			and $ \diam B_{y} > (40K)^{-1} \diam B_{x}$,
		\end{enumerate}
		where $K$ is the constant from page \pageref{WahlvonK}.
		\begin{proof}
		Let $\pi\left(\frac{1}{40K}B_{x}\right) \cap \pi\bigl(\frac{1}{40K}B_{y}\bigr) \neq \emptyset $ 
		and $2d(x) \ge d(y)$. We deduce with Lemma \ref{lem3.9}
		$d(x,y) <19 d(x)$, which implies
		$B_{y} \subset B\big(x,19d(x)+\textstyle{\frac{d(y)}{10}}\big) = 200 B_{x}$.
		With Lemma \ref{lem3.15}, we get
		$ \frac{d(x)}{20K} \le D(\pi(y))+d(\pi(x),\pi(y)) < d(y) + \frac{d(x)}{40K},$
		and hence $ d(y) > (40K)^{-1}d(x)$.
		All in all, we have proven that either case a) or case b) is true.
		\end{proof} \noindent
	V.\,	There exists some constant $C=C(\N)$ so that
		$\sum_{B \in \mathcal{B}_{m}} (\diam B)^{\N} \le C$ for all $1\le m \le N_{0}$.
		\begin{proof}
			Let $1 \le m \le N_{0}$.
			We recursively construct for every $m$ some sequence of numbers, 
			some sequence of balls and some sequence of sets.
			At first, we define the initial elements. Let $d_{m}^{1} := \sup_{B \in \mathcal{B}_{m}} \diam B$.
			We have $d_{m}^{1} < \infty$ because, for all $x \in F \subset B(0,5)$, we have with 
			Lemma \ref{mengebeschraenkt} that
			$ d(x) \le 60$. 
			Now we choose $B_{m}^{1} \in \mathcal{B}_{m}$ with $ \diam B_{m}^{1} \ge \frac{d_{m}^{1}}{2}$ 
			and define
			\[ \mathcal{B}_{m}^{1} := \left\{ B \in \mathcal{B}_{m} 
				\Big| \pi\left(\textstyle{\frac{1}{40K}}B_{m}^{1}\right) \cap \pi\left(\textstyle{\frac{1}{40K}}B\right) \neq \emptyset \right\}.\]
			We continue this sequences recursively. We set
			$d_{m}^{i+1}= \sup_{B^{'} \in  \mathcal{B}_{m} \setminus \bigcup_{j=1}^{i}\mathcal{B}_{m}^{j}} \diam B^{'}$,
			choose
			$B_{m}^{i+1} \in \mathcal{B}_{m} \setminus \bigcup_{j=1}^{i}\mathcal{B}_{m}^{j}$
			with $ \diam B_{m}^{i+1} \ge \frac{d_{m}^{i+1}}{2}$ and define
			\[ \mathcal{B}_{m}^{i+1} := \Bigg\{ B \in \mathcal{B}_{m} \setminus \bigcup_{j=1}^{i}\mathcal{B}_{m}^{j}
				\Big| \pi\left(\textstyle{\frac{1}{40K}}B_{m}^{i+1}\right) \cap \pi\left(\textstyle{\frac{1}{40K}}B\right) \neq \emptyset \Bigg\}.\]
			If there exists some $l \in \mathbb{N}$ so that eventually
			$\mathcal{B}_{m} \setminus \bigcup_{j=1}^{l} \mathcal{B}_{m}^{j} = \emptyset$, 
			we set for all $i \ge l$
			$\mathcal{B}_{m}^{i} :=  \emptyset$, and interrupt the sequences $(d_{m}^{i})$ and $(B_{m}^{i})$.
			We have the following results:\\
			(i)\, For all $l \in \mathbb{N}$ and 
			$B_{m}^{l}=B\big(x_{m}^{l},\frac{d(x_{m}^{l})}{10}\big)$, 
			we have with Lemma \ref{mengebeschraenkt} and
			$x_{m}^{l} \in F \subset B(0,5)$ that $\frac{d(x_{m}^{l})}{10} \le 6$. Hence we get
			$B_{m}^{l} \subset B(0,11)$.\\
			(ii)\, For all $1 \le m \le N_{0}$, we have
			$\bigcup_{i=1}^{\infty}\mathcal{B}_{m}^{i} = \mathcal{B}_{m}$.\\
			\textit{Proof.}
				If there exist only finitely many $d_{m}^{l}$, the construction
				implies $\mathcal{B}_{m} \subset \bigcup_{j=1}^{\infty} \mathcal{B}_{m}^{j}$.\\
				Now we assume that there exist infinitely many $d_{m}^{l}$.
				Since $\mathcal{B}_{m}$ is a family of disjoint balls, the set
				$\{B_{m}^{l} |l \in \mathbb{N}\}$ is also a family of disjoint balls.
				Due to (i), all of those balls are contained in $B(0,11)$. If there exists some $c >0$ 
				with $\diam B_{m}^{l} > c$ for all $l \in \mathbb{N}$,
				there can not be infinitely many of such balls. Hence we deduce
				$\diam B_{m}^{l} \rightarrow 0$ if $l \rightarrow \infty.$
				Let $B \in \mathcal{B}_{m}$. If
				$B \notin  \bigcup_{i=1}^{\infty}\mathcal{B}_{m}^{i}$,
				we obtain $2\diam B_{m}^{l} \ge d_{m}^{l} \ge \diam B$ for all $l \in \mathbb{N}$
				where we used the definition of $d_{m}^{l}$.
				This is in contradiction to $\diam B_{m}^{l} \rightarrow 0$. So we get
				$B \in \bigcup_{i=1}^{\infty}\mathcal{B}_{m}^{i}$.
				All in all, we have proven
				$ \bigcup_{i=1}^{\infty}\mathcal{B}_{m}^{i} \supset \mathcal{B}_{m}$.
				The inverse inclusion follows by definition of $\mathcal{B}_{m}^{i}$.
			\hfill$\square$\\
			(iii)\, Let $1 \le m \le N_{0}$, $l \in \mathbb{N}$ and 
			$B_{y}=B\left(y,\frac{d(y)}{10}\right) \in \mathcal{B}_{m}^{l}$,  
			$B_{m}^{l}=B\left(x_{m}^{l},\frac{d(x_{m}^{l})}{10}\right) \in \mathcal{B}_{m}^{l}$.
			We have
			$\pi\left(\frac{1}{40K}B_{m}^{l}\right) \cap \pi\left(\frac{1}{40K}B_{y}\right) \neq \emptyset$
			and 
			$2d(x_{m}^{l}) = 10 \diam B_{m}^{l} \ge 10 \frac{d_{m}^{l}}{2} \ge 10 \frac{ \diam B_{y}}{2} = d(y)$.
			Hence IV implies $B_{y} \subset 200B_{m}^{l}$ and 
			$\diam B_{y} > (40K)^{-1} \diam B_{m}^{l}$.
			The balls in $\mathcal{B}_{m}^{l}$ are disjoint, so, with Lemma \ref{22.2.2012.1} 
			($s=\frac{\diam B_{m}^{l}}{80K}$, $r=200\frac{\diam B_{m}^{l}}{2}$),
			we deduce $\# \mathcal{B}_{m}^{l} \le (200 \cdot 80K)^{\n}$.\\
			(iv)\,	$\{\frac{1}{40K}B_{m}^{l}\}_{l\in \mathbb{N}}$ is a family of disjoint balls and
			with (i) we get $\pi\left(\textstyle{\frac{1}{40K}}B_{m}^{l}\right) \subset \pi(B(0,11))$
			for all $l \in \mathbb{N}$.
			Hence we obtain
			$\sum_{l=1}^{\infty} \left(\diam \pi\left( \textstyle{\frac{1}{40K}}B_{m}^{l} \right)\right)^{\N}
				 \le \frac{2^{\N}}{\vol} \cH^{\N}\left(\pi\left( B(0,11) \right)\right)
				 = 22^{\N}.$
				
			Now we are able to prove V by using (ii),(iii) and (iv):
			\[\sum_{B \in \mathcal{B}_{m}} \left(\diam B \right)^{\N}
				 \le \sum_{l=1}^{\infty} \sum_{B \in \mathcal{B}_{m}^{l}} \left(\textstyle d_{m}^{l}\right)^{\N}
				 = C(\N) \sum_{l=1}^{\infty} \left(\diam \pi\left( \textstyle \frac{1}{40K}B_{m}^{l} \right)\right)^{\N} 
				 \le C(\N).\]
		\end{proof}
	Finally, we can finish the proof of Theorem \ref{thm3.18}.
	Let $p_{B}$ denote the centre of some ball $B$.
	Using the definition of $ \tilde{F}$ and Lemma \ref{lem3.16}, there exists some constant $C=C(\N)$ 
	so that we obtain
	\begin{align*}
		\varepsilon^{\frac{1}{2}} \mu(F \setminus(\tilde F \cup G)) 
		& \stackrel{\hphantom{\text{III}}}{<} 
			\int_{F \setminus(\tilde F \cup G)}\frac{d(u,\pi(u) + A(\pi(u)))}{d(u)} \dd \mu(u)\\ \displaybreak[1]
		&  \stackrel{\substack{{\hphantom{\text{III}}}\\{\text{II}}}}{\le} 
			\sum_{m=1}^{N_{0}} \sum_{B \in \mathcal{B}_{m}} 
			\int_{B \setminus (\tilde F \cup G)}\frac{d(u,\pi(u) + A(\pi(u)))}{d(u)} \dd \mu(u)\\ \displaybreak[1]
		& \stackrel{\text{III}}{<} \sum_{m=1}^{N_{0}} \sum_{B \in \mathcal{B}_{m}} \frac{2}{d(p_{B})}
			C \varepsilon \left(\frac{\diam B}{2}\right)^{\N+1}\\
		& \stackrel{\substack{{\hphantom{\text{III}}}\\{\text{V}}}}{\le} 
			C(\n,\N) \varepsilon.
	\end{align*}
	This leads to 
	$\mu(F \setminus(\tilde F \cup G)) \le C(\n,\N) \varepsilon^{\frac{1}{2}}$.
	With $\eta < 2\varepsilon \le \varepsilon^{\frac{1}{2}}$ and Lemma \ref{lem3.14}
	the assertion holds.
\end{proof}
\subsection{\texorpdfstring{$F_1$ is small}{F1 is small}}
Now we are able to estimate $\mu(F_{1})$. 
We recall that $\eta$ and $k$ are fixed constants (cf. the first lines of section \ref{04.02.2014.1}),
and that $F_{1}$ depends on the choice of $\alpha, \varepsilon >0$ (cf.~Definition \ref{def3.2}).
\begin{thm}
	Let $0 < \alpha \le \frac{1}{4}$. There exist some $\varepsilon^{*}=\varepsilon^{*}(\n,\N,C_{0})$
	and some $\tilde k \ge 4$ so that, if $\eta < 2 \varepsilon^{*}$ and $k \ge \tilde k$, for all
	$\varepsilon \in [\frac{\eta}{2},\varepsilon^{*})$, we have $\mu(F_{1}) < 10^{-6}$.
\end{thm}
\begin{proof}
	Let $0<\alpha \le \frac{1}{4}$ and let $\hat \varepsilon$, $C_{5}$ and $\tilde k$ 
	be the constants given by Theorem \ref{thm3.18}.
	We set $\varepsilon^{*}:= \min\big\{\hat \varepsilon, \frac{10^{-14}}{C_{5}^2}\big\}$
	and choose some $k \ge \tilde k$ and some $\varepsilon \in [\frac{\eta}{2},\varepsilon^{*})$.
	At first, we prove some intermediate results:\\
		I.\, 	Let 
		$\mathcal{G} = \left\{ B\big(x,\frac{h(x)}{10}\big) \Big| x \in F_{1} \cap \tilde F \right\}$.
		This is a set of nondegenerate balls because $\mathcal{Z}\cap F_{1} = \emptyset$ 
		and, by definition of $h(\cdot)$ (see page \pageref{Definitionvonh}), we get $h(x) \le 50$ for all $x \in F$.
		With Besicovitch's covering theorem \cite[1.5.2, Thm. 2]{Evans}, there exist $N_{0}=N_{0}(\n)$ 
		families $\mathcal{B}_{m} \subset \mathcal{G}$, $m=1,...,N_{0}$, 
		containing countably many disjoint balls with
		\[ F_{1} \cap \tilde F \subset \bigcup_{m=1}^{N_{0}} \bigcup_{B \in \mathcal{B}_{m}}B.\]
		II.\,	Let $1 \le m \le N_{0}$ and $B=B\big(x,\frac{h(x)}{10}\big)$
			where $x \in F_{1} \cap \tilde F$.
			Due to the definition of $F_{1}$, there exists some $y \in F$ and some 
			$\tau \in \left[ \frac{h(x)}{5}, \frac{h(x)}{2} \right]$
			with $d(x,y) \le \frac{\tau}{2}$ and $ \delta(B(y,\tau)) \le \delta$.
			For any $z \in B$, we get
			$d(z,y) \le \frac{h(x)}{10} + \frac{\tau}{2} \le \tau$.
			Hence we obtain $B \subset B(y,\tau)$ and conclude
			$\mu(B) \le \delta \tau^{\N} < 3^{\N}  \delta (\diam B)^{\N}$.\\
		III.\,
		For all $1 \le m \le N_{0}$, we have $\sum_{B \in \mathcal{B}_{m}} (\diam B)^{\N} \le 192^{\N}$.
	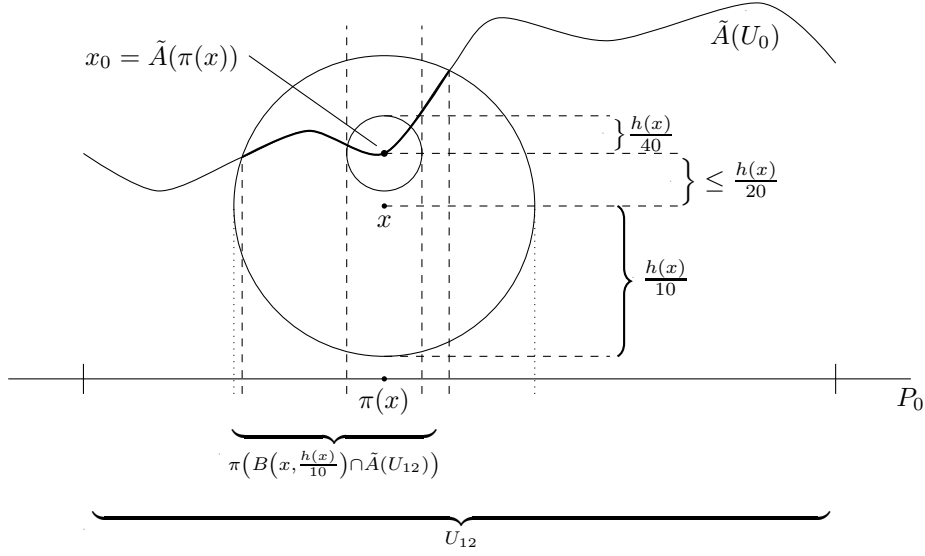
\begin{figure}[ht]
	\begin{center}
	\begin{tikzpicture}[scale=1]
		\draw[-] (3,3)  -- (15,3)node [below] {$P_{0}$};
		\draw[-] (4,2.8)  -- (4,3.2);
		\draw[-] (14,2.8)  -- (14,3.2);
		\path [draw] (8,5.3) circle (2);
		\path [fill] (8,5.3) circle (1pt)node [below] {$x$};
		\draw[dashed] (8,3.3)  -- (11,3.3);
		\draw[decoration={brace,amplitude=5},decorate,thick] (11.1,5.3) -- (11.1,3.3);
		\path [fill] (11.3,4.35) circle (0pt)node [right] {$ \frac{h(x)}{10}$};

		\path [fill] (8,6) circle (1.3pt);
		\draw [-] (7.9,6.1) -- (6.2,7.3) node [left] {$x_{0}=\tilde A (\pi(x))$};
		\path [draw] (8,6) circle (0.5);
		\path [fill] (8,3) circle (1pt) node [below] {$\pi(x)$};
		\draw plot [smooth] coordinates {(4,6) (5,5.5) (6.11,5.95) (7,6.3) (8,6) (8.862,7.1)  (9.5,7.8) (11,7.5) (13,8) (14,7.2)};
		\path [fill] (12.8,7.9) circle (0pt)node [below] {$\tilde A(U_0)$};
		\draw[thick] plot [smooth] coordinates {(6.11,5.95) (7,6.3) (8,6) (8.862,7.1)};
		\path [fill] (4,1) circle (0pt)node [right] {$\underbrace{\hspace{9.8cm}}_{U_{12}}$};
		\draw[dotted] (6,2.8)  -- (6,5.3);
		\draw[dotted] (10,2.8)  -- (10,5.3);
		\draw[dashed] (7.5,2.8)  -- (7.5,7.7);
		\draw[dashed] (8.5,2.8)  -- (8.5,7.7);

		\draw[dashed] (6.11,2.8)  -- (6.11,5.95);
		\draw[dashed] (8.862,2.8)  -- (8.862,7.1);
		\path [fill] (5.8,2) circle (0pt)node [right] {$\underbrace{\hspace{2.7cm}}_{\pi\left( B\left(x,\frac{h(x)}{10} \right) \cap \tilde A(U_{12}) \right)}$};

		\draw[dashed] (8,6)  -- (12,6);
		\draw[dashed] (8,5.3)  -- (12,5.3);
		\path [fill] (11.8,5.63) circle (0pt)node [right] {$\Big\} \le \frac{h(x)}{20}$};
		\draw[dashed] (8,6.5)  -- (11,6.5);
		\path [fill] (10.9,6.27) circle (0pt)node [right] {$\big\} \frac{h(x)}{40}$};
	\end{tikzpicture}
	\end{center}
	\caption[bla]{$\pi\left( B\left(x_{0},\frac{h(x)}{40} \right) \right) \subset\pi\left( B\left(x,\frac{h(x)}{10} \right) \cap \tilde A(U_{12}) \right)$}
	\end{figure}
		\begin{proof}
			We define the function
			$\tilde{A}: U_{12} \rightarrow \mathbb{R}^{\n}, u \mapsto u + A(u)$,
			where $U_{12}=B(0,12) \cap P_{0}$.
			$\tilde{A}$ is Lipschitz continuous with Lipschitz constant less than $2$ 
			because $A$ is defined on $U_{12}$ (see page \pageref{Aeindeutigdefiniert}), 
			$3\alpha$-Lipschitz continuous (see Lemma \ref{ALipschitz}) 
			and $\alpha \le \frac{1}{4}$. 
			Let $ B=B\left(x,\frac{h(x)}{10}\right) \in \mathcal{B}_{m}$. 
			We have $F \subset B(0,5)$ (see (A) on page \pageref{Grundeigenschaften})
			and so $\pi(F) \subset P_{0} \cap B(0,5)$ because $\pi$ is the orthogonal projection
			on $P_{0}$ and $0 \in P_{0}$.
			With the definition of $\tilde F$, Lemma \ref{rem3.8} and 
			$\varepsilon^{\frac{1}{2}}<\frac{1}{20}$, we obtain $d(x,x_{0}) < \frac{h(x)}{20}$
			where $x_{0} := \tilde A(\pi(x))$.
			Let $z \in \pi\left( B\left(x_{0},\frac{h(x)}{40} \right) \right) \subset U_{12}$.
			Using triangle inequality with the point $\tilde A(\pi(x_{0}))=x_{0}$ and
			$\tilde A$ is $2$-Lipschitz, we get
			$d(\tilde A(z),x)\le \frac{h(x)}{10}$. This implies 
			$\tilde A (\pi(B(x_{0},\frac{h(x)}{40}))) \subset B \cap \tilde A (U_{12})$, and hence
			we gain $\pi\left( B\left(x_{0},\frac{h(x)}{40} \right) \right) 
				\subset\pi\left( B \cap \tilde A(U_{12}) \right)$.
			Now we have with  \cite[2.4.1, Thm. 1]{Evans}
			\begin{align}
				\frac{\vol}{8^{\N}} \left( \diam B \right)^{\N} 
				 = \vol \left( \textstyle \frac{h(x)}{40} \right)^{\N}
				 = \cH^{\N} \left( \textstyle \pi\left( B\left(x_{0},\frac{h(x)}{40} \right) \right) \right)
				& \le \cH^{\N} (B \cap \tilde A(U_{12})). \label{1.2.10;1}
			\end{align}
			The balls in $\mathcal{B}_{m}$ are disjoint, so we conclude using \cite[2.4.1, Thm. 1]{Evans}
			for the last estimate
			\[ \sum_{B \in \mathcal{B}_{m}} (\diam B)^{\N} \stackrel{\eqref{1.2.10;1}}{\le} 
				\frac{8^{\N}}{\vol} \sum_{B \in \mathcal{B}_{m}} \mathcal{H}^{\N}(B \cap \tilde{A}(U_{12})) 
				\le \frac{8^{\N}}{\vol} \mathcal{H}^{\N}(\tilde{A}(U_{12})) \le 192^{\N}.\]
		\end{proof}
	Now we have
	$\mu(F_{1} \cap \tilde F) \stackrel{\text{I}}{\le} \sum_{m=1}^{N_{0}} \sum_{B \in \mathcal{B}_{m}} \mu(B) 
		 \stackrel{\text{II, III}}{\le} \delta N_{0} \cdot 576^{\N}.$
	Since $\delta \le \frac{10^{-10}}{600^{\N} N_{0}}$ (see \eqref{Wahlvondelta} on page \pageref{Wahlvondelta}) 
	and $\varepsilon^{\frac{1}{2}} < \frac{10^{-7}}{C_{5}}$, we deduce together with Theorem \ref{thm3.18} that
	$\mu(F_{1}) < 10^{-6}$.
\end{proof}

\subsection{\texorpdfstring{$F_2$ is small}{F2 is small}}\label{F2issmall}
We recall that $0< \eta \le 2^{-(\N+1)}$ and $k \ge 1$ are fixed constants (cf. the first lines of 
section \ref{04.02.2014.1}) and that $F_{2}$ depends on the choice of $\alpha, \varepsilon >0$
(cf.~Definition \ref{def3.2}).
\begin{thm}
	Let $\alpha, \varepsilon > 0$. There exists some constant $C=C(\n,\N,\K,\p,C_{0},k)$ so that,
	if $\eta \le \frac{\varepsilon^{\p}}{C}10^{-6}$, we have $\mu(F_{2}) \le 10^{-6}$.
\end{thm}
\begin{proof}
	Let $x \in F_{2}$ and $ t \in (h(x),2h(x))$. 
	It follows that $x \notin F_{1} \cup \mathcal{Z}$ and hence,
	for all $ y \in F$ and for all $ \tau \in \left[ \frac{h(x)}{5},\frac{h(x)}{2} \right]$ 
	with $d(x,y) \le \frac{\tau}{2}$, we obtain
	$\delta(B(y,\tau)) > \delta$. So, in particular, we get
	$\delta\big(B\big( x,\frac{h(x)}{2}\big) \big) > \delta$ for $x=y$ and $\tau = \frac{h(x)}{2}$.
	If $k_{0}=1$, this implies
	$\tilde{\delta}_{k_{0}}(B(x,t))  \ge \delta(B(x,t)) > \frac{\delta}{4^{\N}}$,
	where we used $\frac{h(x)}{2} < t < 2 h(x)$.
	Let $(y,\tau)$ as in the definition of $F_{2}$. 
	Then we have
	$ d(x,y) + \tau < 2\tau \le h(x) < t$
	and hence $B(y,\tau) \subset B(x,t)$. We conclude
	$\beta_{1;k}(x,t) \ge \left( \frac{\tau}{t} \right)^{\N+1} \beta_{1;k}(y,\tau)
		\ge \frac{\varepsilon}{10^{\N+1}}$.
	Now, with Corollary \ref{thm2.4} ($\lambda=\frac{\delta}{4^{\N}}$, $k_{0}=1$), there exists some 
	constant $C=C(\n,\N,\K,\p,C_{0},k)$ so that
	\begin{align*}
		\mathcal{M}_{\K^\p}(\mu) 
		& \ge \frac{1}{C} \int_{F_{2}} \int_{h(x)}^{2h(x)} 
			\beta_{1;k}(x,t)^{\p}
			\Eins_{\{ \tilde{\delta}_{k_{0}}(B(x,t))\ge \frac{\delta}{4^{\N}} \}}
			\frac{\mathrm{d}t}{t} \mathrm{d}\mu(x)  \\
		& \ge \frac{1}{C} \int_{F_{2}}\int_{h(x)}^{2h(x)} 
			\left( \frac{\varepsilon}{10^{\N+1}} \right)^{\p} \frac{\mathrm{d}t}{t} \mathrm{d}\mu(x)  \\
		& \ge \frac{1}{C} \left( \frac{\varepsilon}{10^{\N+1}} \right)^{\p} \mu(F_{2}) \ln(2).
	\end{align*}
	Finally, using the previous inequality, condition (C) from page \pageref{Grundeigenschaften} and
	$\eta \le \frac{\ln(2)}{10^{\p(\N+1)}C}\varepsilon^{\p}10^{-6}$, we get the assertion.
\end{proof}
\subsection{\texorpdfstring{$F_3$ is small}{F3 is small}}\label{F3issmall}
We mention for review that $\tilde F$ is defined on page \pageref{DefinitionvonFtilde} and set
\[ \tilde{\tilde{F}} :=  \left\{x \in \tilde{F} \ \Big| \  
	\mu(\tilde{F} \cap B(x,t)) \ge \frac{99}{100}  \mu(F \cap B(x,t)) \text{ for all } t \in (0,2) \right\}.\]

\begin{lem}\label{lem5.8;1}
	Let $0<\alpha \le \frac{1}{4}$. 
	There exists some 
	$\hat \varepsilon = \hat \varepsilon (\n,\N,C_{0})\le \frac{1}{4}$ and some $\tilde k \ge 4$ so that, if
	$\eta < 2 \hat \varepsilon$ and $k \ge \tilde k$,
	there exists some constant $C=C(\n,\N,\K,\p,C_{0})$ so that,
	for all $\varepsilon \in [\frac{\eta}{2},\hat \varepsilon)$, we have 
	$\mu(F \setminus \tilde{\tilde{F}}) \le C \varepsilon^{\frac{1}{2}}$.
\end{lem}
\begin{proof}
	Let $0<\alpha \le \frac{1}{4}$ and choose $\hat \varepsilon$, $\tilde k$ to be the constants given by 
	Theorem \ref{thm3.18} and let $k \ge \tilde k$, $\eta \le 2 \varepsilon < 2 \hat \varepsilon$.
	Due to Theorem \ref{thm3.18}, we only have to consider $\mu(\tilde{F} \setminus  \tilde{\tilde{F}})$. 
	For all $ x \in \tilde{F} \setminus \tilde{\tilde{F}}$ using the definition of $\tilde F$, 
	there exists some $t_{x} \in (0,2)$ with
	$\mu(\tilde F \cap B(x,t_{x})) \le 99 \mu((F \setminus \tilde F)\cap B(x,t_{x}))$.
	Hence $\tilde{F} \setminus \tilde{\tilde{F}}$ is covered by balls $B(x,t_{x})$ with centre in 
	$\tilde{F} \setminus \tilde{\tilde{F}}$.
	So with Besicovitch's covering theorem \cite[1.5.2, Thm. 2]{Evans} 
	there exist $N_{0}=N_{0}(\n)$ families $\mathcal{B}_{m}$,
	$m = 1,.. , N_{0}$, of disjoint balls $B(x,t_{x})$ so that
	\[\mu(\tilde{F} \setminus \tilde{\tilde{F}}) 
		\le \sum_{m=1}^{N_{0}} \sum_{B \in \mathcal{B}_{m}} \mu(  \tilde{F} \cap B)
		 \le
			 99 \sum_{m=1}^{N_{0}} \sum_{B \in \mathcal{B}_{m}} \mu(  (F \setminus \tilde{F}) \cap B)
		 \le 99 N_{0} \ \mu(F \setminus \tilde{F}),\]
	and with Theorem \ref{thm3.18} the assertion holds.
\end{proof}

\begin{lem} \label{lem5.6}
	Let $\theta,\alpha > 0$. There exist some constant $C=C(\n,\N,C_{0},\theta) > 1$ and some constant
	$\varepsilon_{0}=\varepsilon_{0}(\n,\N,C_{0},\theta) > 0$ so that, if
	$\eta < 2 \varepsilon_{0}$ and $k \ge 4$, we have for all 
	$\varepsilon \in [\frac{\eta}{2},\varepsilon_{0})$ 
	that the following is true. If $(x,t) \in S$ and $100 t \ge \theta$, then we have
	$\varangle(P_{(x,t)},P_{0}) \le C \varepsilon$.
\end{lem}
\begin{proof}
	Let $\theta, \alpha > 0$, $k\ge 4$ and $\eta < 2 \varepsilon < 2\varepsilon_{0}$ where the constant 
	$\varepsilon_{0}$ is given by Lemma \ref{lem2.6}.
	Let $t \ge \frac{\theta}{100}$ and $(x,t) \in S$.
	We get with (A) and (D) (see page \pageref{Grundeigenschaften})
	$\beta_{1;k}^{P_{0}}(x,t) \le \left( \frac{500}{\theta} \right)^{\N+1} 2\varepsilon$.
	Furthermore, we have with Definition \ref{12.07.13.1} that $ \beta_{1;k}^{P_{(x,t)}}(x,t) \le 2 \varepsilon$
	and with $(x,t) \in S \subset S_{total}$ we obtain $ \delta(B(x,t)) \ge \frac{\delta}{2}$.
	Now, with Lemma \ref{lem2.6} 
	($y=x$, $c=1$, $\xi=2\left( \frac{500}{\theta} \right)^{\N+1}$, $t_{x}=t_{y}=t$, $\lambda=\frac{\delta}{2}$),
	there exists some
	constant $C_{3}=C_{3}(\n,\N,C_{0},\theta)$ so that 
	$\varangle(P_{(x,t)},P_{0})  \le C_{3}\varepsilon$. 
\end{proof}

\begin{lem}\label{25.09.2014.2}
	Let $\theta, \alpha >0$. If $k \ge 400$,  there exists some constant 
	$ \varepsilon^{*} = \varepsilon^{*}(\n,\N,C_{0},\alpha,\theta)$ so that, if $\eta < 2 \varepsilon^{*}$,
	we have for all $\varepsilon \in [\frac{\eta}{2},\varepsilon^{*})$ that 
	for all $x \in F_{3}$ we have $h(x) < \frac{\theta}{100}$.
\end{lem}
\begin{proof}
	Let $\theta, \alpha > 0$ and $k \ge 400$.
	We set $\varepsilon^{*}:= \min \{\bar \varepsilon, \varepsilon_{0},\frac{\alpha}{2C}\}$ where 
	$\bar \varepsilon$ is given by Lemma \ref{rem3.3} and $\varepsilon_{0}$ as well as 
	$C$ are given by Lemma \ref{lem5.6}.
	Let $\eta \le 2 \varepsilon < 2 \varepsilon^{*}$ and $x \in F_{3}$.
	With Lemma \ref{rem3.1} (i), we have $(x,h(x)) \in S$ and,
	with Lemma \ref{rem3.3}, we get $\varangle(P_{(x,h(x))},P_{0}) > \frac{1}{2} \alpha$.
	Hence we obtain $h(x) < \frac{\theta}{100}$ with Lemma \ref{lem5.6}.
\end{proof}

\begin{lem}\label{25.09.2014.1}
	Let $p=2$.
	There exists some $\hat k \ge 400$, some $\tilde \alpha =\tilde \alpha (\N) > 0$
	and some $\hat \theta = \hat \theta(\n,\N,C_{0}) \in (0,1)$
	so that for all $\alpha \in (0, \tilde \alpha]$ and $\theta \in (0,\hat \theta]$ there exists some
	$\hat \varepsilon = \hat \varepsilon(\n,\N,C_{0},\alpha,\theta)$ 
	so that, if $k \ge \hat k$ and $\eta < \hat \varepsilon^{2}$, we have for all
	$\varepsilon \in [\sqrt{\eta},\hat \varepsilon)$ that
	there exists some set $H_{\theta} \subset U_{6}$ and some constant
	$C=C(\n,\N,\K,C_{0},k)$ with
	$\cH^{\N}(U_{6} \setminus H_{\theta})<C \left(\frac{\varepsilon}{\theta^{\N+1}\alpha}\right)^{2}$ and, 
	for all $x \in F_{3} \cap \tilde{\tilde{F}}$, we have $d(\pi(x),H_{\theta}) > h(x)$.
\end{lem}
\begin{proof}
	Let $\tilde k$ and $\tilde \alpha(\N)$ be the thresholds given by Theorem \ref{thm4.1} and let
	$\hat C=\hat C(\n,\N)$ be the constant given by Theorem \ref{2.9.2014.1}.
	Moreover, let $C_{1}=C_{1}(\n,\N,C_{0})$ and $C_{2}=C_{2}(\n,\N,C_{0})$ be the constants given by 
	Corollary \ref{04.09.12.1} applied with $\lambda = \frac{\delta}{4}$, and $\delta=\delta(\n,\N)$ is the value fixed on
	page \pageref{Wahlvondelta}. We set 
	$\hat \theta :=\frac{1}{400}\left[ 18\N(10^{\N}+1)\left(\frac{C_{1}}{4}\right)^{\N+1} \hat C \right]^{-1}$
	and choose $\theta \in (0,\hat \theta]$.
	Let $\alpha \in (0, \tilde \alpha]$, and let 
	$\bar \varepsilon_{1} = \bar \varepsilon(\n,\N,C_{0},\alpha)$, 
	$\bar \varepsilon_{2} = \bar \varepsilon(\n,\N,C_{0},\alpha)$, 
	$\tilde \varepsilon=\tilde \varepsilon(\n,\N,C_{0},\alpha)$, 
	$\varepsilon_{0}=\varepsilon_{0}(\n,\N,C_{0},\theta)$,
	$\varepsilon^{*}=\varepsilon^{*}(\n,\N,C_{0},\alpha,\theta)$
	be the thresholds given by Lemma \ref{rem3.3}, \ref{AstetigaufU0}, Theorem \ref{thm4.1}, Lemma \ref{lem5.6} and
	Lemma \ref{25.09.2014.2} respectively.
	Finally, let $C$ be the constant from Lemma \ref{lem5.6}.
	We set
	$\hat \varepsilon:= \min\left\{ \bar \varepsilon_{1},\bar \varepsilon_{2},\tilde \varepsilon, 
		\varepsilon_{0}, \varepsilon^{*}, (\hat C \theta \alpha)^{2},
		\frac{\alpha}{400} \left[4\N(10^{\N}+1)\left(\frac{C_{1}}{4}\right)^{\N+1}2C_{2}\right]^{-1},
		\frac{\alpha}{100C}\right\}$
 	and assume that $k \ge \hat k:=\max\{\tilde k,400\}$ and $\eta \le {\hat \varepsilon}^{2}$. 
	Now let $\varepsilon > 0$ with
	$\eta \le \varepsilon^{2} < {\hat \varepsilon}^{2}$.

	Until now, we defined the map $A$ only on $U_{12}=B(0,12)\cap P_{0}$
	(see page \pageref{Aeindeutigdefiniert}).
	Furthermore, we have shown that $A$ is Lipschitz continuous with Lipschitz constant
	$3\alpha$ (see Lemma \ref{ALipschitz} on page \pageref{ALipschitz}).
	With Lemma \ref{8.11.12.1}, an application of Kirszbraun's Theorem,
	there exists an extension $\tilde A : P_{0} \to \R^{\n}$ of $A$ with compact support, the same Lipschitz 
	constant $3\alpha$ and $A = \tilde{A}$ on $U_{12}$.
	If one wants to omit Zorn's lemma, used for the proof of Lemma \ref{8.11.12.1}, 
	one can get the same result with a slightly larger Lipschitz constant (see the
	remark after Lemma \ref{8.11.12.1} for details).
	\label{FortsetzungvonA}
	We denote this extension of $A$ also by $A$.

	Using Theorem \ref{2.9.2014.1} with $g = A$, $\p=2$ and Theorem \ref{thm4.1}, there exist some set 
	$H_{\theta} \subset U_{6}$ and some constant $C=C(\n,\N,\K,C_{0},k)$ with
	$\cH^{\N}(U_{6} \setminus H_{\theta}) 
		\le \frac{C(\N)}{\theta^{2(\N+1)} \Lip_{A}^{2}} C \varepsilon^{2}$.
	Furthermore, we get for all $y \in P_{0}$ some affine map $a_{y}:P_{0} \to P_{0}^{\perp}$ so that,
	if $r \le \theta$ and $B(y,r) \cap H_{\theta} \neq \emptyset$, we have
	$\|A-a_{y}\|_{L^{\infty}(B(y,r)\cap P_{0},P_{0}^{\perp})} \le \hat C r \theta \Lip_{A}$.
	We recall that $\Lip_{A}=3 \alpha$ (cf. Lemma \ref{ALipschitz}).
	For $x \in F_{3} \cap \tilde{\tilde F} \subset F_{3} \cap \tilde F$, 
	we have with the previous lemma that $h(x) < \frac{\theta}{100}$.
	Let $t \in [h(x),\frac{\theta}{100}]$.
	If $x' \in B(x,2t) \cap \tilde{F}$, we obtain with Lemma \ref{rem3.8} and the definition of $\tilde F$
	$d(x',\pi(x')+ A(\pi(x')))  \le \varepsilon^{\frac{1}{2}} \left( d(x) + d(x,x') \right)
		 \le 3 \varepsilon^{\frac{1}{2}}t$.
	Let $P_{\pi(x)}$ denote the $\N$-dimensional plane, which is the graph of the affine map $a_{\pi(x)}$.
	Now we assume, contrary to the statement of this lemma, that $d(\pi(x),H_{\theta}) \le h(x)$. This implies
	$\pi(B(x,2t)) \cap H_{\theta} \neq \emptyset$, and so we have 
	$d(\pi(x')+ A(\pi(x')),P_{\pi(x)})
		 \le \|A-a_{\pi(x)}\|_{L^{\infty}(B(\pi(x),2t)\cap P_{0},P_{0}^{\perp})}
		 \le 6 \hat C \theta \alpha t$ for all $x' \in B(x,2t) \cap \tilde F$.
	We combine those estimates and obtain using $ 3\varepsilon^{\frac{1}{2}} \le 3\hat C \theta \alpha$
	\begin{align}
		d(x',P_{\pi(x)}) & \le d(x',\pi(x')+A(\pi(x'))) + d(\pi(x')+ A(\pi(x')),P_{\pi(x)}) 
		 \le 9 \hat C \theta \alpha t. \label{19.1.10;11}
	\end{align}
	Since $h(x) \le t$, we get $(x,t) \in S \subset S_{total}$ with Lemma \ref{rem3.1} (i) so that
	we have $\delta(B(x,t)) \ge \frac{\delta}{2}$. If $x \in \tilde{\tilde{F}}$, this estimate and
	the definition of $\tilde{\tilde{F}}$ implies $\delta(\tilde{F} \cap B(x,t)) > \frac{1}{4} \delta$.
	
	Now we apply Corollary \ref{04.09.12.1} ($\Upsilon=\tilde F$, $\lambda = \frac{\delta}{4}$),
	and so there exist constants $C_{1}(\n,\N,C_{0})$, $C_{2}(\n,\N,C_{0})$ 
	and an $(\N,10\N \frac{t}{C_1})$-simplex 
	$T=\Delta(x_0,\dots,x_{\N}) \in F \cap B(x,t) \cap \tilde F$  
	so that \mbox{$\mu(\tilde B_{i}) \ge  \frac{t^{\N}}{C_{2}}$} for all $i \in \{0,\dots, \N \}$
	where $\tilde B_{i}:=B\left(x_i,\frac{t}{C_1}\right) \cap B(x,t) \cap \tilde F$.
	With $(x,t) \in S \subset S_{total}$, we get for all $i\in \{0,\dots,\N\}$
	\begin{align*}
		\frac{1}{\mu(\tilde B_{i})} \int_{\tilde B_{i}} d(z,P_{(x,t)}) \dd \mu(z) 
		& \le C_{2} t \beta_{1;k}^{P_{(x,t)}}(x,t) 
		 \le 2C_{2} t  \varepsilon.
	\end{align*}
	This implies for all $i\in \{0,\dots,\N\}$ the existence of $y_{i} \in \tilde B_{i}$ with
	$d(y_{i},P_{(x,t)}) \le 2C_{2}t \varepsilon$.
	With Lemma \ref{17.11.11.2}, we deduce that $S:=\Delta(y_{0},\dots,y_{\N}) \subset B(x,t)$ is 
	an $(\N,8\N\frac{t}{C_{1}})$-simplex.
	Next, we apply Lemma \ref{21.11.11.2} ($m=\N$, $C=\frac{C_{1}}{8\N}$,$\hat C =1$, $\sigma=2C_{2} \varepsilon$) and get
	$\varangle(P_{(x,t)},P_{y_{0},\dots,y_{\N}}) \le\frac{\alpha}{400}$.
	We have $y_{i} \in \tilde B_{i} \subset B(x,2t) \cap \tilde F$ and hence 
	we get with \eqref{19.1.10;11} and Lemma \ref{21.11.11.2}
	($C=\frac{C_{1}}{8\N}$, $\hat C = 1$, $\sigma= 9 \hat C \theta \alpha$)
	$\varangle(P_{y_{0},\dots,y_{\N}},P_{\pi(x)}) \le \frac{\alpha}{400}$.
	We combine those two angel estimates and conclude 
	$ \varangle(P_{(x,t)},P_{\pi(x)}) \le \frac{\alpha}{200}$,
	which is true for all $x \in F_{3} \cap \tilde{\tilde F}$ with $d(\pi(x),H_{\theta}) \le h(x)$ 
	and all $t \in [h(x),\frac{\theta}{100}]$.
	Now we use this result for $t=h(x)$ and for $t= \frac{\theta}{100}$ and obtain
	$\varangle(P_{(x,h(x))},P_{(x,\frac{\theta}{100})}) \le \frac{\alpha}{100}$.
	Together with Lemma \ref{lem5.6} we get $\varangle(P_{(x,h(x))},P_{0}) \le \frac{\alpha}{50}$.
	This is in contradiction to Lemma \ref{rem3.3} hence our assumption that
	$d(\pi(x),H_{\theta}) \le h(x)$ is invalid for all $x \in F_{3} \cap \tilde{\tilde{F}}$.
\end{proof}

\begin{thm}
	Let $p =2$.
	There exists some constants $\bar {\bar k} \ge 4$, 
	$0 < \bar {\bar \alpha} = \bar {\bar \alpha}(\N) < \frac{1}{6}$ and 
	$0 < \bar { \bar \theta} = \bar {\bar \theta}(\n,\N,C_{0})$ so that, for all 
	$\alpha \in (0,\bar {\bar \alpha}]$ and all $\theta \in (0,\bar {\bar \theta}]$,
	there exists some
	$0<\bar {\bar \varepsilon}=\bar {\bar \varepsilon}(\n,\N,C_{0},\alpha,\theta)< \frac{1}{8}$ so that,
	if $ k \ge \bar {\bar k}$ and $\eta < \bar {\bar \varepsilon}^{2}$, we obtain for all 
	$\varepsilon \in [\sqrt{\eta},\bar{\bar \varepsilon})$ 
	\[ \mu(F_{3}) \le 10^{-6}.\]
\end{thm}
\begin{proof}
	Let $\bar {\bar k}$ be the maximum and $\bar {\bar \alpha}<\frac{1}{6}$ be the minimum of all thresholds for
	$k$ and $\alpha$ given by Lemma \ref{ALipschitz}, \ref{lem5.8;1}, \ref{25.09.2014.2} and \ref{25.09.2014.1}. 
	Furthermore, we set $\bar {\bar \theta}:=\hat \theta$, where $\hat \theta=\hat \theta(\n,\N,C_{0})$ 
	is given by Lemma \ref{25.09.2014.1}. 
	Let $0 < \alpha \le \bar {\bar \alpha}$ and $0 < \theta \le \bar {\bar \theta}$. We define 
	$\bar {\bar \varepsilon}=\bar {\bar \varepsilon}(\n,\N,C_{0},\alpha,\theta)$ 
	as the minimum of $\frac{1}{16}$, a small constant depending on $\n,\N,\K,C_{0},\alpha, \theta$ given by 
	the last lines of this proof, and of all upper bounds for $\varepsilon$ stated in 
	Lemma \ref{ALipschitz}, \ref{lem5.8;1}, \ref{25.09.2014.2} and \ref{25.09.2014.1}.
	Let $k \ge \bar {\bar k}$ and $\eta \le \varepsilon^{2} < \bar { \bar \varepsilon}^{2}$.
	We have
	$ \mu(F_{3}) \le \mu(F_{3} \cap \tilde{\tilde{F}}) + \mu(F_{3}\setminus \tilde{\tilde{F}})$.
	With Lemma \ref{lem5.8;1} ($p=2$), there exists some constant $C=C(\n,\N,\K,C_{0})$ so that
	$\mu(F_{3}\setminus \tilde{\tilde{F}}) \le \mu(F\setminus \tilde{\tilde{F}}) 
		\le C \varepsilon^{\frac{1}{2}}$.
	Hence we only have to consider $\mu(F_{3} \cap \tilde{\tilde{F}})$.
	We set
	$\mathcal{G} := \left\{B(x,2h(x)) | x \in F_{3} \cap \tilde{\tilde{F}}) \right\}$.
	This is a set of nondegenerate balls because $x \in F_{3} \subset F \setminus \mathcal{Z}$. 
	Furthermore, we have $h(x) \le 50$ for all $x \in F$ (see Definition of $h$ on page \pageref{Definitionvonh}).
	With Besicovitch's covering theorem \cite[1.5.2, Thm. 2]{Evans} there exist
	$N_{0}$ families $\mathcal{B}_{l} \subset \mathcal{G}$, $l=1,...,N_{0}$, of disjoint balls such that
	we conclude with property (B) from page \pageref{Grundeigenschaften}
	\begin{align*}
		\mu(F_{3} \cap \tilde{\tilde{F}}) 
		& \le \sum_{l=1}^{N_{0}} \sum_{B \in \mathcal{B}_{l}} \mu(B \cap \tilde{\tilde{F}})
		 \stackrel{\text{(B)}}{\le} C_{0} \sum_{l=1}^{N_{0}} \sum_{B \in \mathcal{B}_{l}} (\diam B)^{\N}.
	\end{align*}
	Let $ 1 \le l \le N_{0}$ and let
	$B_{1} = B(x_{1},2h(x_{1})), B_{2} = B(x_{2},2h(x_{2})) \in \mathcal{B}_{l}$ with $B_{1}\neq B_{2}$.
	Since the balls in $\mathcal{B}_{l}$ are disjoint, we deduce 
	$2h(x_{1})+2h(x_{2}) \le d(x_{1},x_{2})$ and, because of the definition of $\tilde F$ and
	Lemma \ref{rem3.8}, we get 
	$d(x_{i},\pi(x_{i})+A(\pi(x_{i}))) \le \varepsilon^{\frac{1}{2}}d(x_{i}) \le \varepsilon^{\frac{1}{2}}h(x_{i})$
	for $i=1,2$.
	Since $ \varepsilon^{\frac{1}{2}} < \frac{1}{4}$, 
	$\alpha < \frac{1}{6}$ and  $A$ is $3\alpha$ Lipschitz continuous,
	the former two estimates imply $h(x_{1}) + h(x_{2}) < d(\pi(x_{1}),\pi(x_{2}))$.
	Thus $\pi(\frac{1}{2}B_{1})$ and $\pi(\frac{1}{2}B_{2})$ are disjoint. 
	We have $x_{i} \in \left( \tilde{\tilde{F}} \cap F_{3}\right) \subset F \subset B(0,5)$ 
	for $i=1,2$. With Lemma \ref{25.09.2014.2}, we conclude 
	that $h(x_{i}) \le \frac{\theta}{100} < \frac{1}{2}$.
	This implies $\pi(\frac{1}{2}B_{i}) \subset U_{6}$.
	Using Lemma \ref{25.09.2014.1}, there exists some set $H_{\theta} \subset U_{6}$
	and some constant $C=C(\n,\N,\K,C_{0},k)$ with
	$\cH^{\N}(U_{6} \setminus H_{\theta})<C \left(\frac{\varepsilon}{\theta^{\N+1}\alpha}\right)^{2}$ so that
	$d(\pi(x),H_{\theta}) > h(x)$ for all $x \in F_{3} \cap \tilde{\tilde{F}}$,
	in particular for $x=x_{i}$.
	We conclude that
	$ \pi(\frac{1}{2}B_{i}) \cap H_{\theta} = \emptyset$, and hence
		\[\sum_{B \in \mathcal{B}_{l}} (\diam B)^{\N} 
		= 4^{\N}\sum_{B \in \mathcal{B}_{l}} \left({\textstyle\frac{1}{2}}\diam \pi\left({\textstyle\frac{1}{2}}B\right)\right)^{\N}
		= 4^{\N}\sum_{B \in \mathcal{B}_{l}} \frac{1}{\vol} \cH^{\N}\left(\pi\left({\textstyle\frac{1}{2}}B\right)\right)
		\le \frac{4^{\N}}{\vol} \cH^{\N}(U_{6} \setminus H_{\theta}).\]
	Now we obtain
	\[ \mu(F_{3} \cap \tilde{\tilde{F}}) \le C_{0}N_{0} \frac{4^{\N}}{\vol} \cH^{\N}(U_{6} \setminus H_{\theta})
		\le C \left(\frac{\varepsilon}{\theta^{\N+1}\alpha}\right)^{2}.\]
	and we have already shown that $\mu(F_{3} \setminus \tilde{\tilde{F}}) \le C \varepsilon^{\frac{1}{2}}$.
	Using $\varepsilon < \bar{\bar \varepsilon}$, we finally get
	\mbox{$\mu(F_{3}) < 10^{-6}$}.
\end{proof}

\renewcommand{\theequation}{\Alph{section}.\arabic{equation}}
\appendix
\section{Measuretheoretical statements}
\begin{lem} \label{22.2.2012.1}
	Let $\mathcal{E}$ be a set of disjoint balls (open or closed) in $\R^{\n}$ with radius
	equal or larger then $s \in (0,\infty)$ and $B \subset B(x,r)$ 
	for all $B \in \mathcal{E}$.
	Then $\mathcal{E}$ is a finite set with $\# \mathcal{E} \le \left(\frac{r}{s}\right)^\n$.
\end{lem}
\begin{proof}
	Choose $l$ different balls $B_{l} \in \mathcal{E}$
	and let $\vo{\n}$ be the volume of the $\n$-dimensional unit sphere.
	We have
	$l s^{\n}\vo{\n} \le \sum_{i=1}^{l}\mathcal{L}^{\n}\left(B_{i}\right) 
		\le \mathcal{L}^{\n}(B(x,r)) 
		 = \vo{\n}(r)^{\n}$.
	This implies $l \le \left(\frac{r}{s}\right)^{\n}$ and hence $\# \mathcal{E} \le \left(\frac{r}{s}\right)^{\n}$.
\end{proof}

\begin{lem} \label{23.2.2012.1}
	Let $ s>0$ and $B(x,r)$ be an open or closed ball in $\R^m$ with $s < r$. There exists some 
	family  $\mathcal{E} $ of disjoint closed balls with
	$\diam B = 2s$ for all $B \in \mathcal{E}$,
	$ B(x,r) \subset \bigcup_{B \in \mathcal{E}} 5B$
	and
	$\# \mathcal{E}\le \left(\frac{2r}{s}\right)^m$.
\end{lem}
\begin{proof}
	Set $ \mathcal{F} = \left\{ B(y,s) | y \in B(x,r) \right\}$. With Vitali's covering theorem \cite[1.5.1, Thm 1]{Evans}
	there exists a countable family $\mathcal{E}$ of disjoint balls in $\mathcal{F}$ such that
	$B(x,r) \subset \bigcup_{B \in \mathcal{E}}5B$.
	Due to $s<r$, we get $ B \subset B(x,2r) $ for all $B \in \mathcal E$ and hence
	Lemma \ref{22.2.2012.1} implies $\# \mathcal{E} \le \left(\frac{2r}{s}\right)^m$.
\end{proof}

\begin{lem}\label{30.08.11.1}
	Let $A \subset \R^{\n}$ be a closed set with $\nu(A)>0$, where $\nu$ is some outer measure on $\R^n$. 
	There exists some $x \in A$ so that
	$\nu(B(x,h))>0$ for all $h > 0$.
\end{lem}
\begin{proof}
	For every $h > 0$, there exists some $y \in A$ with $\nu(B(y,h))>0$ because otherwise we would obtain 
	$\nu(A)=0$.
	Now, we find a sequence $x_{i}\in A$ with $\lim_{i \to \infty} x_i=x$ and $\nu(B(x_i,\frac{1}{i}))>0$.
	Let $h>0$. With $i$ small enough, we obtain $\nu(B(x,h)) \ge \nu(B(x_i,\frac{1}{i}))>0$.
\end{proof}

\begin{lem}\label{24.10.12.1}
	Let $R$ be an $\N$-dimensional cube in $\R^{\n}$. Then
	$\left( \diam R\right)^{\N} = (\sqrt{\N})^{\N} \cH^{\N}(R)$.
\end{lem}
\begin{proof}
	Let $\cH^{\N}(R) =a^{\N}$. Then $\diam R = \sqrt{\N}a$ implies the assertion.
\end{proof}

\begin{lem}\label{8.11.12.1}
	Let $K \subset \R^{m}$ be a bounded set and $f:K \to \R^{\n}$ be a Lipschitz function.
	Then $f$ has a Lipschitz extension $g : \R^{m} \to \R^{\n}$ with 
	compact support and the same Lipschitz constant.
\end{lem}
Instead of Kirszbraun's Theorem \cite[Thm 2.10.43]{Federer}, we can use
some simpler theorem for the proof \cite[3.1.1, Thm 1]{Evans} and get the same result but with
the larger Lipschitz constant $\text{Lip}_{g}=\sqrt{\n}\text{Lip}_f$.

\begin{proof}
	Let $\text{Lip}_f$ be the Lipschitz constant of $f$ and let $B(z,t)$ be some ball with $K \subset B(z,t)$.
	We define $T  := t + \frac{1}{\text{Lip}_f}\max_{x \in K}|f(x)|$ and set
	$\bar f:=f$ on $K$ and $\bar f =0$ on $\R^{m}\setminus B(z,T)$.
	Now it is easy to see that $\bar f: (\R^{m} \setminus B(z,T)) \cup K \to \R^{\n}$ is Lipschitz continuous 
	with the same Lipschitz constant as $f$.
	By applying Kirszbraun's Theorem \cite[Thm 2.10.43]{Federer} 
	on $\bar f$, we obtain a Lipschitz extension $g : \R^{m} \to \R^{\n}$ with 
	compact support and the Lipschitz constant $\text{Lip}_{f}$.
\end{proof}

\section{Differentiation and Fourier transform on a linear subspace} \label{diff_on_lin_subspace}
Let $P_{0} \in G(\n,\N)$ be an $\N$-dimensional linear subspace of $\R^{\n}$ and 
$f: P_{0} \to R$ be some function, where $R \in \{\R,\R^{\n}\}$. 
In this section, we explain what we mean by differentiating this function and formulating Taylor's theorem
in this setting. Furthermore, we define the Fourier transform of $f$ and give some basic properties.

Let $\phi : \R^{\N} \to P_{0}$ be a fixed isometric isomorphism.
We set $\tilde f : \R^{\N} \to R$, $\tilde f (x)=f(\phi(x))=(f \circ \phi)(x)$.

\begin{dfn}\label{9.10.2014.1}
	Let $l \in \mathbb{N} \cup \{0\}$.
	We say $f \in C^{l}(P_{0},R)$ iff $\tilde f \in C^{l}(\R^{\N},R)$ ($l$-times continuously differentiable).
	If $l \ge 1$ for all $i \in \{1,\dots,\N \}$, we set
	$\partial_{i} f:= D_{i} \tilde f \circ \phi^{-1} = D_{i}(f \circ \phi) \circ \phi^{-1}$,
	$\Delta f := \sum_{j=1}^{\N} \partial_{j} \partial_{j} f$,
	$D f := (\partial_{1}f, \dots,\partial_{\N}f)$,
	and, if $\kappa=(\kappa_1,\kappa_2, \dots,\kappa_{\N})$ is a multi-index, we set
	$\partial^{\kappa} f := \partial_{1}^{\kappa_{1}} \partial_{2}^{\kappa_{2}} \dots \partial_{\N}^{\kappa_{\N}} f$.
	Furthermore for $x,y,z \in \R^{\N}$ and some multi-index $\kappa$, we use the following notations
	$x = (x_{1}, \dots, x_{\N})$,
	$x^{\kappa} = x_{1}^{\kappa_{1}} \cdot x_{2}^{\kappa_{2}} \cdot \dots \cdot x_{\N}^{\kappa_{\N}}$,
	$\kappa ! = \kappa_{1}!\kappa_{2}! \cdot \dots \cdot \kappa_{\N}!$,
	$|\kappa| = \kappa_{1} + \dots + \kappa_{\N}$ and
	$[y,z]: = \{ y + t(z-y) | t \in [0,1] \}$. 
\end{dfn}
The following Lemmas transfer classical results to our setting and are stated without proof.
\begin{lem}\label{8.1.14.1}
	Let $\kappa=(\kappa_1,\kappa_2, \dots,\kappa_{\N})$ be some multi-index with $|\kappa|=l \ge 1$ and 
	$f \in C^l(P_{0},\R^{\n})$.
	We have
	$\partial^{\kappa} f = D^{\kappa} \tilde f \circ \phi^{-1}= [D^{\kappa}(f \circ \phi)] \circ \phi^{-1}$, 
	where\\ $D^{\kappa}\tilde f=(D_{1})^{\kappa_{1}}(D_{2})^{\kappa_{2}} \dots (D_{\N})^{\kappa_{\N}} \tilde f$.
\end{lem}

\begin{lem}[Taylor's theorem] \label{Taylorformel}
	Let $f \in C^{s+1}(P_0,\R^{\n})$ and $[y_0,y] \subset P_0$. We have
	$f(y)= p_s(y) + R_s(y-y_0)$,
	where 
	$p_s(y):= \sum_{|\kappa| \le s} \frac{1}{\kappa !} \partial^{\kappa} f(y_0) (\phi^{-1}(y-y_0))^{\kappa}$
	and
	\[ R_s(y-y_0) := \int_0^1 (s+1) (1-t)^s \Big( \sum_{|\kappa| = s+1} \frac{1}{\kappa !} \partial^{\kappa}
		f(y_0 +t(y-y_0))(\phi^{-1}(y-y_0))^{\kappa}\Big) \dd t. \]
\end{lem}

\begin{lem}[Partial integration] \label{partielleIntegration}
	Let $l \in \mathbb{N}$, $f \in C^{l}(P_{0},\R^{\n})$, $\varphi \in C_{0}^{\infty}(P_{0},\mathbb{R})$. 
	Then for all multi-indices $\kappa$ with $|\kappa|=l$ we have
	$\int_{P_{0}} f(y) \partial^{\kappa} \varphi(y) \dd \cH^{\N}(y) 
		= (-1)^{|\kappa|} \int_{P_{0}} \partial^{\kappa} f (y) \varphi(y) \dd \cH^{\N}(y)$.
\end{lem}

Now we define the Fourier transform for some function $f \in \mathscr{S}(P_{0})$, where $\mathscr{S}(P_{0})$
is the Schwartz space of rapidly decreasing functions $f: P_{0} \to \mathbb{C}$,
cf. \cite[2.2.1 The Class of Schwartz Functions]{FourierAnalysis}. We will get the 
same results as for some function $f \in \mathscr{S}(\R^{\N})$.

\begin{dfn}[Fourier transform]\label{17.01.2013.1}
	Let $y \in P_0$ and $f \in \mathscr{S}(P_{0})$. We set
	\begin{align*}
		\widehat f (y) & := \widehat{(f \circ \phi)}(\phi^{-1}(y)) \index{$\widehat f$}
		 = \int_{\R^{\N}}  f(\phi(z)) e^{-2\pi i \phi^{-1}(y) \cdot z} \dd \mathcal{L}^{\N} (z).
	\end{align*}
	If $f : P_{0} \to \mathbb{C}^{\n}$ with $f_{i}  \in \mathscr{S}(P_{0})$, i.e., every component
	of $f$ is a Schwartz function, then we write $f \in \mathscr{S}(P_{0},\mathbb{C}^{\n})$. 
	We define the Fourier transform of some function $f \in \mathscr{S}(P_{0},\mathbb{C}^{\n})$
	by $\widehat f := (\widehat f_{1}, \dots, \widehat f_{\n})$.
\end{dfn}

\begin{lem}[Fourier transform and convolution] \label{FourFal}
	Let $f,g \in \mathscr{S}(P_{0})$ and let the convolution of $f$ and $g$ be defined by
	$(g * f)(w)= \int_{P_{0}} g(w-v)f(v) \dd \cH^{\N}(v)$ \index{$g * f$}. 
	Then for $w \in P_{0}$ we have
	$\widehat{(g*f)}(w) = \widehat{g}(w)\widehat{f}(w)$.
\end{lem}

\begin{lem}\label{10.12.12.2}
	Let $f \in \mathscr{S}(P_{0})$, $y \in P_{0}$, $t \in \R$ and set $f_t(y):=\frac{1}{t^{\N}}f(\frac{y}{t})$. We have
	$\widehat{\left(\partial^{\kappa} f \right)} (y) = (2\pi i \phi^{-1}(y))^{\kappa} \widehat{f} (y)$ and
	$\widehat{(f_{t})}(y) = \widehat{f}(t y)$.
\end{lem}

\begin{lem}\label{17.01.2014.1} \label{22.02.2013.02}
	Let $f \in \mathscr{S}(P_{0})$ be radial. Then $\widehat f$ and $\Delta f$ are radial as well.
\end{lem}

\setcounter{equation}{0}
\section{Littlewood Paley theorem}
\begin{lem}[Continuous version of the Littlewood Paley theorem]\label{4.3.2013.1}
	Let $\phi$ be an integrable $C^{1}(\R^{\N};\R)$ function with mean value zero fulfilling 
	$|\phi(x)| + |\nabla\phi(x)| \le C(1+|x|)^{-\N-1}$
	and 
	$0<\int_{0}^{\infty}|\widehat{(\phi_{t})}(x)|^{2}   \frac{\dd t}{t} < \infty$,
	where $\phi_{t}(x)=\frac{1}{t^{\N}}\phi(\frac{x}{t})$.
	For all $q \in (1,\infty)$, there exists some constant $C=C(\N,q,\phi)$ such 
	that, for all $f \in L^{q}(\R^{\N};\R^{\n})$, we have
	\[\left\|\left(\int_{0}^{\infty} |\phi_{t} * f|^2 \frac{\dd t}{t} \right)^{\frac{1}{2}}\right\|_{L^{q}(\R^{\N};\R)}
		\le C\|f \|_{L^{q}(\R^{\N};\R^{\n})}.\]
\end{lem}
\begin{proof}
	The proof is completely analogue to the proof of the Littlewood-Paley theorem \cite[Thm, 5.1.2]{FourierAnalysis}.
\end{proof}

\bibliography{literature_bibtex}
\bibliographystyle{amsplain}

\end{document}